\documentclass{amsart}
\usepackage[margin = 3 cm,marginpar= 2cm]{geometry}
\usepackage{mathtools,xcolor,enumitem,subcaption,amssymb,mathrsfs}

\usepackage[foot]{amsaddr}

\numberwithin{equation}{section} 

\let\mathbbalt\mathbb

\usepackage{mleftright} 

\newcommand*{\K}{\mathbbalt{K}}
\newcommand*{\N}{\mathbbalt{N}}
\newcommand*{\R}{\mathbbalt{R}}

\newcommand*{\Acal}{\mathcal{A}}
\newcommand*{\Fcal}{\mathcal{F}}

\newcommand*{\Jcal}{\mathcal{J}}

\newcommand*{\Tcal}{\mathcal{T}}
\newcommand*{\Ucal}{\mathcal{U}}
\newcommand*{\Vcal}{\mathcal{V}}
\newcommand*{\Wcal}{\mathcal{W}}
\newcommand*{\Xcal}{\mathcal{X}}
\newcommand*{\Ycal}{\mathcal{Y}}

\newcommand*{\Dscr}{\mathscr{D}}
\newcommand*{\Hscr}{\mathscr{H}}

\newcommand*{\Kscr}{\mathscr{K}}
\newcommand*{\Lscr}{\mathscr{L}}
\newcommand*{\Mscr}{\mathscr{M}}
\newcommand*{\Nscr}{\mathscr{N}}
\newcommand*{\Pscr}{\mathscr{P}}
\newcommand*{\Rscr}{\mathscr{R}}
\newcommand*{\Tscr}{\mathscr{T}}

\newcommand*{\Vscr}{\mathscr{V}}
\newcommand*{\Wscr}{\mathscr{W}}

\newcommand*{\Lp}[1]{L^{#1}}

\DeclareMathOperator{\Proj}{P}

\newcommand{\dee}{\mathop{}\!\mathrm{d}}
\newcommand*{\e}{\mathrm{e}}

\DeclareMathOperator*{\sgn}{sgn}

\DeclareMathOperator*{\TV}{T.\!V.}
\newcommand*{\Id}{\mathrm{Id}}
\newcommand*{\del}{\partial}

\newcommand*{\epsv}{\varepsilon}
\newcommand*{\phiv}{\varphi}
\newcommand*{\ds}{\displaystyle}

\newtheorem{thm}{Theorem}[section]
\newtheorem{cor}[thm]{Corollary}
\newtheorem{lem}[thm]{Lemma}
\newtheorem{prp}[thm]{Proposition}
\newtheorem*{mainres}{Main result}

\theoremstyle{remark}
\newtheorem{rem}[thm]{Remark}
\newtheorem{exm}[thm]{Example}

\theoremstyle{definition}
\newtheorem{dfn}[thm]{Definition}

\renewcommand{\left}{\mleft}
\renewcommand{\right}{\mright}

\newcommand{\ip}[2]{\langle#1,#2\rangle}

\DeclarePairedDelimiter\abs{\lvert}{\rvert}
\DeclarePairedDelimiter\norm{\lVert}{\rVert}%

\newcommand{\diff}[2]{\frac{\dee #1}{\dee #2}}
\newcommand{\pdiff}[2]{\frac{\partial #1}{\partial#2}}

\allowdisplaybreaks 
\usepackage{hyperref}
\newcommand{\arxiv}[1]{\href{https://arxiv.org/abs/#1}{arXiv:#1}} 
\newcommand{\ssrn}[1]{\href{https://dx.doi.org/10.2139/ssrn.#1}{ssrn:#1}} 

\title[Equivalence of entropy solutions and gradient flows for 1D Euler systems]{Equivalence of entropy solutions and gradient flows for pressureless 1D Euler systems} 

\author[J.\ A.\ Carrillo]{Jos\'{e} Antonio Carrillo\textsuperscript{1}}
\address{\textsuperscript{1}Mathematical Institute, University of Oxford, Oxford OX2 6GG, United Kingdom}
\email{carrillo@maths.ox.ac.uk}

\author[S.\ T.\ Galtung]{Sondre Tesdal Galtung\textsuperscript{2,3}}
\address{\textsuperscript{2}Department of Mathematical Sciences, NTNU -- Norwegian University of Science and Technology, 7491 Trondheim, Norway}
\address{\textsuperscript{3}Department of Mathematics and Cybernetics, SINTEF Digital, 0314 Oslo, Norway (\textit{current address})}
\email{sondre.galtung@sintef.no}


\begin{document}
  
  \begin{abstract}
    We study distributional solutions of  pressureless Euler systems on the line.
    In particular we show that Lagrangian solutions \cite{brenier2013sticky}, introduced by Brenier, Gangbo, Savar\'{e} and Westdickenberg, and entropy solutions \cite{nguyen2015one}, studied by Nguyen and Tudorascu for the Euler--Poisson system, are equivalent.
    For the Euler--Poisson system this can be seen as a generalization to second-order systems of the equivalence between \(L^2\)-gradient flows and entropy solutions for a first-order aggregation equation proved by Bonaschi, Carrillo, Di Francesco and Peletier \cite{bonaschi2015equivalence}. The key observation is an equivalence between Ole\u{\i}nik's E-condition for conservation laws and a characterization due to Natile and Savar\'{e} of the normal cone for \(L^2\)-gradient flows. This new equivalence allows us to define unique solutions after blow-up for classical solutions of the Euler--Poisson system with quadratic confinement due to Carrillo, Choi and Zatorska \cite{carrillo2016pressureless}, as well as to describe their asymptotic behavior.
  \end{abstract}
  
  \maketitle
  
  
  \section{Introduction}
  We will study distributional solutions of the Cauchy problem for the 1D pressureless Euler system
  \begin{subequations} \label{eq:EF}
    \begin{align}
      \del_t \rho + \del_x (\rho v) &= 0, \label{eq:EF:den} \\
      \del_t (\rho v) + \del_x (\rho v^2) &= f[\rho], \label{eq:EF:mom}
    \end{align}
  \end{subequations}
  with initial density \( \rho(0,x) = \rho_0(x) \) and velocity \( v(0,x) = v_0(x) \), where \(f[\rho]\) is a force field generated by the fluid itself.
  Furthermore, we assume initial data \(\rho_0 \in \Pscr_2(\R)\), the set of probability measures \(\Pscr(\R)\) with bounded second moments, and \(v_0 \in \Lp{2}(\R,\rho_0)\), such that the initial kinetic energy is finite.
  That is, we have
  \begin{equation}\label{eq:init} 
     \rho_0 \in \Pscr_2(\R) \iff \int_{\R} \abs{x}^2 \dee\rho_0 < \infty \quad \text{and} \quad v_0 \in \Lp{2}(\R,\rho_0) \iff \int_{\R} v_0^2 \dee\rho_0 < \infty.
  \end{equation}
  Moreover, we assume the map \(f[\rho] \colon \Pscr_2(\R) \to \Mscr(\R)\), where \(\Mscr(\R)\) are the signed Borel measures with finite total variation, is absolutely continuous with respect to \(\rho\), that is,
  \begin{equation}\label{eq:f:RN}
    f[\rho] = f_\rho \rho, \qquad f_\rho \in \Lp{2}(\R,\rho),
  \end{equation}
  where \(f_\rho\) is the Radon--Nikodym derivative of \(f[\rho]\) with respect to \(\rho\).
  The system \eqref{eq:EF} was studied under these assumptions in \cite{brenier2013sticky} with forcing and no forcing in \cite{brenier1998sticky,natile2009wasserstein}. 
  A special case of \eqref{eq:EF} which is of interest to us is the pressureless Euler--Poisson system, studied in, e.g., \cite{nguyen2008pressureless,nguyen2015one} in the form
  \begin{subequations} \label{eq:EP}
    \begin{align}
      \del_t \rho + \del_x (\rho v) &= 0, \label{eq:EP:den} \\
      \del_t (\rho v) + \del_x (\rho v^2) &= -\left(\alpha \Phi_x + \beta \right) \rho, \label{eq:EP:mom}
    \end{align}
  \end{subequations}
  where \( \alpha, \beta \in \R \), while \(\Phi_x(t,x) = \phi' \ast \rho(t,x) \) with \( \phi(x) = \frac12\abs{x} \) such that \( \del^2_x\Phi = \rho \) in the distributional sense.
  Note that in \cite{nguyen2008pressureless,nguyen2015one} one only specifies \(\del^2_x\Phi = \rho\), which determines \(\Phi_x\) up to a constant which can easily be absorbed in \(\beta\), and this explains why they work with a different \(\phi\) corresponding to the Heaviside function.
  We prefer to work with the symmetric \(\phi\) given above, which then can be thought of as a 1D Poisson potential.
  A positive or negative sign for \(\alpha\) then yields respectively an attractive or repulsive force.
  
  We are particularly interested in the two, seemingly distinct, notions of solution introduced in \cite{brenier2013sticky} and \cite{nguyen2008pressureless,nguyen2015one}.
  Each of which yields a unique solution to their respective systems under appropriate assumptions on the forcing terms, even faced with mass concentration.
  Our aim is to show that these notions in fact are equivalent. More precisely, we prove the equivalence of Lagrangian solutions \cite{brenier2013sticky} and entropy solutions \cite{nguyen2008pressureless,nguyen2015one} for the Euler system \eqref{eq:EF}, which contains the attractive and repulsive Euler--Poisson systems \eqref{eq:EP} as special cases. A very recent paper \cite{suder2023lagrangian} describes this equivalence between sticky particles flow described in \cite{hynd2019lagrangian} and entropy solutions of \eqref{eq:EF} for \(f[\rho] \equiv 0\) .
  
  An analogous equivalence has been shown in \cite{bonaschi2015equivalence} for the nonlocal interaction equation where a probability measure \(\mu \in \Pscr_2(\R)\) satisfies
  \begin{equation}\label{eq:interact}
    \del_t \mu + \del_x\left( \mu \; \del_x \phi \ast \mu \right) = 0, \qquad t > 0, \quad x \in \R
  \end{equation}
  for an interaction potential \(\Phi(x) = \alpha \abs{x}\) with \(\alpha \in \{-1,1\}\), and the Burgers-type conservation law
  \begin{equation}\label{eq:Burgers}
    \del_t M + \del_x \Acal(M) = 0, \qquad t > 0, \quad x \in \R
  \end{equation}
  with flux function \(\Acal(m) = \alpha (m-m^2)\) and \(M\) being the (cumulative) distribution function of \(\mu\), \(M(t,x) = \mu(t,(-\infty,x])\).
  Here \(\alpha = 1\) results in an attractive potential, while \(\alpha = -1\) gives a repulsive potential.
  The dynamics of \eqref{eq:interact} is relatively simple in the sense that an attractive potential promotes accumulation of mass \(\mu\), while a repulsive potential spreads the mass.
  This is particularly clear for a sum of Dirac masses \(\mu = \sum_{i=1}^n m_i \delta_{x_i}\), i.e., a set of \(n\) ``particles'', as an attractive potential will aggregate the particles into one particle of unit mass, while a repulsive potential will instantly diffuse the Dirac measures into a measure which is absolutely continuous with respect to the Lebesgue measure.
  The aggregation and diffusion turn out to have a one-to-one correspondence with respectively shock and rarefaction waves of the conservation law \eqref{eq:Burgers}, see \cite[Section 2]{bonaschi2015equivalence}.
  In fact, there the equivalence of \eqref{eq:interact} and \eqref{eq:Burgers} is first established at the particle level before passing to a limit to cover the general case.
  
  Note that \eqref{eq:interact} is a first-order system, unlike the second-order Euler--Poisson system \eqref{eq:EP}, and so it does not feature the issues of sub- and supercritical velocities in the repulsive case.
  Indeed, if particles in the repulsive \eqref{eq:interact} are initially separate, they will never meet, as the inter-particle repulsive forces will only drive them further apart.
  On the other hand, for particles in the repulsive \eqref{eq:EP}, for certain, so-called supercritical, initial velocity configurations, particles might collide despite the mutually repulsive forces acting on them.
  This also raises the question of what is the ``natural'' or physical way of extending the solutions beyond collision-time, or singularity formation.
  
  In this connection to particle dynamics, we also confirm the final remark of Nguyen and Tudorascu \cite{nguyen2015one} that their entropy solution of the Euler--Poisson system coincides with the corresponding Lagrangian solution of Brenier et.\ al \cite{brenier2013sticky} in the attractive Poisson case on the basis that they both can be realized as limits of the same sticky discrete particle systems; indeed, we find these notions of solution are equivalent directly from their definitions regardless of the attractive/repulsive character of the force. Furthermore, we deal with the following generalization of \eqref{eq:EP},
  \begin{equation}\label{eq:EPQ}
    \begin{aligned}
      \del_t \rho + \del_x (\rho v) &= 0, \\
      \del_t (\rho v) + \del_x (\rho v^2) &= -\gamma \rho v -\lambda^2 (\tilde{\phi}'\ast\rho)\rho.
    \end{aligned}
  \end{equation}
  In the above, \(\gamma \ge 0\) is a linear damping parameter, while \(\tilde{\phi} = \frac12 x^2 - \abs{x}\) represents a potential with Coulomb repulsion and quadratic confinement.
  This system has been studied in detail in \cite{carrillo2016pressureless} for classical solutions up until blow-up, or mass concentration, for which precise critical thresholds are given.
  We want to use the aforementioned solution concepts to study \eqref{eq:EPQ} beyond the time of blow-up and its asymptotic behavior.
  
  There is a rich literature on pressureless Euler systems, and so we will only present a few works relevant to us, mostly in the one-dimensional case.
  The system \eqref{eq:EF} with \(f[\rho] \equiv 0\), often named simply the pressureless Euler system, has been thoroughly studied, and has found applications for instance in modeling the formation of the early universe, in terms of ``sticky dust'', cf.\ \cite{zeldovich1970,shandarin1989large}.
  Indeed, in \cite{brenier1998sticky} the authors show that sticky, or adhesive, particle evolution can be formulated in terms of a scalar conservation law and give rise to weak solutions of this system satisfying an entropy condition;
  in particular, they construct weak solutions for compactly supported \(\rho_0 \in \Pscr(\R)\), denoted \(\Pscr_{\mathrm{c}}(\R)\), and continuous and bounded velocity \(v_0 \in C_{\mathrm{b}}(\R)\).
  This system has been studied from a slightly different point of view in \cite{e1996generalized}, which also covers the self-gravitating case \eqref{eq:EP} with \(\beta =0\) and \(\alpha =1\), where one employs a continuation of characteristics-type argument.
  In \cite{bouchut1999duality} the authors introduce the notion of duality solutions for the pressureless Euler system, which are also related to solutions of an associated scalar conservation law, and prove existence and uniqueness of these.
  Another existence and uniqueness result is found in \cite{huang2001well} for \(v_0\) bounded and measurable with respect to the nonnegative Radon measure \(\rho_0\),  where one relies on an entropy inequality, see \eqref{eq:onesided} below, and weak continuity of the initial energy.
  
  In \cite{natile2009wasserstein}, the authors study the pressureless system using adhesive dynamics and gradient flows with Wasserstein metrics while recovering the results of \cite{brenier1998sticky}. The authors in \cite{cavalletti2015simple} give a  more direct proof of construction of sticky solutions of \eqref{eq:EP} in the attractive case using projections.  
  Inspired by \cite{brenier1998sticky}, the authors of \cite{nguyen2008pressureless} use particle approximations and scalar conservation laws to introduce what they call entropy solutions of the Euler--Poisson system \eqref{eq:EP}.
  Later they refined their results in \cite{nguyen2015one} to allow for the same assumptions on initial data as in \cite{brenier2013sticky}, as well as treating additional viscous terms in the right-hand side.
  
  Several works for the Euler--Poisson system deal with classical solutions up until blow-up \cite{ELT01,TW08,tadmor2011variational,CCTT16,carrillo2016pressureless,BLT23}. Their strategy is based on explicit representations of the solutions along characteristics, leading to critical thresholds for the existence and uniqueness of classical solutions, see \cite{CCP17} for numerical illustrations. An existence and uniqueness result of entropy solutions related to ours using the framework of conservation laws \cite{lefloch2015existence} covers \eqref{eq:EP} in the attractive case, see Remark \ref{rem:lefloch} for more details. Furthermore, in the recent work \cite{huang2023global} uniqueness of entropy solutions for the attractive Euler--Poisson system \eqref{eq:EP} with \(\alpha = 1\), \(\beta=0\), and initial data \(\rho_0 \in \Pscr_\mathrm{c}(\R)\) and \(v_0 \in \Lp{\infty}(\R,\rho_0)\) is established. Note however that in this work, like in \cite{huang2001well}, entropy solutions are defined directly for the system \eqref{eq:EP} through a one-sided Lipschitz condition, see \eqref{eq:onesided} below, and weak continuity of \(\rho v^2\) to \(\rho_0 v_0^2\) as \(t\to0+\), see Remark \ref{rem:huang} for more details.
  
  Other pressureless Euler-type systems \cite{TT14} have recently been treated using the connection between sticky particle solutions and entropy solutions of scalar balance laws. For instance, solutions of the Euler-alignment system have been obtained for compactly supported \(\rho_0 \in \Pscr(\R)\) and \(v_0 \in \Lp{\infty}(\R,\rho_0)\) in \cite{leslie2021sticky}, where \(f[\rho]\) is replaced by an alignment force of Cucker--Smale-type involving the velocity \(v\).
  
  Connections between scalar conservation laws, optimal transport, monotone nondecreasing rearrangements, and gradient flows have been explored in several papers \cite{bolley2005contractive,CDL06,CDL07, brenier2009L2,bonaschi2015equivalence,difrancesco2016scalar}.
  
  Finally, we emphasize that a technical difficulty one needs to overcome is the lack of Lipschitz bounds for the flux of the associated conservation law to \eqref{eq:EF}.
  As in \cite{nguyen2015one}, we then build upon the uniqueness results of \cite{golovaty2012existence}, which is based on earlier ideas of \cite{carrillo1999entropy,karlsen2002note} rather than classical existence and uniqueness results for scalar conservation laws, cf.\ \cite{kruzkov1970first,bressan2000hyperbolic,holden2015front}.
  
  Since we only consider systems without pressure terms in the flux, we will from here on omit the word pressureless when we speak of the systems \eqref{eq:EF} and \eqref{eq:EP}.
  
  \subsection{Main results}
  Here we present a formal overview of our main results, namely the equivalence of the Lagrangian and entropy solution concepts for the Euler system \eqref{eq:EF}.
  
  A first clue to the connection between these solution concepts is the fact that they both apply a ``trick'' to reduce the second-order system for the variables \(\rho\) and \(v\) to a first order equation.
  For Lagrangian solutions, the new unknown is the optimal transport map, or monotone rearrangement, \( X \colon [0,T]\times (0,1) \to \R \), which pushes the Lebesgue measure \(\mathfrak{m}\) on \((0,1)\) to \(\rho(t,\cdot)\), that is, \(X(t,\cdot)_\#\mathfrak{m} = \rho(t,\cdot)\).
  For a given time \( t \in [0,T]\), the map \(X(t,\cdot)\) belongs to the cone \(\Kscr\) of nondecreasing functions in \(\Lp{2}(0,1)\).
  Note that concentration of mass in \(\rho(t,\cdot)\) corresponds to ``flat'' sections of \(X(t,\cdot)\), and in this setting another useful subspace of \(\Lp{2}(0,1)\) is the set \(\Hscr_{X}\) of functions which are constant wherever \(X\) is constant.
  The time-evolution of the transport map is subject to a differential inclusion, rather than an equation, involving a velocity \(U\), which may very well depend on \(X\), prescribed by the forcing term in the Euler system.
  This differential inclusion reads \(U-\dot{X} \in N_X\Kscr\), where \(N_X\Kscr\) is the normal cone of \(\Kscr\) at \(X\).
  However, when there is mass concentration, the prescribed velocity \(U\) may be incompatible with \(X\) remaining in the cone \(\Kscr\), that is, \(U\) does not belong to \(T_X\Kscr\), the tangent cone of \(\Kscr\) at \(X\).
  This is analogous to when in the method of characteristics, the characteristics are prescribed velocities which lead them to cross.
  In such situations, one way of ensuring that \(X\) remains nondecreasing  is to let its velocity \(\dot{X}\) be given by the \(\Lp{2}\)-projection of \(U\) onto \(\Hscr_{X} \subset T_X\Kscr\), i.e., \(\dot{X} = \Proj_{\Hscr_{X}}U\).
  This then ensures a flat section remains flat, moving with a common velocity.
  Of course, we then require \(U-\Proj_{\Hscr_{X}}U \in N_X\Kscr\) for \(X\) to still satisfy the differential inclusion.
  It then turns out that for all times, the velocity of \(X\) coincides with  \(\Lp{2}\)-projection \(\Proj_{T_X\Kscr}U\) of \(U\) onto the tangent cone \(T_X\Kscr\), i.e., the minimal element of the set \(U - N_X\Kscr\); hence, this acts as a selection principle.
  We refer to Section \ref{s:gradflow} for details on this solution framework.
  
  On the other hand, the new unknown in the entropy solution-setting is the distribution function of \(\rho\), that is \( M \colon [0,T]\times \R \to [0,1] \).
  In order to combine the two equations of the Euler system into a single conservation law, one then needs to devise an appropriate time-dependent flux function \(\Ucal(t,M)\), which in fact turns out to be the primitive of the prescribed velocity \(U\) of the Lagrangian solutions. This novel procedure generalizes the ideas in \cite{nguyen2008pressureless,leslie2021sticky} and leads to a conservation law of the form
  \begin{equation}\label{eq:clawintro}
    \del_t M + \del_x \Ucal(t,M) = 0.
  \end{equation}
  Here, concentration of mass in \(\rho(t,\cdot)\) corresponds to ``jumps'' in \(M(t,\cdot)\), i.e., a shock for the conservation law.
  If \(M\) is sufficiently regular, say, for simplicity, a piecewise smooth solution of the conservation law, we know from the theory that for \(M\) to be an admissible weak solution, such jumps must satisfy the Rankine--Hugoniot condition which determines the shock velocity as the ratio of the jump in \(\Ucal(t,M)\) to the jump in \(M\).
  Moreover, uniqueness of entropy solutions to \eqref{eq:clawintro} can be established if the jumps satisfy an inequality known as \textit{Ole\u{\i}nik's E-condition}, see \cite{oleinik1959uniqueness}.
  This condition determines whether a jump should be sustained as a shock, or if it should be smoothed out, turning it into a rarefaction wave.
  This framework is detailed in Section \ref{s:entropy}.
  
  We emphasize that the condition above should not be confused with another inequality \cite{oleinik1957discontinuous}, see also \cite[Section 11.2]{dafermos2016hyperbolic}, frequently called the Ole\u{\i}nik entropy condition and used to single out the desired solution.
  For \eqref{eq:EF} with \(f[\rho] \equiv 0\) this inequality takes the following form, cf.\ \cite{huang2001well,nguyen2008pressureless,natile2009wasserstein}; for a.e.\ \(t > 0\) we have
  \begin{equation}\label{eq:onesided}
    \frac{v(t,x_2)-v(t,x_1)}{x_2-x_1} \le \frac1t \enspace \text{for \(\rho\)-a.e.} \ x_1 \le x_2.
  \end{equation}
  
  Note that for a given probability measure \(\rho\), if its corresponding optimal transport map \(X\) and distribution function \(M\) both are taken to be right-continuous, \(X\) and \(M\) are each other's generalized inverses, see, e.g., \cite{delafortelle2015}.
  Then at least for the initial data, there is a correspondence between the two notions of solution.
  A more interesting issue is whether this reciprocity remains valid as these functions evolve in time, as it is not obvious that these procedures ``pick'' the same solutions.
  The key to answering this question lies in a more interesting correspondence between elements of the two notions presented above.
  For instance, as a consequence of being each others generalized inverses, a horizontal segment of an optimal transport map corresponds to a vertical segment of the distribution function.
  Then it would perhaps not be so surprising if the mechanisms preserving these segments, the projection \(\Proj_{\Hscr_{X}}U\) for the former and the Rankine--Hugoniot condition for the latter, were equivalent.
  A closer look at the definition of the projection \(\Proj_{\Hscr_{X}}\) and the relation between the prescribed velocity \(U\) and flux function \(\Ucal\) reveals that this is indeed the case.
  
  The next connection is perhaps less obvious, but thanks to a characterization of the normal cone \(N_X\Kscr\) due to Natile and Savar\'{e} \cite{natile2009wasserstein}, one can show that the requirement that the projection satisfies \(U-\Proj_{\Hscr_{X}}U \in N_X\Kscr\) is equivalent to the Ole\u{\i}nik E-condition for \(M\) and \(\Ucal(t,M)\).
  \begin{table}[ht]
    \begin{tabular}{l|lcl}
      Concept \(\backslash\) Framework & Lagrangian solution & & Entropy solution \\ \hline
      Solution variable & Optimal transport map \(X\)& 
      & Distribution function \( M \) \\
      Shock admissibility & \(\dot{X} = \Proj_{\Hscr_{X}}U\)
      & \(\iff\)
      & Rankine--Hugoniot condition \\
      Selection principle & \(U - \Proj_{\Hscr_{X}}U \in N_X\Kscr \) & 
      & Ole\u{\i}nik E-condition
    \end{tabular}
    \caption{Correspondence table for the notions of solution.}
    \label{tab:corr}
  \end{table}
  
  These correspondences are summarized in Table \ref{tab:corr} and lie at the heart of the argument for our main result, which can be informally stated as follows.
  
  \begin{mainres}
    Lagrangian $\iff$ Entropy solutions for 1D pressureless Euler systems. 
    The equivalence is based on showing the correspondence of the concepts in Table \ref{tab:corr}. 
  \end{mainres}
  
  Section \ref{s:equiv} contains precise statements and proofs of the above results.
  
  This framework allows us to establish a well-defined entropy solution for systems which previously have been studied up to the point where classical solutions break down. The main contributions of our work in this respect are:
  \begin{itemize}
    \item For the repulsive Euler--Poisson system, i.e., \eqref{eq:EP} with \(\alpha > 0\), mass may concentrate depending on the initial velocity configuration, but will tend to diffuse or spread out over time. Even if this can be obtained for Lagrangian solutions after a very careful and detailed reading of \cite{brenier2013sticky}, we here establish this result for entropy solutions to \eqref{eq:clawintro}, and therefore for Lagrangian solutions, as a consequence of the equivalence of the entropy and Lagrangian solution concepts.
    
    \item For the Euler system \eqref{eq:EPQ} with damping, i.e., \(\gamma > 0\), despite breakdown of classical solutions, we show for the first time that a globally defined entropy solution exists, tending asymptotically to a uniform mass distribution. This is in line with the behavior of globally defined classical solutions in \cite{carrillo2016pressureless}, while now this result is established asymptotically independently of mass concentration happening or not during some time interval. We show examples of temporary mass concentration in Section \ref{s:examples}.
  \end{itemize}
  
  Finally, we emphasize that the one-dimensional setting is essential for both frameworks:
  For the Lagrangian solutions, it gives a well-defined ordering to define the cone \(\Kscr\).
  On the other hand, for the entropy solutions it allows us to work with the distribution function \(M\) rather than the measure \( \rho \), which on the line uniquely determine one another.
  The one-dimensional setting is also needed for the reciprocal relationship between the optimal transport map and distribution function of a probability measure, i.e., that they are each other's generalized inverses.
  Since the 1D setting is essential for both solution concepts, it seems unlikely that these results may be generalized to higher dimensions, at least in any naive way.
  However, on the line one might allow the inclusion of more general forcing terms, e.g., the alignment force studied in the recent work \cite{leslie2021sticky}.
  
  Following \cite{brenier2013sticky}, Section \ref{s:motivation} is devoted to motivate the Lagrangian solutions using particle dynamics, as this concept might be more intuitive in this simpler setting.
  The particle dynamics will be studied further in Section \ref{s:EPparticle} in both repulsive and attractive Poisson forces showcasing the right continuation after collisions between particles dictated by the Lagrangian and entropy solution concepts. We finally discuss in Section \ref{s:examples} the repulsive Poisson force case with/without confinement and/or damping. These novel cases show the full power of this equivalence to characterize the unique continuations after blow-up in these additional examples pressureless Euler systems of the type \eqref{eq:EPQ}.
  
  
  \section{Motivation from particle dynamics}\label{s:motivation}
  In this section we will motivate the need for well-defined notions of solution from a particle-point of view, as the sticky particle dynamics lays the foundations for the works \cite{brenier2013sticky} and \cite{nguyen2008pressureless,nguyen2015one}.
  In particular we want to, as in \cite{brenier2013sticky}, set the stage for the Lagrangian solutions.
  This allows us to introduce a few concepts from convex analysis in a somewhat simple setting, which we can then recall as specific cases when we present the Lagrangian solutions in generality in Section \ref{s:gradflow}.
  We will also return to the particle dynamics in Section \ref{s:EPparticle} where we will study the particle dynamics for \eqref{eq:EP} in detail.
  
  When we talk about particle solutions of \eqref{eq:EP}, we mean linear combinations of Dirac measures in the following form, which is typically called an empirical measure in the case of \(\rho\),
  \begin{equation}\label{eq:particlesols}
    \rho(t,x) = \sum_{i=1}^n m_i \delta_{x_i(t)}, \qquad \rho v(t,x) = \sum_{i=1}^n m_i v_i(t) \delta_{x_i(t)},
  \end{equation}
  where \(x_i(t)\) for \(i \in \{1,\dots,n\}\) is the trajectory of a particle with velocity \(v_i(t)\) and mass \(m_i > 0\).
  Since we are considering \(\rho \in \Pscr(\R)\), the masses should always sum to one, i.e., \(\sum_{i=1}^n m_i = 1\).
  For the rest of this section we assume that the \(i\)\textsuperscript{th} particle has initial position \(x_i^0\) and velocity \(v_i^0\), where the initial positions are distinct and ordered:
  \begin{equation*}
    x_1^0 < x_2^0 < \cdots < x_{n-1}^0 < x_n^0.
  \end{equation*}
  At some later time, particles may meet, and so we will have to impose a rule for resolving the dynamics in those cases; for now, the only requirement we impose is that particle trajectories may not cross, that is, the initial ordering of particles is retained.
  
  \subsection{First considerations and simplifications}\label{ss:classical}
  Initially, the \(n\) distinct particles will be governed by the particle system corresponding to \eqref{eq:EP} given by 
  \begin{equation}\label{eq:EPparticle}
    \ddot{x}_i = -\frac{\alpha}{2}\sum_{j=1, j \neq i}^{n}m_j \sgn(x_i-x_j) - \beta, \qquad i \in \{1,\dots,n\},
  \end{equation}
  which can seen as Newton's second law for a set of particles moving in a conservative field governed by the scalar potential \(E_\mathrm{p}(\{x_i\})\) where
  \begin{equation*}
    E_\mathrm{p}(\{x_i\}) = \frac{\alpha}{2}\sum_{i=1}^{n}\sum_{\substack{j=1 \\ j \neq i}}^n m_i m_j \phi(x_i-x_j) + \beta \sum_{i=1}^n m_i x_i = \frac{\alpha}{4}\sum_{i=1}^{n}\sum_{\substack{j=1 \\ j \neq i}}^n m_i m_j \abs{x_i-x_j} + \beta \sum_{i=1}^n m_i x_i.
  \end{equation*}
  The first term on the right-hand side of \eqref{eq:EPparticle}, scaled by \(\alpha\), is the force per unit mass coming from the Poisson interaction forces between particles, and is therefore an internal force for the system.
  On the other hand, the second term scaled by \(\beta\) is the force per unit mass coming from an external force, which could be an external gravitational or electrical field.
  Introducing the kinetic energy \(E_\mathrm{k}(\{x_i\}) = \frac12 \sum_{i=1}^n m_i \dot{x}_i^2\), then the Euler--Lagrange equations for the Lagrangian \(L = E_\mathrm{k}-E_\mathrm{p}\) yields \eqref{eq:EPparticle}.
  Moreover, the energy \(E_\mathrm{k}+E_\mathrm{p}\), which corresponds to the Hamiltonian after introducing the momentum \(p_i = m_i \dot{x}_i\), is conserved.
  The total momentum \(\sum_{i=1}^n m_i \dot{x}_i\) however, is not conserved because of the external force.
  If we introduce the center of mass \(\bar{x} \coloneqq \sum_{i=1}^nm_i x_i\) and average velocity, which here coincides with the total momentum, \(\bar{v} = \sum_{i=1}^nm_i\dot{x}_i\), we find that \(\dot{\bar{x}} = \bar{v}\).
  Furthermore, \(\dot{\bar{v}} = -\beta\) since the sum of internal forces equals zero by Newton's third law, meaning that the total momentum changes at a linear rate, while the center of mass will follow a parabolic trajectory.
  If we instead of the absolute positions \(x_i\) consider the positions \(x_i'\) relative to the center of mass, i.e., \(x_i' = x_i - \bar{x}\), we can consider the system \eqref{eq:EPparticle} without external forces.
  As we have seen, the external force only affect the center of mass \(\bar{x}\), while the internal forces only affect the relative positions \(x_i'\) governed by \eqref{eq:EPparticle} with \(x_i\) replaced by \(x_i'\) and \(\beta =0\).
  The relative center of mass \(\sum_{i=1}^n m_i x_i' = 0\) and momentum \(\sum_{i=1}^n m_i \dot{x}_i' = 0\) are conserved, and the kinetic energy \(E_k\) can be decomposed as
  \begin{equation*}
      E_\mathrm{k}(\{\dot{x}_i\}) = \frac12\sum_{i=1}^nm_i\dot{x}_i^2 = \frac12 \dot{\bar{x}}^2 + \frac12 \sum_{i=1}^n m_i (\dot{x}_i')^2.
  \end{equation*}
  Therefore, we could without loss of generality consider the dynamics of the relative positions, not relabeled with primes, governed by \eqref{eq:EPparticle} with \(\beta = 0\), and add the motion of the center of mass afterwards to obtain the full dynamics.
  However, we choose to keep \(\beta\) in our approach, for ease of comparison to the works \cite{nguyen2008pressureless,nguyen2015one}.
  
  \subsection{Particle dynamics as weighted \(\ell^2\)-gradient flow}\label{ss:gradflow:disc}
  In this section, we follow \cite{brenier2013sticky} in motivating the differential inclusion formulation of the particle dynamics.
  For ease of notation, let us write \(\mathbf{m} = (m_1,\dots,m_n)\), \(\mathbf{x} = (x_1,\dots,x_n)\) and \(\mathbf{v} = (v_1,\dots,v_n) \in \R^n\), such that \(\mathbf{m} \in \R^n_+\) while \(\mathbf{x}, \mathbf{v} \in \R^n\). 
  Let us for the moment consider a slightly more general setting than in the previous section, where the dynamics initially are given by the accelerated motion
  \begin{equation}\label{eq:partgen}
    \dot{\mathbf{x}} = \mathbf{v}, \qquad \dot{\mathbf{v}} = \mathbf{a}_\mathbf{m}(\mathbf{x}),
  \end{equation}
  where \(\mathbf{a}_\mathbf{m} = (a_{\mathbf{m},1}(\mathbf{x}),\dots,a_{\mathbf{m},n}(\mathbf{x}))\) is defined for \(\mathbf{x} \in \R^n\).
  For the set of particles \(\mathbf{x} \in \R^n\), our rule of non-crossing trajectories, or well-ordering, is equivalent to requiring \(\mathbf{x}\) to lie in the convex cone
  \begin{equation}\label{eq:cone:part}
    \K^n \coloneqq \{ \mathbf{x} \in \R^n \colon x_j \le x_{j+1}, \, j \in \{1,\dots,n-1\} \}.
  \end{equation}
  When a non-empty subset of particles collide, we have \(\mathbf{x} \in \partial \K^n\), the boundary of the convex set.
  To be precise, let us introduce the set of indices for collided particles \(\Omega_{\mathbf{x}} \coloneqq \{j \colon x_j = x_{j+1}, \, j \in \{1,\dots,n-1\}\}\), then \(\partial \K^n = \{\mathbf{x} \in \K^n \colon \Omega_{\mathbf{x}} \neq \emptyset\}\).
  Now, since \(\mathbf{x}\) is required to lie within \(\K^n\), not all velocities are admissible for a given \(\mathbf{x}\).
  The set of admissible velocities at position \(\mathbf{x}\) is given by the \textit{tangent cone} \(T_\mathbf{x} \K^n\), defined as
  \begin{equation}\label{eq:TC1:disc}
    T_\mathbf{x} \K^n \coloneqq \mathrm{cl}\{\vartheta(\mathbf{y}-\mathbf{x}) \colon \mathbf{y} \in \K^n, \vartheta \ge 0 \},
  \end{equation}
  the closure of line segments in the direction \(\mathbf{y}-\mathbf{x}\) for which \(\mathbf{y}\) belongs to \(\K^n\).
  One can show that \eqref{eq:TC1:disc} is equivalent to
  \begin{equation}\label{eq:TC2:disc}
    T_\mathbf{x} \K^n = \{ \mathbf{v} \in \R^n \colon v_j \le v_{j+1} \: \forall {j, j+1} \in \Omega_{\mathbf{x}} \}.
  \end{equation}
  We see that if \(\Omega_{\mathbf{x}} = \emptyset\), then all velocities are admissible, i.e., \(T_\mathbf{x}\K^n = \R^n\).
  However, when particles collide at time \(t\), i.e., \(\mathbf{x}(t) \in \del \K^n\), we have to change the dynamics \eqref{eq:partgen} in order to stay within the cone.
  One could imagine several ways of doing this, but assuming inelastic collisions, one is led to a projection
  \begin{equation*}
    \mathbf{v}(t+) = \Proj_{T_{\mathbf{x}(t)}\K^n}\mathbf{v}(t-),
  \end{equation*}
  where \(\Proj_{T_{\mathbf{x}(t)}\K^n}\) is the weighted \(\ell^2\)-projection onto \(T_{\mathbf{x}(t)}\K^n\) with respect to the norm \(\norm{\cdot}_{\mathbf{m}}\) induced by the weighted inner product \(\ip{\cdot}{\cdot}_{\mathbf{m}}\), that is
  \begin{equation*}
    \ip{\mathbf{v}}{\mathbf{w}}_{\mathbf{m}} = \sum_{i=1}^n m_i v_i w_i,  \qquad \norm{\mathbf{v}}_\mathbf{m} = \left(\sum_{i=1}^n m_i v_i^2 \right)^{1/2}.
  \end{equation*}
  It is then also true that the new velocity is minimizing in the sense \(\ip{\mathbf{v}(t-)-\mathbf{v}(t+)}{\mathbf{u}}_{\mathbf{m}} \le 0\) for all \(\mathbf{u} \in T_{\mathbf{x}(t)}\K^n\).
  Thinking of this change in velocity as the result of an instantaneous force, this force must be an element of the \textit{normal cone} \(N_{\mathbf{x}(t)}\K^n\), defined by
  \begin{equation*}
    N_\mathbf{x}\K^n = \{ \mathbf{n} \in \R^n \colon \ip{\mathbf{n}}{\mathbf{y}-\mathbf{x}}_{\mathbf{m}} \le 0 \; \forall \mathbf{y} \in \K^n \}.
  \end{equation*}
  which coincides with the subdifferential \(\del I_{\K^n}(\mathbf{x})\) of the indicator function \(I_{\K^n}\) of \(\K^n\) at the point \(\mathbf{x}\).
  This then suggests the following generalization of \eqref{eq:partgen} to describe the dynamics also during collisions, namely the second order differential inclusion
  \begin{equation}\label{eq:DI2:disc}
    \dot{\mathbf{x}} = \mathbf{v}, \qquad \dot{\mathbf{v}} + N_{\mathbf{x}}\K^n \ni \mathbf{a}_\mathbf{m}(\mathbf{x}).
  \end{equation}
  
  It was shown in \cite{natile2009wasserstein} that along a curve \(t \mapsto \mathbf{x}(t)\) satisfying a global stickiness condition, there is a monotonicity property for the normal cones; that is,
  \(N_{\mathbf{x}(s)}\K^n \subset N_{\mathbf{x}(t)}\K^n\) for \(s < t\).
  This property can then be used to reduce the second order differential inclusion \eqref{eq:DI2:disc} to a first order inclusion.
  Under this monotonicity, for a map \(\xi \colon [0,\infty) \to \R^n\) with \(\xi(t)\in N_{\mathbf{x}(t)}\K^n\), e.g., \(\mathbf{a}_\mathbf{m} - \dot{\mathbf{v}}\) in \eqref{eq:DI2:disc}, then
  \begin{equation*}
    \int_s^t \xi(r)\dee r \in N_{\mathbf{x}(t)}\K^n
  \end{equation*}
  for all \(t > s\).
  Formally, we then have
  \begin{equation*}
    \mathbf{v}(t) + N_{\mathbf{x}(t)}\K^n \ni \mathbf{v}(s) + \int_s^t \mathbf{a}_\mathbf{m}(\mathbf{x}(r))\dee r,
  \end{equation*}
  and so we can rewrite \eqref{eq:DI2:disc} as the first-order differential inclusion
  \begin{equation}\label{eq:DI1:disc}
    \dot{\mathbf{x}} + N_{\mathbf{x}}\K^n \ni \mathbf{u}, \qquad \dot{\mathbf{u}} = \mathbf{a}_\mathbf{m}(\mathbf{x}).
  \end{equation}
  Now, the idea presented in \cite{brenier2013sticky} is that one can introduce a well-defined solution of the differential inclusion \eqref{eq:DI1:disc} by choosing \(\dot{\mathbf{x}} = \mathbf{v}\) where the velocity \(\mathbf{v}\) is the minimal element in \(\mathbf{u}-N_{\mathbf{x}}\K^n\), i.e., \(\mathbf{v} = \Proj_{T_{\mathbf{x}}\K^n}\mathbf{u}\).
  Then, since we define the particle solution to be the trajectory for which the velocity is the unique element in the tangent cone \(T_\mathbf{x}\K^n\) closest to the free velocity \(\mathbf{u}\) with respect to the weighted \(\ell^2\)-norm \(\norm{\cdot}_{\mathbf{m}}\), this can be regarded as a discrete version of \(\Lp{2}\)-gradient flow.
  In the next section we will see how this principle can be adapted to the continuous setting and an acceleration induced by the forcing term \(f[\rho]\) in \eqref{eq:EF}.
  
  \section{\(\Lp{2}\)-gradient flow} \label{s:gradflow}
  We will here present the concept of Lagrangian solutions for the Euler system \eqref{eq:EF} as introduced by Brenier, Gangbo, Savar\'{e} and Westdickenberg \cite{brenier2013sticky}.
  Note that for the first order interaction equation \eqref{eq:interact} in \cite{bonaschi2015equivalence}, this corresponds to their \(\Lp{2}\)-gradient flow solutions.
  This notion of solution relies on concepts from convex analysis and theory on differential inclusions found in the monograph of Br\'{e}zis \cite{brezis1973operateurs}, in addition to results of Natile and Savar\'{e} \cite{natile2009wasserstein} from their study of the Euler system without forcing, i.e., \(f[\rho] \equiv 0\).  
  For the reader's convenience we will restate several of their definitions and results which we make use of. We refer to Section 3, see also Section 6, of \cite{brenier2013sticky} for the complete picture.
  
  \subsection{Elements of optimal transport}
  Consider the space \(\Pscr(\R^m)\) of Borel probability measures on \(\R^m\).
  The push-forward \(\nu \coloneqq Y_\#\mu\) of a measure \(\mu \in \Pscr(\R^m)\) under the Borel map \(Y\colon \R^m \to \R^n\) is the measure \(\nu\) defined through \(\nu(A) = \mu(Y^{-1}(A))\) for all Borel sets \(A \subset \R^n\).
  For any Borel map \(\phiv \colon \R^n \to [0,\infty]\) we then have the change-of-variable formula
  \begin{equation} \label{eq:varchange}
    \int_{\R^n} \phiv(y)\dee Y_\#\mu(y) = \int_{\R^m} \phiv(Y(x))\dee\mu(x).
  \end{equation}
  For \(p \in [1,\infty) \) we let \(\Pscr_p(\R^n)\) denote the space of all \(\rho \in \Pscr(\R^n)\) with bounded \(p\)\textsuperscript{th} moment, i.e., \(\int_{\R^n}\abs{x}^p\dee\rho(x) < \infty\).
  Then the \(p\)-Wasserstein distance between two measures \(\rho_1, \rho_2 \in \Pscr_p(\R^n) \) can be defined as
  \begin{equation} \label{eq:Wasserstein}
    d_{\Wscr_p}(\rho_1,\rho_2)^p \coloneqq \min \left\{  \int_{\R^n\times\R^n} \abs{x-y}^p \dee\varrho(x,y) \colon \varrho \in \Pscr(\R^n\times\R^n), \: \varpi^{i}_\#\varrho = \rho_i \right\},
  \end{equation}
  where \( \varpi^i(x_1,x_2) = x_i \) is the projection onto the \(i\)\textsuperscript{th} coordinate.
  It can be shown that there always exists an optimal transport plan \(\varrho\) which achieves the infimum of the integral in \eqref{eq:Wasserstein}.
  We will be concerned with \(p=2\) and measures on the line, i.e., \(n=1\).
  In this setting there is a unique optimal plan, which can be characterized explicitly.
  
  For any \(\rho \in \Pscr(\R)\) we define its right-continuous distribution function
  \begin{equation*}
    M_\rho(x) \coloneqq \rho((-\infty,x]) \quad\text{for all } x \in \R,
  \end{equation*}
  noting that \( \del_x M_\rho = \rho \) in \(\Dscr'(\R)\), i.e., in the distributional sense.
  Then we can define its right-continuous generalized inverse as
  \begin{equation*}
    X_\rho(m) \coloneqq \inf\{x \colon M_\rho(x) > m \} \quad \text{for all } m \in \Omega,
  \end{equation*}
  where \(\Omega \coloneqq (0,1)\), which is clearly nondecreasing.
  This is in fact the optimal transport map pushing the one-dimensional Lebesgue measure \(\mathfrak{m} \coloneqq \mathcal{L}^1|_\Omega\) to \(\rho\) on \(\R\),
  the Lebesgue measure on \(\Omega\), meaning
  \begin{equation*}
    (X_\rho)_\#\mathfrak{m} = \rho, \qquad \int_\R \phiv(x)\dee\rho(x) = \int_\Omega \phiv(X_\rho(m))\dee m
  \end{equation*}
  for any Borel map \( \phiv \colon \R \to [0,\infty] \) by \eqref{eq:varchange}, where \(\dee m\) denotes the integration with respect to \(\mathfrak{m}\). 
  For instance, \( \rho \in \Pscr_2(\R) \) if and only if \( X_\rho \in \Lp{2}(\Omega) \).
  The Hoeffding--Frech\'{e}t theorem, cf.\ \cite[Theorem 2.18]{villani2003topics}, then characterizes the optimal transport plan \(\varrho\) between \(\rho_1\) and \(\rho_2\) as 
  \begin{equation}\label{eq:optplan}
    \varrho = (X_{\rho_1,\rho_2})_\#\mathfrak{m}, \quad X_{\rho_1, \rho_2} = (X_{\rho_1}, X_{\rho_2}) \quad \text{for all } m \in \Omega.
  \end{equation}
  Therefore we have that their 2-Wasserstein distance is given by
  \begin{equation*}
    d_{\Wscr_2}(\rho_1,\rho_2)^2 = \int_\Omega \abs{X_{\rho_1}(m)-X_{\rho_2}(m)}^2\dee m = \norm{X_{\rho_1}-X_{\rho_2}}^2_{\Lp{2}(\Omega)}.
  \end{equation*}
  The map \(\rho \mapsto X_\rho\) is an isometry between \(\Pscr_2(\R)\) and \(\Kscr\), where
  \begin{equation}\label{eq:cone}
    \Kscr \coloneqq \{ X \in \Lp{2}(\Omega) \colon X \text{ is nondecreasing} \},
  \end{equation}
  and without loss of generality we may consider right-continuous representatives in \(\Kscr\), which are defined everywhere.
  
  \subsubsection{Comparison to the particle case}
  The idea is now to replace the set of particles labeled by \(i \in \{1,\dots,n\}\) with ``particles'' labeled by \(m \in \Omega = (0,1)\) and which at time \(t\) has position \(X(t,m) \in \R\).
  Of course, these are no longer discrete particles, but a continuum of mass distributed over the real line.
  For a fixed \(t\), we know from the previous considerations that \(X\) can be uniquely characterized by a measure \(\rho\), i.e., it is the nondecreasing and right-continuous map \(X \colon \Omega \to \R\) such that
  \begin{equation*}
    X(m) \le x \iff m \le \rho((-\infty,x]) \qquad \text{for all } x \in \R.
  \end{equation*}
  
  Going the other way, to any solution \((\rho,v)\) of \eqref{eq:EF} we can associate a map \(X \colon [0,\infty) \times \Omega \to \R\) with \(X(t,\cdot)\) nondecreasing, and a velocity \(V \colon [0,\infty) \times \Omega \to \R\) such that for \(t \ge 0\) we have
  \begin{equation}\label{eq:XV}
    X(t,\cdot)_\#\mathfrak{m} = \rho(t,\cdot), \qquad V(t,\cdot) = v(t,X(t,\cdot)) = \del_t X(t,\cdot).
  \end{equation}
  The idea is now, analogous to the discrete particle case, to associate \eqref{eq:EF} to a differential inclusion for the pair \((X,V)\).
  Here the cone \(\K^n\) is replaced with the set of optimal transport maps \(\Kscr\) defined in \eqref{eq:cone}, which can be shown to be a closed, convex cone in \(\Lp{2}(\Omega)\).
  To handle mass concentration, corresponding to particle collisions in the discrete case, we will recall some results from convex analysis.
  
  \subsection{Elements of convex analysis}\label{ss:convex}
  
  The normal cone \(N_X \Kscr\) of \(\Kscr\) at \(X \in \Kscr\) is
  \begin{equation*}
    N_X \Kscr \coloneqq \left\{ W \in \Lp{2}(\Omega) \colon \int_\Omega W (Y-X)\dee m \le 0 \; \forall Y \in \Kscr \right\},
  \end{equation*}
  where, as in the discrete case, \(N_X \Kscr = \del I_\Kscr(X)\), the subdifferential at \(X\) of the indicator function \(I_\Kscr \colon \Lp{2}(\Omega) \to [0,+\infty]\),
  \begin{equation*}
    I_\Kscr(X) = \begin{cases}
      0, & X \in \Kscr, \\ +\infty, & X \notin \Kscr.
    \end{cases}
  \end{equation*}
  Moreover, the tangent cone \(T_X\Kscr\) at \(X\) can be defined as the polar cone of \(N_X\Kscr\), that is,
  \begin{equation}\label{eq:TC1}
    T_X\Kscr \coloneqq \left\{ U \in \Lp{2}(\Omega) \colon \int_\Omega U(m) W(m)\dee m \le 0 \; \forall W \in N_X\Kscr \right\}.
  \end{equation}
  We recall from the particle dynamics in Section \ref{ss:gradflow:disc} that there were two equivalent characterizations \eqref{eq:TC1:disc} and \eqref{eq:TC2:disc} of the tangent cone \(T_{\mathbf{x}}\K^n\).
  This is true also in our current setting, for the tangent cone \(T_X\Kscr\) as well as the normal cone \(N_X\Kscr\), and it will be convenient to work with both characterizations.
  For the alternative descriptions, we need, analogously to the set of collided particles \(\Omega_\mathbf{x}\), the set of concentrated mass \(\Omega_X\), namely
  \begin{equation}\label{eq:OX}
    \Omega_X \coloneqq \{ m \in \Omega \colon X \text{ is constant in a neighborhood of } m \}.
  \end{equation}
  The following characterization, illustrated in Figure \ref{fig:NX}, can be found in \cite[Lemma 2.3]{brenier2013sticky} and builds upon \cite[Theorem 3.9]{natile2009wasserstein}.
  \begin{lem}[Characterization of the normal cone \(N_X\Kscr\)]\label{lem:NX}
    Let \(X \in \Kscr\) and \(W \in \Lp{2}(\Omega)\) be given, and write
    \begin{equation*}
      \Xi_W(m) \coloneqq \int_0^m W(\omega)\dee\omega \quad \text{for all } m \in [0,1]. 
    \end{equation*}
    Then \(W \in N_X\Kscr\) if and only if \(\Xi_W \in \Nscr_X\), where \(\Nscr_X\) is the convex cone defined as
    \begin{equation*}
      \Nscr_X \coloneqq \left\{  \Xi \in C([0,1]) \colon \Xi \ge 0 \enspace\mathrm{in}\enspace [0,1] \enspace\mathrm{and}\enspace \Xi = 0 \enspace \mathrm{in}\enspace \Omega\setminus\Omega_X \right\}.
    \end{equation*}
    In particular, for every \(X_1, X_2 \in \Kscr\) we have
    \begin{equation}\label{eq:OX&NX}
      \Omega_{X_1} \subset \Omega_{X_2} \implies N_{X_1}\Kscr \subset N_{X_2}\Kscr.
    \end{equation}
  \end{lem}
  \begin{figure}
    \includegraphics[width=0.8\linewidth]{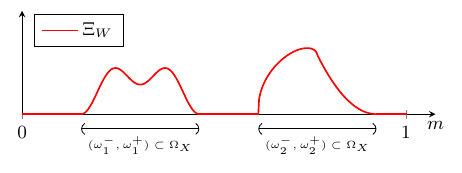}
    \caption{Example of \(\Xi_W \in \Nscr_X\) from Lemma \ref{lem:NX}. The corresponding \(X\) is essentially constant on the intervals \((\omega_1^-, \omega_1^+)\) and \((\omega_2^-, \omega_2^+)\).}
    \label{fig:NX}
  \end{figure}
  
  Moreover, we have a useful tangent cone characterization, found in \cite[Lemma 2.4]{brenier2013sticky}: given \(X \in \Kscr\), then
  \begin{equation}\label{eq:TC2}
    T_X\Kscr = \left\{ U \in \Lp{2}(\Omega) \colon U \text{ is nondecreasing on each interval } (m_l, m_r) \subset \Omega_X \right\}.
  \end{equation}
  
  Based on the set \(\Omega_X\) in \eqref{eq:OX} we can for \(X \in \Kscr\) also introduce the closed subspace \(\Hscr_X \subset \Lp{2}(\Omega)\) given by
  \begin{equation}\label{eq:HX}
    \Hscr_X \coloneqq \left\{ U \in \Lp{2}(\Omega) \colon U \text{ is constant on each interval } (m_l,m_r) \subset \Omega_X \right\}.
  \end{equation}
  For \(U \in \Lp{2}(\Omega)\), the \(\Lp{2}\)-projection onto this subspace is then given by
  \begin{equation}\label{eq:Proj:HX}
    \Proj_{\Hscr_X}U = \begin{cases}
      U, & \text{for a.e.\ } m \in \Omega\setminus\Omega_X \\
      \frac{1}{m_r-m_l}\int_{m_l}^{m_r} U(\omega)\dee \omega, & \text{for a.e.\ } m \in (m_l, m_r), \text{ a maximal interval of } \Omega_X.
    \end{cases}
  \end{equation}
  As can be expected from the particle dynamics, such a projection can be applied to the velocity to ensure that \(X\) stays in the cone \(\Kscr\).
  A direct consequence of \eqref{eq:TC2} is the equivalence relation
  \begin{equation*}
    U \in \Hscr_{X} \iff \pm U \in T_X \Kscr.
  \end{equation*}
  Moreover, if \(X_1, X_2 \in \Kscr\), we have the implication
  \begin{equation}\label{eq:OX&HX}
    \Omega_{X_1} \subset \Omega_{X_2} \implies \Hscr_{X_2} \subset \Hscr_{X_1}
  \end{equation}
  Observe that if we take \(U \in \Hscr_{X}\) then the integral in \eqref{eq:TC1} vanishes, meaning \(N_X\Kscr \subset \Hscr_X^\perp\) for all \(X \in \Kscr\), where \(\Hscr_X^\perp\) is the orthogonal complement of \(\Hscr_{X}\).
  As a consequence,
  \begin{equation}\label{eq:Proj:K&HY}
    Y = \Proj_{\Kscr}X \implies Y = \Proj_{\Hscr_{Y}}X, \quad \Hscr_{Y} \subset \Hscr_{X}. 
  \end{equation}
  In fact, \cite[Lemma 2.5]{brenier2013sticky} provides a more precise characterization of \(\Hscr_X\) in terms of \(N_X\Kscr\). Indeed, we have
  \begin{equation*}
    \Hscr_X^\perp = \left\{ W \in \Lp{2}(\Omega) \colon W = 0 \text{ a.e.\ in } \Omega\setminus \Omega_X, \int_{m_l}^{m_r}W(m)\dee m = 0 \: \forall \text{ maximal interval } (m_l, m_r) \subset \Omega_X \right\},
  \end{equation*}
  and it is the closed, linear subspace of \(\Lp{2}(\Omega)\) generated by \(N_X\Kscr\). Moreover it satisfies
  \begin{equation}\label{eq:HXperp}
    \Hscr_X^\perp = \left\{ W \in \Lp{2}(\Omega) \colon \int_{\Omega} W(m) \phiv(X(m))\dee m = 0 \: \forall \phiv \in C_{\mathrm{b}}(\R) \right\}.
  \end{equation}
  
  \subsection{Lagrangian solutions}\label{ss:LagDI}
  Following the same arguments as for the particle dynamics in Section \ref{ss:gradflow:disc}, we are led to study first-order differential inclusions of the form
  \begin{equation}\label{eq:DIV}
    \dot{X}(t) + \del I_\Kscr(X(t)) \ni V^0 + \int_{0}^{t}F[X(s)]\dee s \text{ for a.e. } t \ge 0, \qquad X(0) = \lim\limits_{t\to0+}X(t) = X^0.
  \end{equation}
  Here \(\dot{X}\) is the differential map of \(X \colon [0,\infty) \to \Kscr \subset \Lp{2}(\Omega) \).
  \subsubsection{The Lagrangian force representation}
  The link between \eqref{eq:DIV} and the Euler system \eqref{eq:EF} lies in the Lagrangian representation \(F \colon \Kscr \to \Lp{2}(\Omega)\) of the force distribution \(f \colon \Pscr(\R) \to \Mscr(\R)\), which must satisfy the distributional identity
  \begin{equation}\label{eq:f&F}
    \int_\R \phiv(x) f[\rho](\dee x) = \int_\Omega \phiv(X_\rho(m)) F[X](m)\dee m \text{ for all } \phiv \in \mathscr{D}(\R) \text{ and } \rho \in \Pscr(\R),
  \end{equation}
  where \(X_\rho \in \Kscr\) and \((X_\rho)_{\#} \mathfrak{m} = \rho\).
  From \eqref{eq:HXperp} we see that \(F[X]\) is not uniquely defined unless \(\Hscr_X^\perp = \{0\} \), meaning \(\Hscr_X = \Lp{2}(\Omega)\), that is, \(\Omega_X = \emptyset\); this is exactly when \(X\) is strictly increasing.
  In order to facilitate the study of solutions for \eqref{eq:DIV}, in addition to \eqref{eq:f&F} we will require \(F\) to be everywhere defined in \(\Kscr\) and satisfy additional boundedness and continuity properties. 
  \begin{dfn}[Boundedness] \label{dfn:bnd:F}
    The operator \(F \colon \Kscr \to \Lp{2}(\Omega)\) is bounded if there is a constant \(C \ge 0\) such that
    \begin{equation}\label{eq:bnd}
      \norm{F[X]}_{\Lp{2}(\Omega)} \le C \left(1 + \norm{X}_{\Lp{2}(\R)} \right) \text{ for all } X \in \Kscr.
    \end{equation}
    We say it is pointwise linearly bounded if there is a constant \(C_\mathrm{p} \ge 0\) such that
    \begin{equation*}
      \abs{F[X](m)} \le C_\mathrm{p} \left( 1 + \abs{X(m)} + \norm{X}_{\Lp{1}(\Omega)} \right).
    \end{equation*}
  \end{dfn}
  Observe that if \(F\) is pointwise linearly bounded, then it is bounded with \(C = 2 C_\mathrm{p}\) in \eqref{eq:bnd}.
  Recall that a modulus of continuity \(\omega\) is a continuous and concave function \(\omega \colon [0,\infty) \to [0,\infty)\) satisfying  \(0 = \omega(0) < \omega(r)\) for all \(r > 0\).
  \begin{dfn}[Uniform continuity]\label{dfn:unicont:F}
    We say that \(F \colon \Kscr \to \Lp{2}(\Omega)\) is uniformly continuous if there is a modulus of continuity \(\omega\) such that
    \begin{equation}\label{eq:unicont:F}
      \norm{F[X_1] - F[X_2]}_{\Lp{2}(\Omega)} \le \omega\left(\norm{X_1-X_2}_{\Lp{2}(\Omega)}\right) \quad \text{for all } X_1, X_2 \in \Kscr.
    \end{equation}
    Similarly, \(F\) is Lipschitz continuous if \eqref{eq:unicont:F} holds with \(\omega(r) = Lr\) for some \(L > 0\) and all \(r \ge 0\).
  \end{dfn}
  If a uniformly continuous \(F\) is defined by \eqref{eq:f&F} on the convex subset \(\Kscr_\mathrm{si} \subset \Kscr\) of strictly increasing maps and satisfies \eqref{eq:unicont:F} on \(\Kscr_\mathrm{si}\), it admits a unique extension to \(\Kscr\) which preserves the compatibility \eqref{eq:f&F} and continuity \eqref{eq:unicont:F} conditions.
  
  A final important property of \(F\) concerns its action on the subset \(\Omega_X\) where \(X\) is constant, i.e., where mass concentrates.
  For a sticky evolution of \(X\), this set would be nondecreasing in time, and since the force map determines the change in velocity, it would be natural to assume \(F[X] \in \Hscr_X\) for \(X \in \Kscr\) in this setting.
  However, it turns out that the following weaker condition is sufficient to preserve the sticky property.
  \begin{dfn}[Sticking]\label{dfn:sticking}
    The map \(F \colon \Kscr \to \Lp{2}(\Omega)\) is called sticking if
    \begin{equation}\label{eq:sticking}
      F[X] - \Proj_{\Hscr_{X}}F[X] \in \del I_\Kscr(X) \enspace\text{for all } X \in \Kscr.
    \end{equation}
  \end{dfn}
  
  \subsubsection{Definition, existence and uniqueness}
  The definition of a Lagrangian solution of \eqref{eq:DIV} is given in \cite[Definition 3.4]{brenier2013sticky}, which we recall below.
  \begin{dfn}[Lagrangian solutions of \eqref{eq:DIV}]\label{dfn:LagSol}
    Let \(F\colon \Kscr \to \Lp{2}(\Omega)\) be a uniformly continuous operator and let \(X^0 \in \Kscr\), \(V^0 \in \Lp{2}(\Omega)\) be given.
    A Lagrangian solution to \eqref{eq:DIV} with initial data \((X^0,V^0)\) is a curve \(X \in \mathrm{Lip}_\mathrm{loc}([0,\infty);\Kscr)\) satisfying \(X(0) = X^0\) and \eqref{eq:DIV}.
  \end{dfn}
  If we introduce the velocity \(U\) prescribed by the initial data \(V^0\) and the force \(F\), i.e.,
  \begin{equation}\label{eq:U}
    U(t) \coloneqq V^0 + \int_0^t F(X(s))\dee s,
  \end{equation}
  we see that \eqref{eq:DIV} is equivalent to
  \begin{subequations}\label{eq:DIU}
    \begin{alignat}{2}
      \dot{X}(t) &+ \del I_\Kscr(X(t)) \ni U(t) \text{ for a.e.\ } t \ge 0, &\qquad X(0) &= X^0, \label{eq:DIU:X} \\
      \dot{U}(t) &= F(X(t)), & U(0) &= V^0, \label{eq:DIU:U}
    \end{alignat}
  \end{subequations}
  where we note that the continuity of \(F\) yields \(U \in C^1([0,\infty); \Lp{2}(\Omega))\).
  Here and for the rest of the paper we use the term \textit{prescribed velocity} to distinguish \(U\) from the actual velocity of \(X\), since the two do not coincide in general; indeed, this happens when \(U \notin T_X\Kscr\).
  
  Lagrangian solutions exhibit several useful properties, contained in the following \cite[Theorem 3.5]{brenier2013sticky}.
  \begin{lem}[Properties of Lagrangian solutions]\label{lem:LagSol:props}
    Let \(F \colon \Kscr \to \Lp{2}(\Omega)\) be a uniformly continuous operator and let \((X,U)\) be a corresponding Lagrangian solution of \eqref{eq:DIU}.
    Then the following properties hold:
    \begin{itemize}
      \item The right-derivative of \(X(t)\) exists for all \(t \ge 0\) and equals
      \begin{equation*}
        V \coloneqq \diff{^+}{t}X.
      \end{equation*}
      \item The velocity \(V(t)\) for \(t \ge 0\) is the minimal element of the closed and convex set \(U(t)-\del_\Kscr I(X(t))\). Note that if we replace \(\dot{X}\) with \(V\) then \eqref{eq:DIV} and \eqref{eq:DIU} hold for all \(t \ge 0\).
      \item For all \(t \ge 0\), the minimal element \(V\) is the projection on the tangent cone given by
      \begin{equation*}
        V(t) = \Proj_{T_X(t)\Kscr} U(t)
      \end{equation*}
      \item \(V(t)\) is right-continuous for all \(t \ge 0\), and in particular
      \begin{equation*}
        \lim\limits_{t\to0+}V(t) = V^0 \iff V^0 \in T_{X^0}\Kscr.
      \end{equation*}
      Moreover, if \(\Tcal \subset (0,\infty)\) is the subset of times at which \(t \mapsto \norm{V(t)}_{\Lp{2}(\Omega)}\) is continuous, then \((0,\infty)\setminus\Tcal\) is negligible, and \(V\) is continuous, \(X\) is differentiable in \(\Lp{2}(\Omega)\) at every point of \(\Tcal\).
      Setting \(\rho(t) = X(t)_\#\mathfrak{m}\), there exists a unique map \(v(t,\cdot) \in \Lp{2}(\Omega,\rho)\) such that
      \begin{equation*}
        \dot{X}(t) = V(t) = \Proj_{\Hscr_{X(t)}} U(t) = v(t,\cdot)\circ X(t) \in \Hscr_{X(t)}, \quad \text{for every } t \in \Tcal.
      \end{equation*}
      \item If moreover \(F\) is linked to \(f\) by \eqref{eq:f&F}, \(\rho_0 = X^0_\#\mathfrak{m}\), \(V^0 = v_0\circ X^0\), then \((\rho,v)\) defined above is a distributional solution to \eqref{eq:EF} such that
      \begin{equation*}
        \lim\limits_{t\to0+}\rho(t,\cdot) = \rho_0 \text{ in } \Pscr_2(\R), \qquad \lim\limits_{t\to0+}\rho(t,\cdot)v(t,\cdot) = \rho_0 v_0 \text{ in } \Mscr(\R)
      \end{equation*}
    \end{itemize}
  \end{lem}
  
  Assuming \(F\) to be merely pointwise linearly bounded and uniformly continuous, one can show existence of Lagrangian solutions to \eqref{eq:DIU}.
  We will however be interested in the cases where we also have uniqueness, and to this end we have the following result from \cite[Theorem 3.6]{brenier2013sticky}.
  \begin{thm}[Existence and uniqueness of Lagrangian solutions]\label{thm:LagSol:EU}
    Assume \(F \colon \Kscr \to \Lp{2}(\Omega)\) is Lipschitz.
    Then for every \(T >0\) and \((X^0,V^0) \in \Kscr \times \Lp{2}(\Omega)\) there exists a unique Lagrangian solution \(X\) to \eqref{eq:DIV} for \(t \in [0,T]\).
  \end{thm}
  
  Let us return to the matter of sticky evolution: it turns out that we may define a subclass of the Lagrangian solutions with this behavior.
  \begin{dfn}[Sticky Lagrangian solutions]\label{dfn:LagSol:sticky}
    A Lagrangian solution is called \textit{sticky} if for \(t_1 \le t_2\) we have \(\Omega_{X(t_1)} \subset \Omega_{X(t_2)}\).
  \end{dfn}
  Because of the implications \eqref{eq:OX&NX} and \eqref{eq:OX&HX},
  any sticky Lagrangian solution satisfies the monotonicity condition
  \begin{equation*}
    \del I_{\Kscr}(X(t_1)) \subset \del I_{\Kscr}(X(t_2)), \quad \Hscr_{X(t_2)} \subset \Hscr_{X(t_1)} \enspace \text{for any } t_1 \le t_2.
  \end{equation*}
  From this property one may derive the following representation formulas for sticky solutions, \cite[Proposition 3.8]{brenier2013sticky}.
  \begin{prp}[Projection formula]\label{prp:projFormula}
    If \(X\) is a sticky Lagrangian solution, then
    \begin{equation*}
      V(t) \in \Hscr_{X(t)} \enspace \text{for all times } t \ge 0
    \end{equation*}
    and \((X,V)\) satisfy
    \begin{align}
      X(t) &= \Proj_{\Kscr} \left( X^0 + \int_0^t U(s)\dee s \right) = \Proj_{\Kscr} \left( X^0 + t V^0 + \int_0^t (t-s) F[X(s)]\dee s \right), \label{eq:sticky:X} \\
      V(t) &= \Proj_{\Hscr_{X(t)}} U(t) = \Proj_{\Hscr_{X(t)}} \left( V^0 + \int_0^t F[X(s)]\dee s \right). \label{eq:sticky:V}
    \end{align}
  \end{prp}
  
  As mentioned before, there is a sufficient criterion for a Lagrangian solution to be sticky.
  Indeed, by \cite[Theorem 3.11]{brenier2013sticky}, if the force \(F\) is Lipschitz continuous and sticking in the sense of Definition \ref{dfn:sticking}, then every Lagrangian solution of \eqref{eq:DIV} with \((X^0,V^0) \in \Kscr \times \Hscr_{X^0}\) is (globally) sticky,
  and given by the projection formula of Proposition \ref{prp:projFormula}.
  
  \begin{rem}[Generalized Lagrangian solutions]
    In \cite[Section 4]{brenier2013sticky} the authors introduce the notion of generalized Lagrangian solutions.
    We will not go into detail on these, as they lack uniqueness with respect to initial data.
    However, we mention that forcing the evolution of \(X\) to be sticky, despite \(F\) not satisfying the sticking condition \eqref{eq:sticking}, which would correspond to replacing \(F[X]\) with \(\Proj_{\Hscr_{X}}F[X]\) in \eqref{eq:DIV} (which would still satisfy \eqref{eq:f&F}), yields such a generalized Lagrangian solution.
    We will compare the Lagrangian solution to this type of generalized solution in Example \ref{exm:2particles} and in Section \ref{ss:BGSW}.
  \end{rem}
  
  \subsection{Lagrangian solutions for the Euler system}\label{ss:LagEF}
  From the final bullet point of Lemma \ref{lem:LagSol:props} we have that a Lagrangian solution of the differential inclusion \eqref{eq:DIV} provides us with a solution of the original Euler system \eqref{eq:EF}.
  In this section we will see how suitable assumptions on the force distribution \(f[\rho]\) allow us to apply Theorem \ref{thm:LagSol:EU} for the differential inclusion involving its Lagrangian representation \(F[X]\), and thereby infer that this solution also is a \textit{stable selection} with respect to initial data in an appropriate metric space, that is
  \begin{equation}\label{eq:TP2}
    \Tscr_2 \coloneqq \left\{ (\rho, v) \colon \rho \in \Pscr_2(\R), v \in \Lp{2}(\R,\rho) \right\}
  \end{equation}
  endowed with the metric
  \begin{equation}\label{eq:TP2met}
    d_{\Tscr_2}( (\rho_1, v_1), (\rho_2, v_2) )^2 \coloneqq d_{\Wscr_2}(\rho_1,\rho_2)^2 + d_{\Vscr_2}((\rho_1, v_1), (\rho_2, v_2))^2,
  \end{equation}
  involving the semi-distance
  \begin{equation*}
    d_{\Vscr_2}((\rho_1,v_1), (\rho_2,v_2))^2 \coloneqq \int_{\R \times \R} \abs{v_1(x)-v_2(y)}^2 \dee\varrho(x,y) = \int_\Omega \abs{v_1(X_{\rho_1}(m))-v_2(X_{\rho_2}(m))}^2\dee m,
  \end{equation*}
  where, recalling \eqref{eq:optplan}, \(\varrho\) is the unique optimal transport map between the measures \(\rho_1\) and \(\rho_2\).
  Here we may recall our assumptions \eqref{eq:init} on the initial data \((\rho_0, v_0)\) and observe that this aligns with the definition of \eqref{eq:TP2}.
  The monotone rearrangements of the measures are \(X_{\rho_1}\) and \(X_{\rho_2}\), and defining \(V_{\rho_i} = v_i \circ X_i\) for \(i=1,2\) we observe by \eqref{eq:optplan} that we in fact have the relation
  \begin{equation}\label{eq:TP2&L2}
    d_{\Tscr_2}( (\rho_1, v_1), (\rho_2, v_2) )^2 = \norm{X_{\rho_1}-X_{\rho_2}}_{\Lp{2}(\Omega)}^2 + \norm{V_{\rho_1}-V_{\rho_2}}_{\Lp{2}(\Omega)}^2.
  \end{equation}
  Next, in a similar vein to the Lagrangian setting, we define the notion of boundedness and continuity in this metric space.
  \begin{dfn}[Boundedness] \label{dfn:bnd:f}
    A map \(f \colon \Pscr_2(\R) \to \Mscr(\R) \) as in \ref{eq:f:RN} is bounded if there is a constant \( C \ge 0\) such that
    \begin{equation*}
      \norm{f_\rho}^2_{\Lp{2}(\R,\rho)} \le C \left( 1 + \int_\R \abs{x}^2\dee\rho \right) \text{ for all } \rho \in \Pscr_2(\R).
    \end{equation*}
    The map is called pointwise linearly bounded if there is \(C_\mathrm{p} \ge 0 \) such that
    \begin{equation*}
      \abs{f_\rho(x)} \le C_\mathrm{p} \left( 1 + \abs{x} + \int_\R \abs{y}\dee\rho(y) \right) \text{ for a.e.\ \(x \in \R\) and all \(\rho \in \Pscr_2(\R) \). }
    \end{equation*}
  \end{dfn}
  \begin{dfn}[Uniform continuity I] \label{dfn:unicont:f}
    A map \( f \colon \Pscr_2(\R) \to \Mscr(\R) \) as in \eqref{eq:f:RN} is uniformly continuous if there exists a modulus of continuity \( \omega\) such that
    \begin{equation}\label{eq:unicont:f}
      d_{\Tscr_2} ((\rho_1, f_{\rho_1}), (\rho_2,f_{\rho_2})) \le \omega\left( d_{\Wscr_2}(\rho_1, \rho_2) \right) \text{ for all } \rho_1, \rho_2 \in \Pscr_2(\R).
    \end{equation}
    As before, if \(\omega(r) = Lr\) for some \(L \ge 0\) and all \(r \ge 0\), then \(f\) is Lipschitz continuous. 
  \end{dfn}
  Recall from Section \ref{ss:convex} the correspondence between measures \(\rho \in \Pscr_2(\R)\) and optimal transport maps \(X \in \Lp{2}(\Omega)\), namely \( X \in \Kscr \) and \( X_\#\mathfrak{m} = \rho \).
  For such correspondences we want a map \( F \colon \Kscr \to \Lp{2}(\Omega) \) satisfying \eqref{eq:f&F}, and a possible choice is simply \( F[X] = f_\rho \circ X \) for \(f_\rho\) given by \eqref{eq:f:RN}.
  Then the boundedness and continuity properties for the force distribution \(f\) in Definitions \ref{dfn:bnd:f} and \ref{dfn:unicont:f} translates into the corresponding properties for \(F\) in Definitions \ref{dfn:bnd:F} and \ref{dfn:unicont:F}.
  However, different choices for \(F\) can be convenient, and for this one can make use of another definition.
  
  \begin{dfn}[Uniform continuity II]
    The map \( f \colon \Pscr_2(\R) \to \Mscr(\R) \) is \textit{densely} uniformly continuous if \eqref{eq:unicont:f} holds for measures that are uniformly continuous with respect to the Lebesgue measure with bounded densities. Dense Lipschitz continuity is defined analogously.
  \end{dfn}
  
  By \cite[Lemma 6.4]{brenier2013sticky}, if \( f \colon \Pscr_2(\R) \to \Mscr(\R) \) is densely uniformly continuous, there exists a unique uniformly continuous map \( F \colon \Kscr \to \Lp{2}(\Omega) \) such that \eqref{eq:f&F} holds for all \((X, \rho)\) satisfying \( X \in \Kscr \) and \( X_\#\mathfrak{m} = \rho \).
  Furthermore, we can then define a densely uniformly continuous force distribution \(f\) to be \textit{sticking} if its corresponding map \(F\) described above satisfies Definition \ref{dfn:sticking}.
  
  Having established these definitions, \cite[Theorem 6.6]{brenier2013sticky} tells us that if \( f \colon \Pscr_2(\R) \to \Mscr(\R) \) is densely Lipschitz continuous, there exists a stable selection of a solution \((\rho, v)\) of \eqref{eq:EF} with respect to initial data \( (\rho_0, v_0) \) in \( (\Tscr_2, d_{\Tscr_2}) \).
  
  At the end of Section 6 in \cite{brenier2013sticky} the authors give several examples of force functionals covered by their theory.
  Below we see how two of these connect our examples in Section \ref{s:examples}.
  
  \subsubsection{Euler--Poisson system}
  Theorem \ref{thm:LagSol:EU} can be applied to \eqref{eq:EP} where according to \cite[Example 6.9]{brenier2013sticky}, compare also with \cite[Proposition 2.10]{bonaschi2015equivalence}, we have \(F[X](m) = -\alpha(m-\frac12) -\beta \).
  That is, the force operator \(F[X]\) does not even depend on \(X\), meaning it is Lipschitz continuous in the sense of Definition \ref{dfn:unicont:F}.
  Furthermore, this means we can compute the prescribed velocity \(U\) in \eqref{eq:U} for all times directly:
  \begin{equation}\label{eq:U:EP}
    U(t,m) = V^0(m) + t A(m), \qquad A(m) \coloneqq -\alpha\left(m -\frac12\right) - \beta.
  \end{equation}
  In this case, the force operator \(f[\rho]\) turns out to be pointwise linearly bounded and densely uniformly continuous.
  Moreover, for \( \alpha \ge 0\) it is also sticking, meaning the unique Lagrangian solution is given by the projection formula of Proposition \ref{prp:projFormula}.
  We study the non-sticking case in Section \ref{ss:BGSW}.
  
  \subsubsection{Euler--Poisson system with quadratic confinement}
  A more general force distribution compared to the above is given in \cite[Example 6.10]{brenier2013sticky}.
  Given \( \sigma \in \Lp{\infty}(\R) \) they define \( f[\rho] = - \rho \del_x q_\sigma \) for all \( \rho \in \Pscr_2(\R) \),
  where \( q_\sigma \) satisfies \( -\del_{xx} q_\sigma = \lambda (\rho -\sigma) \).
  Of course, this only defines \( \del_x q_\sigma\) up to a constant, and making the same choice for the \(\rho\)-part as the previous example and defining
  \( S_\sigma(x) = \int_0^x\sigma(y)\dee y \) one arrives at the following Lagrangian representation
  \begin{equation*}
    F[X](m) = -\lambda \left( m-\frac12 - S_\sigma(X(m)) \right) \quad \text{for all } m \in \Omega.
  \end{equation*}
  Then \(F \colon \Kscr \to \Lp{2}(\Omega) \) is pointwise linearly bounded and Lipschitz-continuous.
  Choosing \( \sigma \equiv 1 \) such that \( S_\sigma(x) = x \) and \(\lambda > 0\) this corresponds to a repulsive Poisson force combined with a confining force driving mass towards the origin.
  In Section \ref{ss:CCZ} we will consider a variant of this studied in \cite{carrillo2016pressureless} where the confining force comes from a quadratic confinement potential which drives mass toward the center of mass \(\bar{x}(t)\).
  In the above situation this behavior can be achieved by choosing \(S_\sigma(t,x) = x-\bar{x}(t)\).
  
  \section{Entropy solutions}\label{s:entropy}
  In this section, following the terminology used for the first order interaction equation \eqref{eq:interact} in \cite{bonaschi2015equivalence}, we will introduce the concept of entropy solutions for the Euler system \eqref{eq:EF}.
  This was already done for the specific case of the Euler--Poisson system \eqref{eq:EP} by Nguyen and Tudorascu \cite{nguyen2008pressureless,nguyen2015one}, which in turn is based on the work of Brenier and Grenier \cite{brenier1998sticky} on the Euler system without forcing.
  The main idea is to approximate the density \(\rho\) with an empirical measure, and noting that its piecewise constant distribution function satisfies a conservation law with a piecewise linear flux function, much in the spirit of front tracking \cite{dafermos1972polygonal}.
  Then, applying the calculus for functions of bounded variation (\(BV\)) developed by Vol'pert \cite{volpert1967spaces}, it turns out that the corresponding solution of the conservation law gives rise to solutions of the Euler system, and a solution of the original problem is obtained from a limiting argument.
  This idea can be regarded as using the flow along characteristics to combine \eqref{eq:EF:den} and \eqref{eq:EF:mom} into a single scalar conservation law for an appropriate flux function, analogous to what was done in \cite{leslie2021sticky}, and this is how we motivate the flux function below.
  Under our assumptions, the flux function will in general be time-dependent and not Lipschitz-continuous in the solution variable.
  Therefore, as in \cite{nguyen2015one}, we will rely on results developed by Golovaty and Nguyen \cite{golovaty2012existence}, rather than the standard existence and uniqueness theory for conservation laws which typically requires the flux function to be Lipschitz.
  
  \subsection{Derivation of the scalar conservation law}
  Here we will derive a scalar conservation law with the aim of finding solutions of \eqref{eq:EF}.
  To this end we introduce the right-continuous distribution functions
  \begin{equation}\label{eq:MQ:def}
    M(t,x) = \int_{(-\infty,x]}\dee\rho(t,y), \qquad Q(t,x) = \int_{(-\infty,x]} v(t,y)\dee\rho(t,y),
  \end{equation}
  where we have used the dummy variable \(y\) to indicate integration over the spatial variable.
  In order to clarify the notation in the rest of the paper, we emphasize that integration in space with respect to a time-dependent measure \(\rho = \rho(t,x)\) is indicated by the differential \(\dee\rho(t,\cdot)\), where we here have omitted the variable of integration.
  If there is additional integration in time, there will be a separate, additional differential for this, e.g., \(\dee t\).
  We write \(M(0,x) = M^0(x)\) and \(Q(0,x) = Q^0(x)\) for the initial data of \eqref{eq:MQ:def}, and we recall from our assumptions \eqref{eq:init} on \(\rho_0\) and \(v_0\) that these distribution functions are initially well-defined functions of bounded variation.
  Then, integrating \eqref{eq:EF:den} and \eqref{eq:EF:mom} together with \eqref{eq:f:RN}, \(f[\rho] = f_\rho \rho\), yields two advection-type equations
  \begin{subequations}\label{eq:MQ:evol}
    \begin{align}
      \del_t M + v \del_xM &= 0, \label{eq:M:evol} \\
      \del_t Q + v \del_x Q &= \int_{(-\infty,x]}f_\rho(t,y) \dee\rho(t,y). \label{eq:Q:evol}
    \end{align}
  \end{subequations}
  Now the goal is to rewrite \eqref{eq:M:evol} as a conservation law for \(M\) by using the relation \(\del_x Q = v \del_x M\) coming from \eqref{eq:MQ:def}, and expressing \(Q\) as a function of \(M\) using the evolution equation \eqref{eq:Q:evol}. 
  From \eqref{eq:M:evol} we see that \(M\) is constant along characteristics traveling with velocity \(v\), which we want to take advantage of in the second equation.
  One way of deriving the conservation laws corresponding to the Euler (\(f[\rho] \equiv 0\)) and the Euler--Poisson system in  \cite{brenier1998sticky, nguyen2008pressureless} is then to introduce characteristics \(\eta(t,x)\) satisfying
  \begin{equation*}
    \del_t \eta(t,x) = v(t,\eta(t,x)), \qquad \eta(0,x) = x
  \end{equation*} 
  as in \cite{leslie2021sticky}, cf.\ Example \ref{exm:flux:EP} below.
  However, in order to highlight the connection to the Lagrangian solutions of Section \ref{s:gradflow}, we will instead use the optimal transport map \(X\) defined
  through the push-forward relation \(\rho = X_\#\mathfrak{m}\).
  We recall that \(X(t,m)\), with \(X(0,m) = X^0(m)\), is the right-continuous generalized inverse \(M^{-1}(t,m)\) of \(M(t,x)\).
  Then, if \(\rho(t,\cdot)\) contains no atoms, we have \(M(t,X(t,m)) = m\), which together with \eqref{eq:M:evol} implies that \(\dot{X}(t,m) = v(t,X(t,m))\).
  Assuming smooth, i.e., atomless, solutions, we can recast \eqref{eq:Q:evol} as
  \begin{equation*}
    \del_t Q(t,X(t,m)) + v(t,X(t,m)) \del_x Q(t,X(t,m)) = \int_\R \chi_{(-\infty,0]}(y-X(t,m))f_\rho\dee\rho(t,y).
  \end{equation*}
  Using the push-forward relation, this is equivalent to
  \begin{equation*}
    \diff{}{t} Q(t,X(t,m)) = \int_{0}^m F[X](t,\omega)\dee\omega,
  \end{equation*}
  where \(F\) corresponds to the Lagrangian representation \eqref{eq:f&F} of the forcing term from Section \ref{s:gradflow}, and we recall that one possible choice for this is \(F[X] = f_\rho \circ X\).
  Integrating this equation in time and interchanging the order of the integrals leads to
  \begin{equation*}
    Q(t,X(t,m)) = Q(0,X(0,m)) + \int_0^m \int_0^t F[X](s,\omega)\dee s \dee \omega .
  \end{equation*}
  Plugging in for \(m = M(t,x)\) and taking into account the initial conditions from \eqref{eq:MQ:def}, we arrive at
  \begin{equation}\label{eq:Q&U}
    \begin{aligned}
      Q(t,x) &= Q^0(X^0(M(t,x))) + \int_0^{M(t,x)} \int_0^t F[X](s,m)\dee s \dee m \\
      &= \int_0^{M(t,x)} \left( v_0(X(0,m)) + \int_0^t F[X](s,m)\dee s \right) \dee m \eqqcolon \Ucal(t,M(t,x)),
    \end{aligned}
  \end{equation}
  and we have, rather formally, expressed \(Q(t,x)\) as a function of \(M(t,x)\).
  Using \eqref{eq:Q&U} to define the flux function \(\Ucal(t,m)\), we express \eqref{eq:M:evol} as the scalar conservation law
  \begin{equation}\label{eq:claw}
    \del_t M + \del_x \Ucal(t,M) = 0.
  \end{equation}
  Note that the flux function \(\Ucal\) is exactly the primitive of the prescribed velocity \(U\) in \eqref{eq:U} from Section \ref{s:gradflow}.
  Differentiating \(\Ucal(t,M)\) with respect to \(x\) we get a Radon measure which is absolutely continuous with respect to \(\rho(t,x) = \del_x M(t,x)\), and the Radon--Nikodym derivative of \(\del_x Q\) with respect to \(\del_x M\) corresponds to the velocity of the characteristics for the conservation law \eqref{eq:claw}, which shows a formal connection between the two solution concepts.
  If \(\rho\) contains atoms, then \(M\) contains jumps, and for any sufficiently regular flux function we can apply a \(BV\)-chain rule as developed in \cite{volpert1967spaces}.
  Note that in general, compositions of \(BV\) functions are not necessarily of bounded variation, cf.\ \cite{josephy1981composing}, and in order to have a chain rule for general \(BV\) functions we would need the outer function to be at least Lipschitz.
  However, in our one-dimensional setting where \(M\) is the distribution function of a probability measure, we can require less regularity.
  The following result, found in \cite[Lemma 2.1]{nguyen2008pressureless}, will be of use.
  \begin{lem}\label{lem:decomp}
    Let \(\mu\) be a Borel probability measure on \(\R\) and let \(M\) be its right-continuous distribution function.
    Write
    \begin{equation*}
      \mu = \sum_{j \in \Jcal} m_j \delta_{x_j} + \rho,
    \end{equation*}
    where \(\{x_j\}_{j\in\Jcal}\) is the set of (at most countable) discontinuities of \(M\), with \(m_j \coloneqq \mu({x_j})\).
    If \(\rho\) is nonzero, then we have
    \begin{equation}
      M_\#\rho = \chi_{J^c} \mathfrak{m} \quad \text{for} \quad J \coloneqq \bigcup_{j\in\Jcal} \left( M(x_j-), M(x_j) \right).
    \end{equation}
  \end{lem}
  Now we can state the generalized \(BV\)-chain rule, comparable to the results \cite[Theorem 2.2, Corollary 2.3]{nguyen2008pressureless} which were proved with approximation arguments.
  \begin{lem}\label{lem:BVcalc}
    Consider \(p\in [1,\infty]\) and a function \(f \in \Lp{p}(0,1)\), for which we define \(F(m) \coloneqq \int_0^m f(\omega)\dee \omega\).
    If \(M\) is the right-continuous distribution function of some Borel probability measure \(\mu\) on \(\R\), then \(F \circ M \in BV(\R)\) with distributional derivative \(g \mu\), where
    \begin{equation}\label{eq:BVchain}
      g(x) \coloneqq \begin{cases}
        f \circ M(x), & \mu(\{x\}) = 0, \\
        \frac{F\circ M(x)-F\circ M(x-)}{\mu(\{x\})}, & \mu(\{x\}) \neq 0
      \end{cases}
    \end{equation}
    in the \(\mu\)-a.e.\ sense. Furthermore \(g \in \Lp{p}(\R,\mu)\) with \(\norm{g}_{\Lp{p}(\R,\mu)} \le \norm{f}_{\Lp{p}(0,1)}\), where the inequality is replaced by equality if \(\mu\) is diffuse, i.e., there are no \(\mu\)-atoms.
  \end{lem}
  
  Note that the function \(g\) in \eqref{eq:BVchain}, which in fact is the Radon--Nikodym derivative of the (signed) Radon measure \(\del_xF(M)\) with respect to the Radon measure \(\del_x M\), can be equivalently written as 
  \begin{equation}\label{eq:VolpertAvg}
    g(x) = \int_0^1 f((1-s)M(x-) + sM(x))\dee s,
  \end{equation}
  a representation due to Vol'pert \cite{volpert1967spaces} for compositions of general Lipschitz functions and \(BV\) functions.
  
  \begin{proof}[Proof of Lemma \ref{lem:BVcalc}]
    Note that \(f\colon\R\to\R\) is a function of bounded variation if and only if we for any partition \(\{x_i\}\) of \(\R\) have
    \begin{equation*}
      \TV(f) \coloneqq \sup_{x_i} \sum_{i} \abs{f(x_i)-f(x_{i-1})} < \infty.
    \end{equation*}
    Since \(M\) is a monotone increasing function, any partition \(\{x_i\}_i\) of \(\R\) defines a partition \(\{M(x_i)\}_i\) of \([0,1]\), and the set of such partitions of \([0,1]\) generated by partitions of \(\R\) is strictly smaller than the set of all possible partitions of \([0,1]\) since there is no way to pick \(x_i\) such that \(M(x_i) \in J\), the `jump set' of \(M\).
    Indeed, the range of \(F\circ M\) is strictly included in the range of \(F\).
    Hence
    \begin{equation*}
      \sup_{x_i} \sum_{i} \abs{F(M(x_i))-F(M(x_{i-1}))} \le \sup_{m_i} \sum_{i} \abs{F(m_i)-F(m_{i-1})},
    \end{equation*}
    from which we observe it to be sufficient if \(F \in BV([0,1])\).
    Now, since \(F \in W^{1,p}(0,1)\), it is absolutely continuous and consequently of bounded variation, and so \(F \circ M \in BV(\R)\).
    
    A (signed) Radon measure on \(\R\) is uniquely determined by its distribution function, and we note that \(M\) can be expressed as
    \begin{equation*}
      M(x) = \int_{(-\infty,x]}\dee \rho + \sum_{j \in \Jcal} m_j \chi_{(-\infty,x]}(x_j).
    \end{equation*}
    To verify \eqref{eq:BVchain} we shall similarly rewrite \(F\circ M\) in a form from which it is clear that it is the distribution function of the signed measure \(g \mu \ll \mu\) with \(g\) as in \eqref{eq:BVchain}.
    Indeed, by Lemma \ref{lem:decomp} we have    
    \begin{align*}
      \int_0^{M(x)}f(m)\dee m &= \int_0^1 \chi_{(0,M(x)]} \chi_{J^c} f(m) \dee m + \int_0^1 \chi_{(0,M(x)]} \chi_{J} f(m) \dee m \\
      &= \int_\R \chi_{(0,M(x)]}(M(y)) f(M(y))\dee\rho(y) + \sum_{j\in\Jcal} \chi_{(-\infty,x]}(x_j) \int_{M(x_j-)}^{M(x_j)}f(m)\dee m \\
      &= \int_{(-\infty,x]} f\circ M \dee\rho + \sum_{j\in\Jcal} \frac{F(M(x_j))-F(M(x_j-))}{m_j} m_j \chi_{(-\infty,x]}(x_j),
    \end{align*}
    from which the result follows.
    The last statement for \(p \in [1,\infty)\) can be proved using Lemma \ref{lem:decomp} and Jensen's inequality, namely
    \begin{align*}
      \norm{g}_{\Lp{p}(\R,\mu)}^p &= \int_\R \abs{g(x)}^p\dee\mu = \int_\R \abs{f\circ M(x)}^p\dee\rho + \sum_{j \in \Jcal} m_j \left|\frac{1}{m_j}\int_{M(x_j-)}^{M(x_j)}f(m)\dee m\right|^p \\
      &\le \int_{J^c} \abs{f(m)}^p\dee m + \int_{J} \abs{f(m)}^p\dee m = \norm{f}_{\Lp{p}(0,1)}^p.
    \end{align*}
    The case of \(p = \infty\), i.e., \(F\) Lipschitz, follows from a similar decomposition.
  \end{proof}
  
  We see from \eqref{eq:Q&U} that the integral term in the velocity \(U\) at time \(t\) involves the generalized inverse \(X(s,m)\) for \(s \in [0,t]\).
  As it turns out, in some cases one can use \textit{a priori} knowledge about the solution to express \(\Ucal(t,M)\) in a more explicit form.
  Below we give some examples of such cases, and the first of these corresponds to the Euler--Poisson system \eqref{eq:EP}.
  
  \begin{exm}[Poisson force]\label{exm:flux:EP}
    From the assumption \(\rho_0 \in \Pscr(\R)\) and the mass conservation due to \eqref{eq:EP:den} we have that a solution of \eqref{eq:EP} must satisfy \(\rho \in \Pscr(\R)\).
    This allows us to rewrite \eqref{eq:EP:mom} as
    \begin{align*}
      (\rho v)_t + (\rho v^2)_x 
      &= -\frac{\alpha}{2} \left( \int_{(-\infty, x)} \rho(t,y)\dee y +  \int_{(-\infty,x]} \rho(t,y)\dee y - 1 \right) \rho - \beta \rho \\
      &= -\frac{\alpha}{2}(M(t,x-) + M(t,x)-1) \rho - \beta \rho = -\del_x \left(\frac{\alpha}{2}(M(t,x)^2-M(t,x)) + \beta M(t,x)\right),
    \end{align*}
    where in the last identity we have used \eqref{eq:BVchain}.
    Integrating \eqref{eq:EP} we arrive at the nonconservative, transport-like system
    \begin{equation}\label{eq:MQ:evol:EP}
      \del_t M + v \del_xM = 0, \qquad \del_t Q + v \del_x Q = \frac{\alpha}{2} \left( M^2 - M \right) + \beta M.
    \end{equation}
    Since \(\del_x Q = v \del_x M\) and \(M\) is constant along characteristics, we see that if we can find a flux function \(\Vcal^0 \colon [0,1] \to \R\) satisfying the compatibility condition \(\Vcal^0(M^0) = Q^0\), then the transport equations above can be combined into a single conservation law, as shown in \cite{nguyen2008pressureless}.
    The appropriate function \(\Vcal^0\), depending only on the initial data of \eqref{eq:EP}, is given by
    \begin{equation*}
      \Vcal^0(m) = \int_0^m v_0((M^0)^{-1}(\omega))\dee \omega,
    \end{equation*}
    where \((M^0)^{-1}(m)\) is the right-continuous generalized inverse of \(M^0(x)\).
    Note that the integrand is exactly \(V^0(m) \coloneqq v_0 \circ X^0\), that is, \(\Vcal^0\) is the primitive of \(V^0\).
    In the end we obtain the conservation law \eqref{eq:claw} with flux function
    \begin{equation}\label{eq:flux:EP}
      \Ucal(t,m) = \Vcal^0(m) + t \Acal(m), \qquad \Acal(m) = -\frac{\alpha}{2}(m^2-m) - \beta m.
    \end{equation}
    Observe that the flux function \(\Ucal(t,m)\) is the primitive of the prescribed velocity \(U(t,m)\) from \eqref{eq:U:EP}.
  \end{exm}
  
  \begin{exm}[Poisson force with linear damping]\label{exm:flux:EPld}
    We can consider a slightly more general system than \eqref{eq:EP} by including a linear damping term for the momentum, that is, we look at
    \begin{equation} \label{eq:EPld}
      \begin{aligned}
        \del_t \rho + \del_x (\rho v) &= 0, \\
        \del_t (\rho v) + \del_x (\rho v^2) &= -\left(\tfrac{\alpha}{2} (\sgn\ast\rho) + \beta \right) \rho - \gamma \rho v.
      \end{aligned}
    \end{equation}
    Here \(\gamma \in \R\) should satisfy \(\gamma > 0\) to be considered a damping term.
    As before, the idea is to integrate \eqref{eq:EPld} and consider the distribution functions \eqref{eq:MQ:def} of \(\rho\) and \(\rho v\).
    Analogous to \eqref{eq:MQ:evol:EP} we then arrive at the transport-like equations
    \begin{equation*}
      \del_t M + v \del_x M = 0, \qquad \del_t Q + v \del_x Q = -\left( \frac{\alpha}{2}(M^2-M) + \beta M + \gamma Q \right).
    \end{equation*}
    The second equation can be rewritten, using an integrating factor, as
    \begin{equation*}
      \del_t (\e^{\gamma t} Q) + v \del_x(\e^{\gamma t} Q) = -\e^{\gamma t}\left( \frac{\alpha}{2}(M^2-M) + \beta M\right),
    \end{equation*}
    which we can integrate along characteristics, where \(M\) is constant.
    As before, choosing \(\Vcal^0\) such that \(Q(0,x) = \Vcal^0(M(0,x))\) we find that \(M\) satisfies the conservation law \eqref{eq:claw}, this time with a time-dependent flux function \(\Ucal(t,m)\) given by
    \begin{equation}\label{eq:flux:EPld}
      \Ucal(t,m) = \e^{-\gamma t} \Vcal^0(m) + \frac{1-\e^{-\gamma t}}{\gamma} \Acal(m) = \Vcal^0(m) -\frac{1-\e^{-\gamma t}}{\gamma}(\gamma \Vcal^0(m)-\Acal(m)),
    \end{equation} 
    with \(\Acal\) as defined in \eqref{eq:flux:EP}.
    Note that for \(\gamma = 0\) we have to interpret the fraction above as a limit \(\gamma \to 0\), for which we recover \eqref{eq:flux:EP}.
    We observe that the damping reduces the time it takes for the forcing term to dominate the flux function.
    
    Note that \eqref{eq:EPld} is not directly covered by \eqref{eq:EF}, as the forcing term now depends on the velocity in addition to the density.
    However, this can still be seen as a Lipschitz perturbation of a gradient flow, as argued in \cite{brenier2013sticky}, see Section \ref{ss:CCZ}.
  \end{exm}
  
  \subsection{Entropy solutions for the scalar conservation law}
  We now consider existence and uniqueness of weak solutions for \eqref{eq:claw}.
  Introducing a convex, Lipschitz entropy function \(\eta \colon [0,1] \to \R\) with corresponding entropy-flux function \(q \colon \R\times[0,1] \to \R\) satisfying \(\del_m q(t,m) = \eta'(m) \del_m \Ucal(t,m)\), one typically requires an entropy solution to satisfy
  \begin{equation*}
    \del_t \eta(M) + \del_x q(t,M) \le 0
  \end{equation*}
  in the distributional sense, cf.\ \cite[Section 4.5.1]{dafermos2016hyperbolic}.
  That is, considering the associated Cauchy problem for \eqref{eq:claw}, for any \(\phiv \in C_{\mathrm{c}}^\infty([0,T)\times\R)\) with \(\phiv \ge 0\) we have
  \begin{equation}\label{eq:Kruzkov}
    \int_{0}^{T} \int_{\R} \left[ \eta(M) \phiv_t + q(t,M)\phiv_x  \right]\dee x\dee t + \int_{\R} \phiv(0,x) \eta(M^0(x)) \dee x \ge 0.
  \end{equation}
  Since such convex functions can be seen as limits of piecewise affine functions, this is equivalent to \eqref{eq:Kruzkov} being true for the Kru\v{z}kov entropy-entropy flux pair
  \begin{equation}\label{eq:KruzkovPair}
    \eta_k(m) = \abs{m-k}, \qquad q_k(t,m) = \sgn(m-k) \left(\Ucal(t,m)-\Ucal(t,k)\right) = \int_k^m \sgn(\omega-k) U(t,\omega)\dee \omega,
  \end{equation}
  for all \(k \in \R\), see, e.g., \cite{holden2015front}.
  In fact, it is sufficient to consider \(k\) in a dense set of \(\R\), cf.\ \cite[Proposition 2]{panov2002generalized} or \cite[Section 3]{golovaty2012existence}.
  
  In Example \ref{exm:flux:EP} we saw that the flux function \(\Ucal(t,M)\) corresponding to the Euler--Poisson system takes the form
  \begin{equation*}
    \Ucal(t,m) = \int_0^m V^0(\omega)\dee \omega + t\int_0^m A(\omega)\dee \omega,
  \end{equation*}
  where \(A(m) = -\alpha(m-\frac12) -\beta\) and \(V^0 = v_0 \circ X^0 \in \Lp{2}(0,1)\) due to \(v_0 \in \Lp{2}(\R,\rho_0)\).
  In particular, we see that \(m \mapsto \Ucal(t,m) \in W^{1,2}(0,1)\), which means that the flux function is in general not Lipschitz-continuous.
  This contrasts the case studied in \cite{brenier1998sticky}, where they assume \(v_0 \in \Lp{\infty}(\R,\rho_0)\).
  Then their flux function \(\Ucal(t,M) \equiv \Vcal^0(M)\) is Lipschitz and one can employ the standard assumptions used in the classical version of Kru\v{z}kov's doubling of variables-argument, cf.\ \cite{bressan2000hyperbolic}, to prove uniqueness for \eqref{eq:claw}.
  Another issue with the flux function \eqref{eq:flux:EP} may be degenerate in the sense that \(m \mapsto \Ucal(t,m)\) is affine on non-degenerate intervals, and this can cause issues in existence proofs for \eqref{eq:claw}.
  
  The above issues for the Euler--Poisson flux \eqref{eq:flux:EP} motivated the generalizations of existence and uniqueness proofs found in \cite{golovaty2012existence}, where they consider scalar conservation laws on the line where the flux function is merely continuous in \(m\) and may in general depend on both time \(t\) and position \(x\).
  These results were then used in \cite{nguyen2015one} to relax the assumptions used to study the Euler--Poisson system \eqref{eq:EP} in \cite{nguyen2008pressureless}.
  We follow \cite[Definition 3.1]{nguyen2015one} in defining entropy solutions of \eqref{eq:claw}.
  \begin{dfn}[Entropy solution of \eqref{eq:claw}]\label{dfn:entropy}
    Let \(M^0 \in \Lp{1}_\mathrm{loc}(\R)\).
    A function \(M \in \Lp{\infty}\left((0,T)\times\R\right) \cap C\left([0,T];\Lp{1}_\mathrm{loc}(\R)\right)\) is an entropy solution of \eqref{eq:claw} if
    \begin{equation}\label{eq:Kruzkov2}
      \int_{0}^{T} \int_{\R} \left( \abs{M-k} \phiv_t + \sgn(M-k)(\Ucal(t,M)-\Ucal(t,k))\phiv_x  \right)\dee x\dee t + \int_\R \abs{M^0(x)-k}\phiv(0,x)\dee x \ge 0
    \end{equation}
    for all \(k\) in a dense set of \(\R\) and nonnegative test functions \(\phiv \in C^\infty_{\mathrm{c}}([0,T)\times\R)\).
  \end{dfn}
  As usual, replacing the entropy-entropy flux pair with the identity \(\Id\) and \(\Ucal(t,M)\) and removing the nonnegativity requirement on \(\phiv\) in \eqref{eq:Kruzkov2}, we recover the definition of a weak solution, i.e.,
  \begin{equation}\label{eq:weak}
    \int_{0}^{T} \int_{\R} \left( M \phiv_t + \Ucal(t,M) \phiv_x  \right)\dee x\dee t + \int_\R M^0(x)\phiv(0,x)\dee x = 0.
  \end{equation}
  
  \subsubsection{The Rankine--Hugoniot- and Ole\u{\i}nik E-condition}
  For piecewise smooth solutions of \eqref{eq:claw}, the requirements for weak and entropy solutions can be reduced to conditions on the discontinuities, cf.\ \cite[Chapter 2]{holden2015front}.
  To this end, assume \(M(t,x)\) takes the values \(M_l(t)\) and \(M_r(t)\) to the left and right of some shock curve \(\Gamma = \{(t,x) \colon x = s(t)\}\).
  Denoting by \([[M]]\) the jump in \(M\) across the curve \(\Gamma\), and correspondingly \([[\eta(M)]]\), \([[q(t,M)]]\) for the other quantities, this leads to the inequality
  \begin{equation*}
    \int_{\Gamma} \phiv(t,x) \left([[\eta(M)]]\dot{s}(t) - [[q(t,M)]] \right) \ge 0,
  \end{equation*}
  which gives us \(\dot{s}(t)[[\eta(M)]] \ge [[q(t,M)]]\).
  For \(\eta(M) = M\) and \(q(t,M) = \Ucal(t,M)\) the above is an equality and we get the \textit{Rankine--Hugoniot condition}
  \begin{equation}\label{eq:RH}
    \dot{s}(t) = \frac{[[\Ucal(t,M)]]}{[[M]]}.
  \end{equation}
  Using the Kru\v{z}kov entropy-flux pair, the inequality \eqref{eq:Kruzkov2} is equivalent to
  \begin{equation}\label{eq:Oleinik}
    \frac{\Ucal(t,M_r)-\Ucal(t,k)}{M_r-k} \le \dot{s}(t) \le \frac{\Ucal(t,k)-\Ucal(t,M_l)}{k-M_l}
  \end{equation}
  for any \(M_l < k < M_r\), which is known as the \textit{Ole\u{\i}nik E-condition}, cf.\ \cite{oleinik1959uniqueness,dafermos2016hyperbolic}.
  In fact, for the distributional entropy inequality to reduce to the Ole\u{\i}nik E-condition, it is enough for the weak solution to be of class \(BV_\mathrm{loc}\), see \cite[Section 4.5]{dafermos2016hyperbolic}, which is exactly what we are interested in.
  
  \subsubsection{Existence and uniqueness results}
  To establish existence and uniqueness of entropy solutions of \eqref{eq:claw} we have to make some assumptions on its flux function.
  However, as it stands, the formal expression for \(\Ucal(t,m)\) in \eqref{eq:Q&U} has an apparent intricate dependence on the generalized inverse of the solution \(M\) up until time \(t\); that is, for \(F[X] = f_\rho \circ M^{-1}\) we have
  \begin{equation*}
    \Ucal(t,m) = \int_0^m \left( (v_0 \circ M_0^{-1})(\omega) + \int_0^t f_\rho(M^{-1}(s,\omega))\dee s\right)\dee\omega.
  \end{equation*}
  On the other hand, in cases such as Examples \ref{exm:flux:EP} and \ref{exm:flux:EPld} the part involving the forcing term \(f_\rho\) simplifies, and we are led to consider flux functions of the form
  \begin{equation}\label{eq:flux}
    \Ucal(t,m) = \Vcal^0(m) + \sum_{i=1}^n \int_0^t b(s)\dee s \int_0^m F_i(\omega)\dee\omega \eqqcolon \sum_{i=0}^n B_i(t) \Fcal_i(m)
  \end{equation}
  for \(\Vcal^0 = v_0 \circ M^{-1}_0 \in \Lp{2}(0,1)\), \(F_i \in \Lp{p}(0,1)\) and \(b_i \in \Lp{1}(0,T)\).
  Indeed, since \(\Ucal(t,m)\) in \eqref{eq:Q&U} can be thought of as the primitive of the prescribed velocity \(U\) defined in \eqref{eq:U}, it is natural to expect \(U\), and then also \(\Ucal\), to be a function of \(t\) for fixed \(m \in \Omega\).
  The functions \(\Fcal_i\) in \eqref{eq:flux} are continuous, while \(B_i \in \Lp{\infty}(0,T) \subset \Lp{2}(0,T)\), which means that our problem is covered by the theory in \cite{golovaty2012existence}.
  In particular we have the following crucial \(\Lp{1}\)-contraction principle \cite[Theorem 2.2]{golovaty2012existence}, the proof of which relies on establishing a Kato-type inequality.
  Here, for \(c \in \R\), its positive part is written \(c^+ = \max\{c,0\}\).
  \begin{thm}[\(\Lp{1}\)-contraction]\label{thm:L1ctr}
    Assume \(M\) and \(\tilde{M}\) are entropy solutions of \eqref{eq:claw} with flux function of the form \eqref{eq:flux} and initial data \(M^0, \tilde{M}^0 \in \Lp{\infty}(\R)\). Then
    \begin{equation}
      \int_{\R} (M(t,x)-\tilde{M}(t,x))^+\dee x \le \int_{\R} (M^0(x)-\tilde{M}^0(x))^+\dee x.
    \end{equation}
  \end{thm}
  Several important results can be deduced from Theorem \ref{thm:L1ctr}, cf.\ Corollaries 2.8 and 2.9 in \cite{golovaty2012existence}.
  For instance, since the constant functions 0 and 1 are entropy solutions, we see that if \(M\) is an entropy solution with initial data \(M^0(x) \in [0,1]\), then we must have \(M(t,x) \in [0,1]\) for a.e.\ \(x\).
  Moreover, we would like to use it to show that the distributional derivatives of entropy solutions are probability measures, and to this end we introduce some useful function spaces as in \cite[Definition 3.2]{nguyen2015one}.
  \begin{dfn}\label{dfn:Rset}
    A function \(M \colon \R \to \R\) belongs to the set \(\Rscr\) if it is nondecreasing, right-continuous and has limits \(0\) and \(1\) at respectively \(-\infty\) and \(+\infty\).
    Moreover, \(M \in \Rscr\) belongs to \(\Rscr_p\) for \(p \ge 1\) if its distributional derivative \(\del_x M\) has finite \(p\)-moment, that is, \(\int_\R \abs{x}^p \del_x M\) is finite.
  \end{dfn}
  We note that the distributional derivative \(\del_x M\) of any function \(M \in \Rscr\) is a Borel probability measure on \(\R\) with \(M(x) = \del_x M((-\infty,x])\).
  The map \(M \mapsto \del_x M\) is injective from \(\Rscr\) (\(\Rscr_p\)) to \(\Pscr(\R)\) (\(\Pscr_p(\R)\)).
  The space \((\Pscr_p(\R), d_{\Wscr_p})\) is a complete metric space equipped with the \(p\)-Wasserstein metric \(d_{\Wscr_p}\).
  In view of the above discussion, \(\Rscr_p\) from Definition \ref{dfn:Rset} is a metric space with the induced metric
  \begin{equation}\label{eq:metricEquiv}
    d_{\Rscr_p}(M,\tilde{M}) \coloneqq d_{\Wscr_p}(\del_x M, \del_x \tilde{M}) = \left(\int_\Omega \abs*{X(m)-\tilde{X}(m)}^p\dee m\right)^{1/p},
  \end{equation}
  where \(X\) and \(\tilde{X}\) are the respective right-continuous generalized inverses of \(M\) and \(\tilde{M}\).
  Moreover, as a consequence of Fubini's theorem we have
  \begin{equation}\label{eq:metricEquiv1}
    d_{\Wscr_1}(\del_x M, \del_x \tilde{M}) = \norm{X-\tilde{X}}_{\Lp{1}(0,1)} = \norm{M-\tilde{M}}_{\Lp{1}(\R)},
  \end{equation}
  cf.\ \cite[Proposition 2.17]{santambrogio2015optimal}, \cite[Remarks 2.19]{villani2003topics}.
  The following important corollary of Theorem \ref{thm:L1ctr} guarantees that entropy solutions of \eqref{eq:claw} have a representative which remains in the set \(\Rscr\).
  \begin{cor}\label{cor:remainInR}
    Let \(M\) be a entropy solution of \eqref{eq:claw} with flux function \eqref{eq:flux} and initial data \(M^0 \in \Rscr\). Then there exists a unique solution \(\tilde{M} \in C([0,T]; \Lp{1}_\mathrm{loc}(\R))\) such that \(M(t,\cdot) \in \Rscr\) and \(\tilde{M}(t,\cdot) = M(t,\cdot)\) as functions in \(\Lp{1}_\mathrm{loc}(\R)\) for \(t \in [0,T]\).
  \end{cor}
  The proof of this result is identical to that of \cite[Proposition 3.16]{nguyen2015one}.
  As a consequence of Corollary \ref{cor:remainInR}, we will from now on, when there is no confusion, always identify the entropy solution with its unique representative in \(\Rscr\).
  Furthermore, for an entropy solution \(M \in \Rscr\) of \eqref{eq:claw} with initial data \(M^0 \in \Rscr\), we will, when there is no confusion, identify the respective distributional derivatives with \(\rho, \rho_0 \in \Pscr(\R)\) for ease of notation.
  
  So far we have been considering uniqueness, but we would of course also like to have existence of entropy solutions.
  This is established using vanishing viscosity methods in \cite[Theorem 3.7]{golovaty2012existence}, from which the next result follows directly.
  \begin{thm}[Existence and uniqueness]\label{thm:entropy:EU}
    Assume a flux function \(\Ucal(t,m)\) of the form \eqref{eq:flux}.
    For any \(M^0 \in \Lp{\infty}(\R)\), \eqref{eq:claw} has a unique entropy solution \(M \in C([0,T];\Lp{1}_\mathrm{loc}(\R))\) satisfying
    \begin{equation*}
      \norm{M}_{\Lp{\infty}((0,T)\times\R)} \le \norm{M^0}_{\Lp{\infty}(\R)}.
    \end{equation*}
    Moreover, if \(M^0 \in \Lp{1}(\R)\), then \(M \in C([0,T]; \Lp{1}(\R))\).
  \end{thm}
  At this stage, we have verified that entropy solutions for \eqref{eq:claw} in the sense of Definition \ref{dfn:entropy} exist, are unique, and belong to the set \(\Rscr\) such that their distributional derivatives belong to \(\Pscr(\R)\).
  The next result, apart from being a useful property of the entropy solutions, will help us establish that entropy solutions with initial data in \(\Rscr_p\) remain in this space.
  
  \begin{lem}[Lipschitz-continuity for \(\Lp{1}\)-norm]\label{lem:L1Lip}
    Let \(M(t) \in \Lscr^\infty((0,T)\times\R) \cap C([0,T];\Lscr^1_\mathrm{loc}(\R))\) be the entropy solution of \eqref{eq:claw} with flux function \eqref{eq:flux} and \(M(t,\cdot) \in \Rscr\).
    Then, for \(0 \le s < t \le T\) we have
    \begin{equation}\label{eq:L1Lip}
      \norm{M(t)-M(s)}_{\Lp{1}(\R)} \le \abs{t-s} \sum_{i=0}^{n} \norm{F_i}_{\Lp{p}(0,1)} \norm{B_i}_{\Lp{\infty}(s,t)}.
    \end{equation}
    Moreover, \(M \in BV_\mathrm{loc}([0,T]\times\R)\).
  \end{lem}
  \begin{proof}
    The argument follows the lines of the proof of \cite[Theorem 7.10]{holden2015front} or \cite[Theorem 4.3.1]{dafermos2016hyperbolic}, with some adjustments to cater for our non-Lipschitzian flux function.
    Let \(\chi^\epsv(t)\) for \(t \in [0,T]\) be a smooth approximation of the characteristic function for the interval \([s,t] \subset [0,T]\) such that \(\lim\limits_{\epsv\to0+}\chi^\epsv = \chi_{[s,t]}\).
    Furthermore we choose \(\phiv_\epsv(t,x) = \chi^\epsv(t) \phi(x)\) for an arbitrary \(\phi \in C_\mathrm{c}^\infty(\R)\).
    For this test function, the weak formulation of \eqref{eq:claw} reads
    \begin{equation*}
      \int_0^T \int_\R \left( M \del_t \phiv_\epsv + \Ucal(t,M) \del_x \phiv_\epsv  \right)\dee x \dee t + \int_\R M^0(x)\phiv_\epsv(0,x)\dee x = 0,
    \end{equation*}
    and letting \(\epsv \to 0+\) we obtain
    \begin{equation*}
      -\int_\R \phi(x) (M(t,x)-M(s,x))\dee x + \int_s^t \int_\R \phi'(x) \Ucal(r,M(r,x))\dee x \dee r = 0.
    \end{equation*}
    Since \(\phi\) was arbitrary we have
    \begin{align*}
      \norm{M(t,\cdot)-M(s,\cdot)}_{\Lp{1}(\R)} &= \sup_{\abs{\phi}\le 1} \int_\R \phi(x) (M(t,x)-M(s,x))\dee x \\
      &= \int_s^t \sup_{\abs{\phi}\le 1} \int_\R \phi'(x) \Ucal(r,M(r,x)) \dee x \dee r \\
      &= \int_s^t \TV(\Ucal(r,M(r,x)))\dee r.
    \end{align*}
    If \(m \mapsto \Ucal(t,m)\) were Lipschitz, i.e., \(p=\infty\) in \eqref{eq:flux}, we could have finished the argument by using \(\TV(M(t,x)) = 1\).
    We will instead use Lemma \ref{lem:BVcalc} to estimate \(\TV(\Ucal(t,M))\) as follows,
    \begin{align*}
      \abs*{\int_\R \phi'(x) \Ucal(t,M(t,x))\dee x} &= \abs*{-\int_\R \phi(x) \del_x \Ucal(t,M(t,x))} = \abs*{-\sum_{i=0}^{n}B_i(t) \int_\R \phi(x) \Fcal_{i,M}'(t,x)\dee\rho(t,x)} \\
      &\le \sum_{i=0}^{n}\abs{B_i(t)} \left( \int_{\R}\abs{\phi(x)}^{p'} \dee\rho(t,x) \right)^{\frac{1}{p'}} \norm{\Fcal_{i,M}'}_{\Lp{p}(\R,\rho)} \le \sum_{i=0}^{n}\abs{B_i(t)} \norm{F_i}_{\Lp{p}(0,1)},
    \end{align*}
    where we for an arbitrary time \(t\) integrate over space.
    Taking supremum over \(\abs{\phi}\le 1\) and combining with the previous estimate we obtain \eqref{eq:L1Lip}.
    Note that we could also have arrived at the same bound by estimating the total variation of \(\Ucal(t,M(t,\cdot))\) as in the proof of Lemma \ref{lem:BVcalc}.
    Since \(\Rscr \subset BV(\R)\), then as in the proof of \cite[Theorem 4.3.1]{dafermos2016hyperbolic} it follows from the characterization of \(BV_{\mathrm{loc}}\)-functions in \cite[Theorem 1.7.2]{dafermos2016hyperbolic} and \eqref{eq:L1Lip} that \(M \in BV_\mathrm{loc}([0,T]\times\R)\).
  \end{proof}
  Moreover, the estimate on the total variation of the flux function \(\Ucal(t,M)\) from \eqref{eq:flux} suggests the following generalization of Lemma \ref{lem:BVcalc}.
  \begin{cor}\label{cor:BVcalcU}
    Assume \(p \in [1,\infty]\), take \(\Ucal(t,m)\) of the form \eqref{eq:flux}, and let \(M\) be the right-continuous distribution function of a measure \(\mu \in \Pscr(\R)\). Then \(\Ucal(t,M) \in BV(\R)\) and \(\del_x \Ucal(t,M) = g(t,\cdot) \mu\) for some \(g(t,\cdot) \in \Lp{p}(\R,\mu)\) where
    \begin{equation*}
      g(t,x) = \begin{cases}
        \sum_{i=0}^n B_i(t) F_i \circ M(x), & \mu(\{x\}) = 0, \\
        \sum_{i=0}^n B_i(t) \frac{\Fcal_i \circ M(x) - \Fcal_i\circ M(x-) }{\mu(\{x\})}, & \mu(\{x\}) \neq 0,
      \end{cases}
    \end{equation*}
    and
    \begin{equation*}
      \norm{g(t,\cdot)}_{\Lp{p}(\R,\mu)} \le \sum_{i=0}^n \abs{B_i(t)} \norm{F_i}_{\Lp{p}(0,1)}.
    \end{equation*}
  \end{cor}
  This also slightly generalizes the corresponding result for the Euler--Poisson-type flux \(\Ucal(t,\cdot) = \Fcal + t \Psi \) in \cite[Proposition 2.1]{nguyen2015one}, where \(\Fcal \in W^{1,p}(0,1)\) and \(\Psi \in C^1([0,1])\).
  The proof of Corollary \ref{cor:BVcalcU} is a straightforward adaption of the proof of Lemma \ref{lem:BVcalc}; that is, the structure of the flux function allows us to apply the same arguments termwise.
  
  Let us from now on denote by \(\Ucal_M'(t,x)\) the Radon--Nikodym derivative of \(\del_x\Ucal(t,M)\) with respect to \(\del_x M\), and according to Corollary \ref{cor:BVcalcU} this takes the form
  \begin{equation*}
    \Ucal_M'(t,x) = \sum_{i=0}^n B_i(t) \Fcal_{i,M}'(t,x),
  \end{equation*}
  where \(\Fcal_{i,M}' \in \Lp{p}(\R,\del_xM(t,\cdot))\) satisfies \(\del_x \Fcal_i(M(t,\cdot)) = \Fcal_{i,M}'(t,x) \del_xM(t,\cdot)\) by Lemma \ref{lem:BVcalc}.
  From the properties of entropy solutions and Corollary \ref{cor:BVcalcU} it follows that
  \begin{equation*}
    \int_0^T \norm{\Ucal'_M(t,\cdot)}_{\Lp{p}(\R,\rho(t))} \dee t \le T \sum_{i=0}^n \norm{B_i}_{\Lp{\infty}(0,T)}\norm{F_i}_{\Lp{p}(0,1)} < \infty.
  \end{equation*}
  
  \begin{lem}\label{lem:pmom}
    Let \(M(t,\cdot)\) for \(t \in [0,T]\) be the unique entropy solution of \eqref{eq:claw} with flux function \eqref{eq:flux} and initial data \(M^0 \in \Rscr_p\) for \(p\in[1,\infty)\). Then \(M(t,\cdot) \in \Rscr_p\).
  \end{lem}
  \begin{proof}
    Let us denote by \(m_p(\rho(t))\) the \(p\)\textsuperscript{th} moment of \(\rho(t) \in \Pscr_p(\R)\).
    The case \(p=1\) follows almost immediately from the triangle inequality and \eqref{eq:L1Lip}.
    Indeed, we have
    \begin{equation*}
      m_1(\rho(t)) = d_{\Wscr_1}(\rho(t),\delta_0) \le d_{\Wscr_1}(\rho(0),\delta_0) + d_{\Wscr_1}(\rho(t),\rho(0)) = m_1(\rho(0)) + \norm{M(t)-M(0)}_{\Lp{1}(\R)},
    \end{equation*}
    and the result follows.
    
    For \(p > 1\) we have to work a bit more, making use of that the probability measure \(\rho\) solves the continuity equation \eqref{eq:EF:den}.
    First, pick \(\phi \in C^\infty_{\mathrm{c}}(\R)\) and introduce the test function \(\phiv(t,x) = \chi^\epsv(t) \phi'(x)\) for \eqref{eq:weak}, where \(\chi^\epsv\) is as in the proof of Lemma \ref{lem:L1Lip}.
    Letting \(\epsv \to 0+\) we obtain
    \begin{equation*}
      -\int_\R \phi'(x) (M(t,x)-M(s,x))\dee x + \int_s^t \int_\R \phi''(x) \Ucal(r,M(r,x))\dee x \dee r = 0,
    \end{equation*}
    and integrating by parts, omitting \(x\) in the differentials for brevity, we arrive at
    \begin{align*}
      \int_\R \phi(x)\dee\rho(t) &= \int_\R \phi(x)\dee\rho(s) + \int_s^t \int_\R \phi'(x) \Ucal'_M(r,x)\dee\rho(r) \dee r \\
      &\le \int_\R \phi(x)\dee\rho(s) + (p-1)\int_s^t \int_\R \abs*{\frac1p \phi'(x)}^{\frac{p}{p-1}}\dee\rho(r) \dee r + \int_s^t \int_\R \abs*{\Ucal'_M(r,x)}^p\dee\rho(r) \dee r,
    \end{align*}
    having applied Young's inequality.
    Now, define \(\phi_R(x) \coloneqq (\abs{x}\wedge R)^p\) such that \(\phi'_R(x) = p \abs{x}^{p-1}\sgn(x)\) for \(\abs{x} < R\) and \(\phi'_R(x) = 0\) for \(\abs{x} > R\).
    Since \(\phi'_R\) is compactly supported we can find a sequence of smooth approximations \((\phi^\epsv_R)'\) such that \((\phi^\epsv_R)'(x) = \phi'_R(x)\) for \(\abs{x} \le R-\epsv\), \((\phi^\epsv_R)'(x) = 0\) for \(\abs{x} \ge R\), and \((\phi^\epsv_R)'(x)\) tends to \(\phi_R'(x)\) from below (above) for \(R-\epsv < x < R\) (\(-R < x < \epsv-R\)). 
    In consequence, the primitive \(\phi^\epsv_R(x) = \int_{-\infty}^x(\phi^\epsv_R)'(y)\dee y\) tends to \(\phi_R(x)-R^p \in C_{\mathrm{c}}(\R)\) pointwise from below.
    By the dominated convergence theorem it then follows that
    \begin{align*}
      \int_\R (\phi_R(x)-R^p)\dee\rho(t) &\le \int_\R (\phi_R(x)-R^p)\dee\rho(s) + (p-1)\int_s^t \int_\R \phi_R \chi_{(-R,R)}(x)\dee\rho(r) \dee r \\
      &\quad+ \int_s^t \int_\R \abs*{\Ucal'_M(r,x)}^p\dee\rho(r) \dee r,
    \end{align*}
    and since \(\rho(t) \in \Pscr\) for \(t \in [0,T]\) we are led to
    \begin{equation*}
      \int_\R \phi_R(x) \dee\rho(t) \le \int_\R \phi_R(x)\dee\rho(s) + (p-1)\int_s^t \int_\R \phi_R(x)\dee\rho(r) \dee r + \int_s^t \norm{\Ucal'_M(r,\cdot)}_{\Lp{p}(\R,\rho_r)}^p\dee r.
    \end{equation*}
    The Grönwall inequality then yields
    \begin{equation*}
      \int_\R \phi_R(x)\dee\rho(t) \le \left( \int_\R \phi_R(x)\dee\rho_0 + \int_0^t \norm{\Ucal'_M(s,\cdot)}_{\Lp{p}(\R,\rho(s))}^p\dee s \right) \e^{(p-1)t},
    \end{equation*}
    and letting \(R \to \infty\) we use once more the dominated convergence theorem to conclude.
  \end{proof}
  
  \subsection{Entropy solutions of the Euler system}
  Now we want to follow the ideas of \cite{nguyen2008pressureless,nguyen2015one} and use the entropy solutions of \eqref{eq:claw} to define solutions of \eqref{eq:EF} by applying a two-dimensional extension of the \(BV\)-chain rule Corollary \ref{cor:BVcalcU}.
  That is, following the proof therein, we can prove a slight modification of \cite[Theorem 2.2]{nguyen2015one}.
  \begin{thm}\label{thm:2DBVcalc}
    Let \(p \ge 2\) and assume that \(M \in C([0,T]; \Lp{1}_\mathrm{loc}(\R))\), \(M \in \Rscr\) for \(t \in (0,T)\) and that it satisfies \eqref{eq:claw} in distributions \(\Dscr'((0,T)\times\R)\) with \(\Ucal(t,m)\) given by \eqref{eq:flux}.
    According to Proposition \ref{cor:BVcalcU} we let
    \begin{equation*}
      \Ucal'_M \del_x M(t,\cdot) = \del_x \Ucal(t,M) \enspace \text{ for every } t \in (0,T).
    \end{equation*}
    Then we have \(M\), \(\Ucal(t,M) \in BV_\mathrm{loc}((0,T)\times\R)\), \(\Ucal'_M \in \Lp{1}(\abs{\nabla M})\) and
    \begin{equation}\label{eq:gradU}
      \nabla \Ucal(t,M) = \Ucal_M' \nabla M + \left( \sum_{i=1}^{n} b_i(t) \Fcal_i(M), \: 0 \right) = \left( -(\Ucal_M')^2, \: \Ucal_M' \right) \del_x M + \left( \sum_{i=1}^{n} b_i(t) \Fcal_i(M), \: 0 \right),
    \end{equation}
    where \(\nabla = (\del_t, \del_x)\) denotes the \(BV\)-gradient.
  \end{thm}
  
  Let us denote by \(\Tscr_p(\R)\) the tangent bundle of the Wasserstein space \(\Pscr_p(\R)\), i.e.,
  \begin{equation*}
    \Tscr_p(\R) = \left\{ (\mu, v) \colon \mu \in \Pscr_p(\R) \text{ and } v \in \Lp{p}(\R,\mu) \right\},
  \end{equation*}
  which we recognize as the generalization of \eqref{eq:TP2}.
  
  Now we can give the following existence proof of solutions to \eqref{eq:EF}, which can be compared with \cite[Theorem 1.3]{nguyen2015one}.
  
  \begin{thm}\label{thm:EFsols}
    Let \(p \in [2,\infty)\), \((\rho_0, v_0) \in \Tscr_p(\R) \) and denote by \(M^0\) the right-continuous distribution function of \(\rho_0\).
    Let \(M \in \Rscr_p\) be the corresponding entropy solution of \eqref{eq:claw} with flux function \eqref{eq:flux}.
    Define \(\rho(t,x) \coloneqq \del_x M(t,x)\) and \(v(t,x) \coloneqq \Ucal_M'(t,x)\) as in Theorem \ref{thm:2DBVcalc}.
    Then \((\rho(t,\cdot), v(t,\cdot)) \in \Tscr_p(\R)\) for \(t \in (0,T)\) and \((\rho,v)\) is a distributional solution of \eqref{eq:EF} for 
    \begin{equation*}
      f[\rho]=\sum_{i=1}^{n} b_i(t) \del_x \Fcal_i(M).
    \end{equation*}
    Moreover, we have \(\rho(t) \in \Pscr_p(\R)\) and \(v(t,\cdot) \in \Lp{p}(\R,\rho(t))\) for every \(t > 0\) and
    \begin{enumerate}[label=(\roman*)]
      \item \(\norm{v(t,\cdot)}_{\Lp{p}(\R,\rho(t))} \le \sum_{i=0}^n \abs{B_i(t)}\norm{F_i}_{\Lp{p}(0,1)}\) for a.e.\ \(t > 0\),
      \item \begin{equation}\label{eq:Wpcont}
        d_{\Wscr_p}(\rho(t), \rho(s)) \le \abs{t-s} \sum_{i=0}^n \norm{F_i}_{\Lp{p}(0,1)} \norm{B_i}_{\Lp{\infty}(0,T)}, \quad 0 \le s \le t \le T. 
      \end{equation}
    \end{enumerate}
    
  \end{thm}
  \begin{proof}
    Since \(M\) satisfies \eqref{eq:claw} in the distributional sense, Theorem \ref{thm:2DBVcalc} yields that \eqref{eq:M:evol} is satisfied in the same sense.
    Differentiating with respect to \(x\) in the sense of distributions shows that \eqref{eq:EF:den} is satisfied.
    For the momentum equation \eqref{eq:EF:mom} we again find from Theorem \ref{thm:2DBVcalc} that
    \begin{align*}
      \del_t (\rho v) &= \del_t (\del_x \Ucal(t,M)) = \del_x (\del_t \Ucal(t,M) ) = \del_x \left( v\del_t M + \sum_{i=1}^{n} b_i(t) \Fcal_i(M)\right) \\
      &= \del_x (-v^2 \rho) + \sum_{i=1}^{n} b_i(t) \del_x \Fcal_{i}(M).
    \end{align*}
    Indeed, we have shown that the following momentum equation is satisfied in the distributional sense,
    \begin{equation*}
      \del_t (\rho v) + \del_x(\rho v^2) = \left(\sum_{i=1}^{n} b_i(t) \Fcal_{i,M}'(t,x)\right) \rho \quad \text{ in } \Dscr'((0,T)\times\R).
    \end{equation*}
    Now, the convergence \(\rho(t,\cdot) \to \rho_0\) in \(\Pscr(\R)\) follows from the convergence of \(M(t,\cdot)\) to \(M^0\) in \(\Lp{1}(\R)\) in \eqref{eq:L1Lip}: indeed, by \eqref{eq:metricEquiv1} this implies convergence in the 1-Wasserstein metric, which again implies narrow convergence, see, e.g, \cite[Theorem 7.12]{villani2003topics}.
    We must also show \(\rho v(t,\cdot) \to \rho_0 v_0\) in \(\Mscr(\R)\), and from \eqref{eq:gradU} we have
    \begin{equation*}
      \lim\limits_{t\to0+} \int_\R \phiv(x) v(t,x)\dee\rho(t,x) = -\lim\limits_{t\to0+} \int_\R \phiv'(x) \Ucal(t,M(t,x))\dee x = \int_\R \phiv'(x) \Vcal^0(M^0(x))\dee x.
    \end{equation*}
    It remains to show that the distributional derivative of \(\Vcal^0(M^0)\) is \(v_0\rho_0\); however, this is in a sense how \(\Vcal^0\) was defined.
    Indeed, using again the push-forward, we have
    \begin{equation*}
      \Vcal^0(M^0(x)) = \int_0^{M^0(x)} v_0 \circ X^0(m) \dee m = \int_0^1 \chi_{(-\infty,x]}v_0\circ X^0(m) \dee m = \int_{(-\infty,x]} v_0(x)\dee \rho_0(x),
    \end{equation*}
    and so we have a distributional solution of \eqref{eq:EF}.
    Statement (i) follows from Corollary \ref{cor:BVcalcU} applied to \(\Ucal(t,M)\).
    For the remaining statement we note that \(\rho(t) \in C([0,T];\Pscr_1(\R))\) follows from \(M \in C([0,T]; \Rscr_1)\), while \(\rho(t) \in \Pscr_p(\R)\) for \(t \in [0,T]\) by Lemma \ref{lem:pmom}.
    Therefore, \(\rho(t) \colon [0,T] \to \Pscr_p(\R)\) is a narrowly continuous curve satisfying the continuity equation \eqref{eq:EF:den} for \(v\) with \(\norm{v(t,\cdot)}_{\Lp{p}(\R,\rho(t))} \in \Lp{1}(0,T)\), and so by \cite[Theorem 8.3.1]{ambrosio2008gradient} it is also an absolutely continuous curve with \(\Wscr_p\)-metric derivative satisfying \(\abs{\rho'(t)} \le \norm{v(t,\cdot)}_{\Lp{p}(\R,\rho(t))}\) for a.e.\ \(t \in (0,T)\). Combining this with (i) we can integrate to obtain \eqref{eq:Wpcont}.
  \end{proof}
  
  \begin{exm}[Examples \ref{exm:flux:EP} and \ref{exm:flux:EPld} revisited]\label{exm:reflux}
    For the flux function \ref{eq:flux:EP}, the forcing term in Theorem \ref{thm:EFsols} becomes
    \begin{equation*}
      1 \cdot \Acal_M'(t,x) \rho = -\left(\frac{\alpha}{2}\left(M(t,x)+M(t,x-)-1\right)+\beta\right)\rho,
    \end{equation*}
    which is exactly the Euler--Poisson forcing.
    On the other hand, with linear damping as in \eqref{eq:flux:EPld} it becomes
    \begin{align*}
      \left(-\gamma \e^{-\gamma t} \cdot (\Vcal^0)'_M(t,x) + \e^{-\gamma t} \cdot \Acal_M'(t,x)\right)\rho &= -\gamma \left( \e^{-\gamma t} (\Vcal^0)_M'(t,x) + \frac{1-\e^{-\gamma t}}{\gamma} \Acal_M'(t,x) \right)\rho + \Acal_M'(t,x) \rho \\
      &= -\gamma v \rho + \Acal_M'(t,x) \rho,
    \end{align*}
    which is the Euler--Poisson forcing with linear damping.
  \end{exm}
  
  \begin{exm}[The Riemann problem for a single Dirac mass]
    Let us now consider the Riemann problem for our scalar conservation law \eqref{eq:claw}, \eqref{eq:flux:EP} with \(\beta = 0\), that is, initial data of the form
    \begin{equation*}
      M^0(x) = \begin{cases}
        M_l, & x < 0, \\
        M_r, & x \ge 0.
      \end{cases}
    \end{equation*}
    This corresponds to a Dirac mass of size \(M_r-M_l\) for the initial density \(\rho_0\), located at \(x=0\).
    Assume that this mass has initial velocity \(v^0\), so that \(\Vcal^0(m) = \Vcal^0(M_l) + v^0 m\) for \(m \in [M_l, M_r]\).
    On the same interval, the full flux function is then
    \(\Ucal(t,\theta) = \Vcal^0(M_l) + \int_{M_l}^\theta (v^0 - t\frac{\alpha}{2}(2m-1))\dee m\).
    If \(\alpha \ge 0\) then this is concave on the interval for all \(t \ge 0\), and its corresponding lower convex envelope is
    \begin{equation*}
      \Ucal_{\smallsmile}(t,m) = \Vcal^0(M_l) + \left(v^0 - t\frac{\alpha}{2}(M_l + M_r -1)\right) (m-M_l).
    \end{equation*}
    Therefore, for \(\alpha \ge 0\), the entropic solution is a shock wave connecting the states \(M_l\) and \(M_r\) with velocity \(\dot{x}(t)\) satisfying the Rankine--Hugoniot-condition , i.e.,
    \begin{equation*}
      M(t,x) = \begin{cases}
        M_l, & x < v^0 t - \frac{\alpha}{4}t^2 (M_l+M_r-1), \\ M_r, & x \ge v^0 t - \frac{\alpha}{4}t^2 (M_l+M_r-1).
      \end{cases}
    \end{equation*}
    That is, the Dirac mass continues moving along a parabolic, or linear if \(M_l = 0\) and \(M_r = 1\), trajectory; that is, particles remain particles.
    
    On the other hand, if \(\alpha < 0\) then the flux function is strictly convex on \([M_l,M_r]\) for all time, meaning \(\Ucal_{\smallsmile}(t,m) = \Ucal(t,m)\).
    For fixed \(t\) we consider the derivative of \(m \mapsto \Ucal_{\smallsmile}(t,m)\), which we denote by \(\Ucal'_{\smallsmile}(t,m)\).
    Since \(m \mapsto \Ucal'_{\smallsmile}(t,m)\) is monotone in \(m\) we may define its inverse \(v \mapsto (\Ucal'_{\smallsmile})^{-1}(t,v)\).
    The left and right constant states will travel with respective velocities \(\Ucal'(t,M_l) = v^0 - t\frac{\alpha}{2}(2M_l-1)\) and \(\Ucal'(t,M_r) > \Ucal'(t,M_l)\).
    The remaining middle section will be a continuous transition from \(m = M_l\) to \(m = M_r\), where the velocity corresponding to \(m\) is \(\Ucal'(t,m)\).
    Integrating we find \(x = v^0 t - \frac14 t^2 \alpha (2m-1)\), which we can invert to find the value of \(M(t,x)\) in the middle section, and we find the solution
    \begin{equation*}
      M(t,x) = \begin{cases}
        M_l, & x < v^0 t - \frac14 t^2 \alpha (2M_l-1), \\
        \frac{x-v^0t-\frac14 \alpha t^2}{\frac12 \alpha t^2}, & v^0 t - \frac14 t^2 \alpha (2M_l-1) \le x < v^0 t - \frac14 t^2 \alpha (2M_r-1) \\
        M_r, & x \ge v^0t - \frac14 t^2 \alpha (2M_r-1).
      \end{cases}
    \end{equation*}
    That is, the Dirac mass is immediately smoothed out with a parabolic rarefaction fan.
    Compare these examples with the second order diffusion of a Dirac mass in \cite[Equation (1.14)]{brenier2013sticky}, as well as the corresponding first-order solutions \cite[Equations (2.23)--(2.24)]{bonaschi2015equivalence}.
  \end{exm}
  
  From Example \ref{exm:reflux} we see that Theorem \ref{thm:2DBVcalc} is all well and good when the forcing term in \eqref{eq:gradU} turns out to match the forcing term for the original Euler system.
  For an entropy solution of \eqref{eq:claw} to yield a distributional solution \((\rho, v) = (\del_x M, \Ucal'_M)\) of \eqref{eq:EF}, we see by applying the \(BV\)-gradient of \eqref{eq:gradU} that this will be the case if
  \begin{equation}\label{eq:f&U}
    \del_x \pdiff{\Ucal(t,M)}{t} = f[\del_x M] \quad \text{in } \Dscr',
  \end{equation}
  which is indeed the compatibility condition we read from the flux function \eqref{eq:flux} and the force distribution appearing in the proof of Theorem \ref{thm:EFsols}.
  To ensure that \eqref{eq:f&U} holds, we can define entropy solutions of the Euler system as follows.
  \begin{dfn}[Entropy solution of the Euler system]\label{dfn:entropyEF}
    We define the distributional solution of \eqref{eq:EF} coming from an entropy solution \(M\) of the scalar conservation law \eqref{eq:claw} in Theorem \ref{thm:EFsols} to be an entropy solution of the Euler system \eqref{eq:EF} if the flux function \(\Ucal(t,m)\) additionally satisfies
    \begin{equation}\label{eq:F&U}
      \Ucal(t,m) = \Vcal^0(m) + \int_0^m \int_0^t F[M_s^{-1}](\omega)\dee\omega,
    \end{equation}
    where \(M_t^{-1}\) is the generalized inverse of \(M(t,\cdot)\), and \(F[X]\) is defined as in \eqref{eq:f&F}.
  \end{dfn}
  Combining Lemmas \ref{lem:decomp} and \ref{lem:BVcalc} with the relation \eqref{eq:f&F}, it follows that \eqref{eq:F&U} implies \eqref{eq:f&U}.
  We note that unless one is in a case where \(F[M^{-1}_t]\) simplifies, such as for Euler--Poisson, \eqref{eq:F&U} as it stands depends on the generalized inverse of \(M\), and in order to write it as a flux function depending only on \(M\), if possible, would require some additional information about the evolution of solutions of the system, e.g., that mass is preserved, which is used in the simplification for Euler--Poisson, and the evolution of the center of mass, see Section \ref{ss:CCZ}.
  In general, it would then seem that the Lagrangian solutions of Section \ref{s:gradflow} are simpler to work with, as one can work directly with the monotone rearrangement \(X\) appearing in \(F[X]\).
  
  \section{Equivalence of solution concepts}\label{s:equiv}
  In this section we will show that the Lagrangian solutions coming from \(\Lp{2}\)-gradient flow in Section \ref{s:gradflow} and the entropy solutions in Section \ref{s:entropy} are equivalent, and uniquely defined.
  Indeed, from their expositions we know that each solution concept gives rise to a uniquely defined distributional solution of the Euler--Poisson system \eqref{eq:EP} under suitable hypotheses on the forcing term.
  Our contribution is to show that they are equivalent already from their definitions.
  
  \subsection{Projections and shock admissibility} \label{ss:proj&admiss}
  As a preparation for the main statement, we present some auxiliary results which contain the core of our argument.
  
  \begin{lem}[Projection on \(\Hscr_{X}\) and the Radon--Nikodym derivative]\label{lem:ProjHX&RH}
    Let \(X \in \Kscr\) and \(W \in \Lp{2}(\Omega)\) be given.
    From these we can define the right-continuous generalized inverse \(M(x)\) of \(X(m)\) and any primitive \(\Wcal(m)\) such that \(\Wcal' = W\).
    Then we have the following identity,
    \begin{equation*}
      \Proj_{\Hscr_{X}} W = \Wcal_M'(X),
    \end{equation*}
    where \(\Hscr_{X}\) is defined in \ref{eq:HX} and \(\Wcal'_M(x)\) is the Radon--Nikodym derivative of \(\del_x \Wcal(M(x))\) with respect to \(\del_x M(x)\).
  \end{lem}
  \begin{proof}
    From Lemma \ref{lem:BVcalc} we have that \(\Wcal'_M(x)\) is given by
    \begin{equation}\label{eq:WcalRN}
      \Wcal'_M(x) = \begin{cases}
        W(M(x)), & [[M(x)]] = 0 \\
        \ds \frac{\Wcal(M(x))-\Wcal(M(x-))}{M(x)-M(x-)}, & [[M(x)]] \neq 0,
      \end{cases}
    \end{equation}
    where \([[M(x)]] = M(x)-M(x-)\).
    As in Section \ref{ss:convex}, we can consider the open set \(\Omega_X\) as consisting of countably many disjoint maximal intervals, and if \(m \notin \Omega_X\), then \(M(t,X(t,m)) = m\).
    On the other hand, if \(m \in (m_l, m_r)\) for a maximal interval \((m_l, m_r) \subset \Omega_X\), then by the right-continuity of \(X\) and \(M\) we have \(M(X(m)) = m_r\) and \(M(X(m)-) = m_l\).
    Combining these properties with the expression for \(\Wcal'_M\),
    we find exactly \(\Proj_{\Hscr_{X}}W\) defined by \eqref{eq:Proj:HX}.
  \end{proof}
  Lemma \ref{lem:ProjHX&RH} provides a connection between the projection onto \(\Hscr_{X}\) and the Rankine--Hugoniot condition \eqref{eq:RH}.
  In particular, if \(\dot{X}(t) = \Proj_{\Hscr_{X(t)}} U(t)\) for some \(U(t) \in \Lp{2}(\Omega)\) with \(\Ucal(t,m) = \int_0^mU(t,\omega)\dee\omega\) we have
  \begin{equation}\label{eq:ProjHX&RH}
    (\Proj_{\Hscr_{X(t)}} U)(m) = \Ucal_M'(t,X(t,m)),
  \end{equation}
  which for \(m \in \Omega_{X(t)}\) corresponds exactly to the Rankine--Hugoniot condition for \(M\) and flux function \(\Ucal(t,M)\).
  
  Now, let \(X(t) \in \Kscr\) be a Lagrangian solution of the differential inclusion \eqref{eq:DIU} for which at time \(t\) we have \(\Omega_{X(t)} \neq \emptyset\).
  Then there is some maximal interval \((m_l, m_r) \subset \Omega_{X(t)}\) where \(X(t,m)\) is `flat', and so its generalized inverse \(M(t,x)\) has a jump of size \(m_r-m_l\) at \(x = X(t,m)\) for \(m \in (m_l,m_r)\).
  Then we observe from Lemma \ref{lem:ProjHX&RH} that having \(\dot{X} = \Proj_{\Hscr_{X(t)}}U\) corresponds exactly to the Rankine--Hugoniot condition \eqref{eq:RH} for \(M\) and flux function \(\Ucal\) given by the primitive of \(U\).
  
  \begin{lem}[Projection on \(T_X\Kscr\) and the Ole\u{\i}nik E-condition]\label{lem:ProjHX&O}
    Let \(X \in \Kscr\) and \(U \in \Lp{2}(\Omega)\) be given, and define \(M\) and \(\Ucal\) as in Lemma \ref{lem:ProjHX&RH}.
    Then \(U - \Proj_{\Hscr_{X}}U \in N_X\Kscr\), or equivalently \(\Proj_{\Hscr_{X}}U = \Proj_{T_X\Kscr}U\), if and only if \(\Proj_{\Hscr_{X}}U\) satisfies
    \begin{equation}\label{eq:ProjHX&O}
      (\Proj_{\Hscr_{X}}U)(m) \le \frac{\Ucal(m)-\Ucal(m_l)}{m-m_l}
    \end{equation}
    for any \(m \in (m_l, m_r)\), a maximal interval of \(\Omega_X\).
    In particular, if \(\dot{X}(t) = \Proj_{\Hscr_{X(t)}} U(t)\) for some \(U(t) \in \Lp{2}(\Omega)\), then for \(m \in \Omega_{X(t)}\), \eqref{eq:ProjHX&O} is exactly the Ole\u{\i}nik E-condition \eqref{eq:Oleinik} for \(M(t,x)\) and flux function \(\Ucal(t,m) = \int_0^mU(t,\omega)\dee\omega\).
  \end{lem}
  \begin{proof}
    Define \(\Xi(m) = \int_0^m (U-\Proj_{\Hscr_{X}}U)(\omega)\dee\omega\) and observe that on any maximal interval \((m_l,m_r) \subset \Omega_X\), \(\Proj_{\Hscr_{X}}U\) by its definition \eqref{eq:Proj:HX} equals the average of \(U\) over the same interval, which shows that \(\Xi(m) = 0\) for any \(m \in \Omega\setminus\Omega_X\).
    On the other hand, for any \(m\) contained in a maximal interval \((m_l, m_r)\) we have
    \begin{equation*}
      \Xi(m) = \Xi(m) - \Xi(m_l) = \Ucal(m)-\Ucal(m_l) - (m-m_l)\Proj_{\Hscr_{X}}U,
    \end{equation*}
    and so \(\Xi(m) \ge 0\) if and only if \eqref{eq:ProjHX&O} holds.
    The claim is then a consequence of Lemma \ref{lem:NX}.
    If \(X(t,\cdot), U(t,\cdot) \in \Lp{2}(\Omega)\) were to depend on time, we see that the corresponding inequality would still hold.
  \end{proof}
  Observe that, strictly speaking, \eqref{eq:ProjHX&O} only gives the rightmost inequality in \eqref{eq:Oleinik}.
  However, we can obtain the other inequality,
  \begin{equation}\label{eq:ProjHX&O2}
    (\Proj_{\Hscr_{X}}U)(m) \ge \frac{\Ucal(m_r)-\Ucal(m)}{m_r-m},
  \end{equation}
  by considering \(\Xi(m_r)-\Xi(m)\) for \(m \in (m_l,m_r)\) together with the fact that \(\Xi(m_r) = 0\).
  
  For a right-continuous probability distribution function \(M(x)\) corresponding to \(\mu \in \Pscr\) we may then apply Lemma \ref{lem:BVcalc} to the Kru\v{z}kov entropy-entropy flux pair \(\eta_k(M)\) and \(q_k(t,M)\) from \eqref{eq:KruzkovPair} to find
  \begin{equation}\label{eq:etaRN}
    \begin{aligned}
      \eta_{k,M}'(x) &= \begin{cases}
        \sgn(M(x)-k), & \mu(\{x\}) = 0, \\[2mm]
        \ds \frac{\abs{M(x)-k}-\abs{M(x-)-k}}{\mu(\{x\})}, & \mu(\{x\}) \neq 0
      \end{cases} \\
      &= \begin{cases}
        \sgn(M(x)-k), & k \notin (M(x-),M(x)), \\[2mm]
        \ds \frac{M(x)+M(x-)-2k}{M(x)-M(x-)}, & k \in (M(x-),M(x))
      \end{cases}
    \end{aligned}
  \end{equation}
  and
  \begin{equation}\label{eq:qRN}
    \begin{aligned}
      q_{k,M}'(t,x) &= \begin{cases}
        \sgn(M(x)-k) U(t,M(x)), & \mu(\{x\}) = 0, \\[2mm]
        \ds \frac{[[\sgn(M-k)\left(\Ucal(t,M)-\Ucal(t,k)\right)]](x)}{\mu(\{x\})}, & \mu(\{x\}) \neq 0,
      \end{cases} \\
      &= \begin{cases}
        \sgn(M(x)-k) U(t,M(x)), & M(x-) = M(x), \\[2mm]
        \ds \sgn(M(x)-k) \frac{\Ucal(t,M(x))-\Ucal(t,M(x-))}{M(x)-M(x-)}, & k \notin (M(x-),M(x)) \neq \emptyset, \\[4mm]
        \ds \frac{\Ucal(t,M(x))+\Ucal(t,M(x-))-2\Ucal(t,k)}{M(x)-M(x-)}, & k \in (M(x-),M(x)),
      \end{cases}
    \end{aligned}
  \end{equation}
  where we as short-hand notation write \([[f(t,M)]](x) \coloneqq f(t,M(x))-f(t,M(x-))\) for a function \(f \colon [0,T]\times\R \to \R\).
  Note that we can apply Lemma \ref{lem:ProjHX&RH} to \eqref{eq:KruzkovPair} to obtain the identities
  \begin{equation*}
    \Proj_{\Hscr_{X}}\left(\sgn(m-k)\right) = \eta'_{k,M}(X), \qquad \Proj_{\Hscr_{X}}\left(\sgn(m-k)U(t,m)\right) = q'_{k,M}(t,X).
  \end{equation*}
  Combining the above results we can deduce the following result, which should be compared with the distributional inequality defining the Kru\v{z}kov entropy condition \eqref{eq:Kruzkov}.
  \begin{cor}
    Let \(X \in \Kscr\) and \(U(t,\cdot) \in \Lp{2}(\Omega)\) be given, and define \(M\) and \(\Ucal(t,\cdot)\) as in Lemma \ref{lem:ProjHX&O}.
    Assume furthermore that \(U-\Proj_{\Hscr_{X}}U \in N_X\Kscr\), then we have the inequality
    \begin{equation}\label{eq:entropyIneq}
      q_{k,M}'(t,X) \le \eta'_{k,M}(X) \Ucal_M'(t,X).
    \end{equation}
  \end{cor}
  \begin{proof}
    Following the proof of Lemma \ref{lem:ProjHX&RH} while comparing \eqref{eq:WcalRN}, \eqref{eq:etaRN} and \eqref{eq:qRN} we see that this holds as an equality for the cases where \(k \notin (M(X(m)-),M(X(m)))\).
    For the remaining case, we consider \(k \in (m_l, m_r) = (M(X(m)-), M(X(m)))\), and by \eqref{eq:qRN} we have
    \begin{align*}
      q'_{k,M}(t,X(m)) &= \frac{\Ucal(t,m_r) + \Ucal(t,m_l) - 2\Ucal(t,k)}{m_r - m_l} \le \frac{m_r-k}{m_r-m_l} (\Proj_{\Hscr_{X}} U)(m) - \frac{k-m_l}{m_r-m_l} (\Proj_{\Hscr_{X}} U)(m) \\
      &= \frac{m_r+m_l-2k}{m_r-m_l} (\Proj_{\Hscr_{X}} U)(m) = \eta'_{k,M}(X(m)) \Ucal'_M(t,X(m)),
    \end{align*}
    where we have combined \eqref{eq:ProjHX&RH}, \eqref{eq:ProjHX&O} and \eqref{eq:ProjHX&O2}.
  \end{proof}
  
  By a similar reasoning, based on the formulas for the Radon-Nikodym derivatives in \eqref{eq:WcalRN}, \eqref{eq:etaRN} and \eqref{eq:qRN}, we can characterize entropy solutions of \eqref{eq:claw} in Definition \ref{dfn:entropy} in the space \(BV_\mathrm{loc}\) by the Ole\u{\i}nik E-condition \eqref{eq:Oleinik}, see \cite[Section 4.5]{dafermos2016hyperbolic} for a similar result.
  \begin{lem}\label{lem:KequivO}
    Assume \(M\) is a weak solution of \eqref{eq:claw} with a flux function of the form \(\Ucal(t,m) = \int_0^mU(t,\omega)\dee\omega\) for \(U(t)\in\Lp{2}(\Omega)\).
    Then it is an entropy solution if and only if it satisfies the Ole\u{\i}nik E-condition \eqref{eq:Oleinik} along its discontinuities, i.e., where \(M(t,x) > M(t,x-)\).
  \end{lem}
  \begin{proof}
    Being a weak solution, \(M\) satisfies \(\del_t M = -\del_x \Ucal(t,M) = -\Ucal'_M(t,x) \del_x M\) in distributions.
    Then the Kru\v{z}kov entropy inequality can be rewritten as \(\del_t \eta + \del_x q \le 0\) in \(\Dscr'\), that is,
    \begin{align*}
      0 &\le -\int_0^T \int_\R \phiv(t,x) \eta'_{k,M}(t,x) \del_t M(t,x) \dee t - \int_0^T \int_\R \phiv(t,x) q'_{k,M}(t,x) \del_x M(t,x) \dee t \\
      &= \int_0^T \int_\R \phiv(t,x) \left( \eta_{k,M}'(t,x) \Ucal_M'(t,x) - q_{k,M}'(t,x) \right) \del_x M(t,x) \dee t \\
      &= \int_0^T \int_{\{x \colon \rho(t,\{x\}) > 0\}} \phiv(t,x) \left( \eta_{k,M}'(t,x) \Ucal_M'(t,x) - q_{k,M}'(t,x) \right) \dee\rho(t) \dee t
    \end{align*}
    for any nonnegative \(\phiv \in C^\infty_{\mathrm{c}}([0,T)\times\R)\), where we have combined \eqref{eq:WcalRN} for \(\Ucal(t,M)\) with \eqref{eq:etaRN} and \eqref{eq:qRN} to deduce the identity \(\eta_{k,M}'(t,x) \Ucal_M'(t,x) = q_{k,M}'(t,x)\) whenever \(\rho(t,\{x\}) = 0\) or \(k \notin (M(t,x-),M(t,x))\).
    That is, the Kru\v{z}kov condition reduces to a condition on the discontinuities of \(M(t,\cdot)\).
    Assume now there is such a discontinuity at \(x = \bar{x}\), i.e., \(m_r \coloneqq M(t,\bar{x}) > M(t,\bar{x}-) \eqqcolon m_l\), and let \(k \in (m_l,m_r)\).
    Then the Kru\v{z}kov condition implies for a.e.\ \(t\) that
    \begin{equation*}
      \frac{m_r+m_l-2k}{m_r-m_l} \frac{\Ucal(t,m_r)-\Ucal(t,m_l)}{m_r-m_l} \ge \frac{\Ucal(t,m_r)+\Ucal(t,m_l)-2\Ucal(t,k)}{m_r-m_l}
    \end{equation*}
    for any \(k \in (m_l,m_r)\), which is equivalent to the Ole\u{\i}nik E-condition \eqref{eq:Oleinik}, cf.\ \cite[Section 2.1]{holden2015front}.
    Therefore, we see that the Kru\v{z}kov condition is satisfied if and only if the Ole\u{\i}nik E-condition above holds for every such discontinuity for a.e.\ \(t\).
  \end{proof}
  
  A final useful observation concerns integrating orthogonal functions, i.e., for \(X \in \Kscr\) we have
  \begin{equation*}
    \int_\Omega U(m) W(m)\dee m = 0, \qquad U \in \Hscr_{X}, \quad W \in \Hscr_{X}^\perp.
  \end{equation*}
  In particular, for any \(W \in \Lp{2}(\Omega)\) we have that \(W - \Proj_{\Hscr_{X}}W \in \Hscr_{X}^\perp\), while for any smooth \(\phiv\) we have \(\phiv\circ X \in \Hscr_{X}\), and so
  \begin{equation}\label{eq:intPerp}
    \int_\Omega (\phiv\circ X) W \dee m = \int_\Omega (\phiv \circ X) (\Proj_{\Hscr_{X}}W) \dee m.
  \end{equation}
  
  \subsection{Equivalence theorem and proof}
  We now state our main results.
  
  \begin{thm}[Lagrangian solutions are entropy solutions]\label{thm:LagToEnt}
    Let \((X(t),V(t))\) for \(t \in [0,T]\) be a Lagrangian solution of \eqref{eq:DIU} in the sense of Definition \ref{dfn:LagSol}. Then \(M(t,\cdot)\), defined as the right-continuous generalized inverse of \(X(t)\), is an entropy solution of \eqref{eq:claw} for \(t \in [0,T]\) with flux function given by the primitive of the prescribed velocity \(U\).
  \end{thm}
  
  \begin{thm}[Entropy solutions are Lagrangian solutions]\label{thm:EntToLag}
    Let \(M(t,x)\) for \(t \in [0,T]\) be an entropy solution of the scalar conservation law \eqref{eq:claw} in the sense of Definition \ref{dfn:entropy} with flux function \eqref{eq:flux}.
    Then, \(X(t)\), defined as the right-continuous generalized inverse of \(M(t,\cdot)\), satisfies the differential inclusion \eqref{eq:DIU:X} for \(U = \pdiff{\Ucal}{m}\).
    If \(M\) additionally provides an entropy solution of the Euler system \eqref{eq:EF} in the sense of Definition \ref{dfn:entropyEF}, then \eqref{eq:DIU:U} is also satisfied, and \(X\) is a Lagrangian solution of \eqref{eq:DIU} for \(t \in [0,T]\). 
  \end{thm}
  
  Now, given initial data \((\rho_0,v_0) \in \Tscr_2\), we want to use the Lagrangian and entropy solution concepts to derive distributional solutions of \eqref{eq:EF}.
  Clearly, the Lagrangian initial data \((X^0, V^0)\) and the initial data \(M^0\) and flux function \(\Ucal^0\) for the conservation law \eqref{eq:claw} are in one-to-one correspondence through \(X^0, M^0\) being generalized inverses of one another and \(\Ucal^0 = \int^m_0 V^0(\omega)\dee\omega\).
  If \(f[\rho]\) furthermore is densely Lipschitz-continuous, then the corresponding \(F[X] \colon \Kscr \to \Lp{2}(\Omega)\) is also Lipschitz and the prescribed velocity \(U\) in \eqref{eq:U} is uniquely defined.
  Based on the previous theorems and the uniqueness of each solution concept in this case, we deduce the following corollary.
  
  \begin{cor}[Equivalence of solutions to the Euler system]
    Let \((\rho_0,v_0) \in \Tscr_2\) be initial data for \eqref{eq:EF}, and assume we have two sets of solutions \((\check{\rho}_t, \check{v}_t), (\hat{\rho}_t, \hat{v}_t) \in \Tscr_2 \) for \(t \in [0,T]\) provided by respectively the Lagrangian and entropy solution concept, see Definitions \ref{dfn:LagSol} and \ref{dfn:entropyEF}.
    Assume \(f[\rho]\) is densely Lipschitz-continuous. Then \( (\check{\rho}_t, \check{v}_t) = (\hat{\rho}_t, \hat{v}_t) \) in \(\Tscr_2\).
  \end{cor}
  
  \begin{proof}[Proof of Theorem \ref{thm:LagToEnt}]
    Assume we have a Lagrangian solution of \eqref{eq:DIU}; in particular, by Lemma \ref{lem:LagSol:props}
    we have \(\rho = X_\#\mathfrak{m}\) and there exists a unique map \(v \in \Lp{2}(\R,\rho)\) such that for \(t \in \Tcal\) we have \(V(t,\cdot) = v\circ X(t,\cdot) \in \Hscr_{X(t)}\). To give an idea of the proof strategy, we first show that the Lagrangian solution is a weak solution of \eqref{eq:claw}.
    
    Let \(\phiv \in C^\infty_{\mathrm{c}}([0,T)\times\R)\) be the test function we want to use in the weak formulation \eqref{eq:weak}, and define \(\Phi(t,x) = -\int_x^{\infty} \phiv(t,y)\dee y \in C_{\mathrm{b}}^\infty([0,T)\times\R)\) such that \(\Phi(T,x) = 0\) for any \(x\), while \(\lim\limits_{x\to\infty}\Phi(t,x) = 0\) for any \(t\).
    Then, the chain rule implies
    \begin{equation}\label{eq:ZeroPhi}
      0 -\lim\limits_{t\to0+}\Phi(t,X(t,m)) = \int_0^T \left( \Phi_t(t,X(t,m)) + \Phi_x(t,X(t,m)) \diff{^+}{t}X(t,m) \right)\dee t.
    \end{equation}
    Since \(X(t)\) converges strongly to \(X^0\) in \(\Lp{2}(\Omega)\) by definition of the Lagrangian solution, it follows that
    \begin{equation*}
      \lim\limits_{t\to0+} \int_\Omega \Phi(t,X(t,m))\dee m = \int_\Omega \Phi(0,X^0(m))\dee m = \int_\R \Phi(0,x)\del_x M^0(x) = -\int_\R \phiv(0,x)M^0(x)\dee x.
    \end{equation*}
    Then, integrating \eqref{eq:ZeroPhi} over \(\Omega = (0,1)\) and taking into account the BV-integration by parts formula in Lemma \ref{lem:BVcalc}, we obtain
    \begin{align*}
      \int_\R \phiv(0,x)M^0(x)\dee x &= \int_0^T \int_\Omega \left( \Phi_t(t,X(t,m)) + V(t,m) \Phi_x(t,X(t,m)) \right)\dee m \dee t \\
      &= \int_0^T \int_\Omega \left( \Phi_t(t,X(t,m)) + \Ucal_M'(t,X(t,m)) \Phi_x(t,X(t,m)) \right)\dee m \dee t \\
      &= \int_0^T \int_\R \left( \Phi_t(t,x) + \Ucal_M'(t,x) \Phi_x(t,x) \right)\del_xM(t,x) \dee t \\
      &= \int_0^T \int_\R \left( \Phi_t(t,x)\del_xM + \Phi_x(t,x)\del_x\Ucal(t,M) \right) \dee t \\
      &= \int_0^T \left[ \Phi_t(t,x) M(t,x) + \Phi_x(t,x) \Ucal(t,M(t,x)) \right]_{-\infty}^\infty \dee t \\
      &\quad- \int_0^T\int_\R \left( \Phi_{xt}(t,x) M(t,x) + \Phi_{xx}(t,x) \Ucal(t,M(t,x)) \right) \dee x \dee t \\
      &= -\int_0^T \int_\R \left( \phiv_{t}(t,x) M(t,x) + \phiv_{x}(t,x) \Ucal(t,M(t,x)) \right) \dee x \dee t.
    \end{align*}
    We conclude that \eqref{eq:weak} is satisfied and this is a weak solution of \eqref{eq:claw}.
    In the above, the second identity is a consequence of Lemmas \ref{lem:LagSol:props} and \ref{lem:ProjHX&RH}, in particular that \(V(t) = \Proj_{\Hscr_{X(t)}}U(t)\) for \(t \in \Tcal\), where \(\Tcal\) is dense in \((0,\infty)\), together with \eqref{eq:ProjHX&RH}.
    
    To show that it is also an entropy solution, note that we may subtract zero in the form of \(\int_0^T\int_\R \phiv_t\abs{k}\dee x\dee t\) on the left-hand side to deal with boundary terms as \(x\to-\infty\) when integrating by parts; indeed, for
    \begin{equation*}
      \int_{(-R,R]} \phiv_t (\abs{M-k}-\abs{k})\dee x = \left[\Phi_t(t,x) (\abs{M(t,x)-k}-\abs{k})\right]_{x=-R}^R - \int_{(-R,R]} \Phi_t \del_x\abs{M-k}
    \end{equation*}
    the boundary terms vanish as \(R \to \infty\).
    Then, recalling \eqref{eq:etaRN} and \eqref{eq:qRN}, we find
    \begin{align*}
      &\int_0^T \int_\R \left(\phiv_t \abs{M-k} + \phiv_x \sgn(M-k)(\Ucal(t,M)-\Ucal(t,k))\right)\dee x\dee t + \abs{k}\int_\R\phiv(0,x)\dee x \\
      &\quad= -\int_0^T \int_\R \left( \Phi_t \eta_M' + \Phi_x q_M' \right)(t,x)\del_xM \dee t \\
      &\quad= -\int_0^T \int_\Omega \left( \Phi_t \eta_M' + \Phi_x q_M' \right)(t,X(m))\dee m\dee t \\
      &\quad\ge -\int_0^T \int_\Omega \eta_M'(X(m)) \left( \Phi_t(t,X(m)) + \Ucal_M'(t,X(m))\Phi_x(t,X(m)) \right)\dee m\dee t \\
      &\quad= -\int_0^T \int_\Omega \sgn(m-k) \diff{^+}{t} \Phi(t,X(m)) \dee m\dee t \\
      &\quad= \int_\Omega \sgn(m-k) \Phi(0,X^0(m))\dee m,
    \end{align*}
    where we have used \eqref{eq:entropyIneq}, \eqref{eq:intPerp} and the BV-integration by parts formula in Lemma \ref{lem:BVcalc}.
    Applying \eqref{eq:intPerp} once more we find
    \begin{equation*}
      \int_\Omega \sgn(m-k) \Phi(0,X^0(m))\dee m = \int_\Omega \eta_M'(0,X^0(m)) \Phi(0,X^0(m))\dee m = \int_\R \eta_{M}'(0,x) \Phi(0,x)\del_x M^0(x),
    \end{equation*}
    and integrating by parts we have
    \begin{equation*}
      \int_\R \eta_{M}'(0,x) \Phi(0,x)\del_x M^0(x) = - \int_\R(\abs{M^0(x)-k}-\abs{k})\phiv(0,x)\dee x.
    \end{equation*}
    Combining the above relations we arrive at the desired inequality
    \begin{align*}
      &\int_0^T \int_\R \left(\phiv_t \abs{M-k} + \phiv_x \sgn(M-k)(\Ucal(t,M)-\Ucal(t,k))\right)\dee x\dee t + \abs{k} \int_\R \phiv(0,x)\dee x \\
      &\quad= -\int_0^T \int_\Omega \left( \Phi_t \eta_M' + \Phi_x q_M' \right)(t,X(t,m))\dee m \dee t \ge -\int_\R \phiv(0,x) \abs{M(0,x)-k}\dee x + \abs{k} \int_\R \phiv(0,x)\dee x,
    \end{align*}
    which we rearrange to find \eqref{eq:Kruzkov2}.
  \end{proof}
  
  \begin{proof}[Proof of Theorem \ref{thm:EntToLag}]
    Assume we have an entropy solution \(M(t,x)\) of \eqref{eq:claw}; in particular we have \(\rho = \del_x M\) and \(\rho v = \del_x \Ucal(t,M) = \Ucal_M'(t,x) \del_x M\), meaning \(v(t,x) = \Ucal_M'(t,x)\).
    
    We then introduce the right-continuous generalized inverse \(X(t,m)\) of \(M(t,x)\) through \(X_\#\mathfrak{m} = \rho\), as well as the velocity \(\diff{^+}{t}X(t,m) = V(t,m) \coloneqq v(t,X(t,m)) = \Ucal_M'(t,X(t,m))\) for \(t \ge 0\). Recall from Lemma \ref{lem:ProjHX&RH} that \(\Ucal'_M(t,X(t,m)) = \Proj_{\Hscr_{X(t)}}U(t)\).
    At any time \(t\), the generalized inverse \(X(t,m)\) must be a nondecreasing function of \(m\), meaning \(X(t,\cdot) \in \Kscr\), where we recall the convex cone \eqref{eq:cone}.
    Moreover, \(X\) is a locally Lipschitz curve \(X \colon [0,T] \to \Kscr\) due to \eqref{eq:Wpcont}.
    
    We must now show that \(U(t)-\Proj_{\Hscr_{X(t)}}U(t) \in N_{X(t)}\Kscr\) for a.e.\ \(t\).
    However, since \(M\) is an entropy solution, by Lemma \ref{lem:KequivO} it satisfies the Ole\u{\i}nik E-condition \eqref{eq:Oleinik} at any discontinuity of \(M(t,x)\) for a.e.\ \(t\), and so by \eqref{eq:ProjHX&RH} and Lemma \ref{lem:ProjHX&O} we conclude that \(U(t)-V(t)\) is contained in the normal cone \(N_{X(t)}\Kscr\) for a.e.\ \(t\).
    Then it follows that \(V(t) = \Proj_{T_{X(t)}\Kscr}U(t)\) for a.e.\ \(t\).
    Therefore \(X(t)\) solves the differential inclusion \eqref{eq:DIU:X} and is almost everywhere differentiable again due to \eqref{eq:Wpcont}.
    Then, as argued in the proof of \cite[Theorem 3.5]{brenier2013sticky}, see also \cite[Remark 3.9]{brezis1973operateurs},
    at every point of differentiability we must have \(\dot{X}(t,m) = V(t,m) = \Proj_{\Hscr_{X(t)}}U(t)\).
    It remains to show that \(\lim\limits_{t\to0+}X(t) = X^0\) in \(\Lp{2}(\Omega)\), but this follows from the identity \eqref{eq:metricEquiv} and the time-continuity \eqref{eq:Wpcont} for \(p=2\).
  \end{proof}
  
  \begin{rem}
    In the proof of Theorem 5.6 we used the Kru\v{z}kov, or integral, formulation to show that we obtain an entropy solution.
    Of course, we could have alternatively used Lemma \ref{lem:KequivO} and only verified the Ole\u{\i}nik E-condition at the discontinuities. 
  \end{rem}
  
  \begin{rem}[Connection to a uniqueness result for a related system]\label{rem:lefloch}
    As mentioned in the introduction, another Euler-type system on the line is studied in \cite{lefloch2015existence}, namely
    \begin{equation*}
      \del_t \rho + \del_x (\rho f'(v)) = 0, \qquad
      \del_t (\rho v) + \del_x (\rho v f'(v)) = \rho h\left(\int_{-\infty}^x\rho\dee y \right),
    \end{equation*}
    where \(h \colon \R \to \R\) is a given Lipschitz function, and \(f \colon \R \to \R\) is a smooth, convex function satisfying \(\lim_{\abs{v}\to+\infty}f(v)/\abs{v} = +\infty\).
    In particular, for \(f(v) = \frac12 v^2\) we recover a version of \eqref{eq:EF}.
    In their study they make use of a generalization of the Hopf--Lax--Ole\u{\i}nik formula, which is similar to the sticky Lagrangian solutions in Section \ref{s:gradflow} in the sense that there is a representation formula, see also \cite{tadmor2011variational}.
    An interesting observation is that their uniqueness results require \(h\) to be nondecreasing, which means that its primitive is concave.
    If one, to make sense of the above system in the case of measures with Dirac parts, replaced \(h(\int_{-\infty}^x\rho\dee y)\) with its Vol'pert average \eqref{eq:VolpertAvg}, the Lagrangian representation satisfying \eqref{eq:f&F} would be \(F[X](m) = h(m)\), independent of \(X\) as in the Euler--Poisson case.
    Furthermore, since \(h\) is nondecreasing, using Lemma \ref{lem:NX} one can then verify that this functional is sticking in the sense of Definition \ref{dfn:sticking}, meaning the unique Lagrangian solutions are globally sticky as in Definition \ref{dfn:LagSol:sticky}.
  \end{rem}
  
  \begin{rem}[Connection to a uniqueness result for the Euler--Poisson system]\label{rem:huang}
    Note that in the particular case of \eqref{eq:EF} with \(f[\rho] \equiv 0\), the ``other'' Ole\u{\i}nik condition \eqref{eq:onesided} has been established independently for both entropy solutions \cite{nguyen2008pressureless} and Lagrangian solutions \cite{natile2009wasserstein}.
    As mentioned in the introduction,  \cite{huang2023global} establishes uniqueness of entropy solutions for the attractive Euler--Poisson system \eqref{eq:EP} with \(\alpha = 1\), \(\beta=0\), and initial data \(\rho_0 \in \Pscr_\mathrm{c}(\R)\) and \(v_0 \in \Lp{\infty}(\R,\rho_0)\). 
    
    Observe that \(\Pscr_\mathrm{c}(\R) \subset \Pscr_2(\R)\), while \(v_0 \in \Lp{\infty}(\R,\rho_0)\) for \(\rho_0 \in \Pscr_\mathrm{c}(\R)\) also belongs to \(\Lp{2}(\R,\rho_0)\), so that Lagrangian solutions also cover such initial data.
    Of course, showing that \(\rho\) and \(v\) remain in these more restrictive spaces requires more analysis, but should rely on the finite speed of propagation due to bounded velocities, or equivalently, the flux function \eqref{eq:flux:EP} here being Lipschitz.
    
    Compared to our result, we merely argue that the Lagrangian solutions of the attractive Euler--Poisson system studied here are also entropy solutions in the sense of \cite{huang2023global}.
    Indeed, the weak continuity of the kinetic energy follows from the right-continuity of \(V\), cf.\ Lemma \ref{lem:LagSol:props}, which then carries over to \(\norm{V}^2_{\Lp{2}(\Omega)}\).
    It remains to verify \eqref{eq:onesided}, and to that end, as noted in \cite{natile2009wasserstein}, it is sufficient that the map \(m \mapsto X(t,m)-tV(t,m)\) is nondecreasing.
    Using the fact that \eqref{eq:EP} for \(\alpha \ge 0\) has globally sticky solutions given by the projection formulas \eqref{eq:sticky:X} and \eqref{eq:sticky:V}, together with \eqref{eq:Proj:K&HY} we find \(X(t)-tV(t) = \Proj_{\Hscr_{X(t)}}(Y(t)-tU(t))\).
    Since \(\Proj_{\Hscr_{X}}\) preserves monotonicity, it is sufficient to show that \(Y-tU\) is nondecreasing, and we easily compute \(Y(t)-tU(t) = X^0(m) + \frac12 t^2 \alpha(m-\frac12)\) which in our attractive case \(\alpha > 0\) clearly is increasing.
  \end{rem}

  \section{Particle dynamics for the Euler--Poisson equation}\label{s:EPparticle}
  To illustrate the equivalence between the Lagrangian and entropy solutions, we will in this section return to the particle dynamics for the Euler--Poisson system \eqref{eq:EP}.
  The motivation for considering particle dynamics comes from its central role in proving the equivalence of entropy solutions and gradient flows in \cite{bonaschi2015equivalence}.
  In addition, limits of particle dynamics play key roles in both the gradient flow formulation of \cite{brenier2013sticky} and the scalar conservation law-point of view in \cite{nguyen2008pressureless,nguyen2015one}.
  
  Particle solutions correspond to piecewise constant distribution functions of the corresponding conservation law, and this is where we will make use of the connection between the \(\Lp{2}\)-projections and shock admissibility conditions from the previous section.
  Indeed, a piecewise constant function will be a piecewise smooth solution of the conservation law, for which the Kru\v{z}kov entropy condition \eqref{eq:Kruzkov2} is equivalent to the Ole\u{\i}nik E-condition \eqref{eq:Oleinik} for the shocks.
  
  The benefit of considering the specific case of the Euler--Poisson system is, as seen in Section \ref{ss:LagEF}, that the forcing term \(F[X](m)\) turns out to be independent of the position \(X\).
  In particular, the prescribed velocity \(U\) can be computed for all times directly from the initial data, and will be a linear function of \(t\) since trajectories undergo constant acceleration between collisions.
  It is also of interest to study the distinction between the repulsive (\(\alpha<0\)) and non-repulsive (\(\alpha \ge 0\)) cases, since the global behavior of solutions is quite different in the two cases, as emphasized in \cite{brenier2013sticky}.
  
  Although the particle dynamics and their connection to the conservation law \eqref{eq:claw} for \(\alpha=0\) and \(\alpha \ge 0\) are respectively treated in \cite{brenier1998sticky} and \cite{nguyen2008pressureless}, we recall the key points here for completeness and ease of comparison.
  
  \subsection{Discrete setup}
  Recalling the setting of Section \ref{s:motivation}, we have \(n\) initially distinct particles located at positions \(x_i\) with velocities \(v_i\) and masses \(m_i\).
  From their masses we can define a partition of \(\Omega = (0,1)\) given by \(0 = \theta_0 < \theta_1 < \dots < \theta_{n-1} < \theta_n = 1\) where 
  \begin{equation}\label{eq:thetai}
    \theta_0 = 0, \qquad \theta_i = \sum_{j=1}^{i}m_j, \quad i \in \{1,\dots,n\}.
  \end{equation}
  Based on \eqref{eq:XV} we can define the functions
  \begin{equation}\label{eq:XVdisc}
    X(m) \coloneqq \sum_{i=1}^{n} x_i \chi_{[\theta_{i-1}, \theta_i)}(m), \qquad V(m) \coloneqq \sum_{i=1}^{n} v_i \chi_{[\theta_{i-1}, \theta_i)}(m),
  \end{equation}
  the finite-dimensional Hilbert space
  \begin{equation*}
    \Hscr_\mathbf{m} \coloneqq \left\{ X(m) = \sum_{i=1}^{n} x_i \chi_{[\theta_{i-1}, \theta_i)}(m) \colon \mathbf{x} = (x_1,\dots,x_n) \in \R^n \right\} \subset \Lp{2}(\Omega),
  \end{equation*}
  and the convex cone
  \begin{equation*}
    \Kscr_\mathbf{m} \coloneqq \left\{ X(m) = \sum_{i=1}^{n} x_i \chi_{[\theta_{i-1}, \theta_i)}(m) \colon \mathbf{x} = (x_1,\dots,x_n) \in \K^n \right\} \subset \Kscr \subset \Lp{2}(\Omega),
  \end{equation*}
  where we recall \(\K^n\) from \eqref{eq:cone:part}, so that \(X \in \Kscr_{\mathbf{m}} \subset \Kscr\) and \(V \in \Hscr_X\).
  Here \(X\) and \(V\) provide us with particle solutions as in \eqref{eq:particlesols}, since
  \begin{equation*}
    \rho = X_\#\mathfrak{m} = \sum_{i=1}^n m_i \delta_{x_i}, \qquad V = v \circ X, \qquad \rho v = \sum_{i=1}^n m_i v_i \delta_{x_i}.
  \end{equation*}
  Letting \(H(x)\) be the right-continuous Heaviside function, that is,
  \begin{equation*}
    H(x) = \begin{cases}
      1, & x \ge 0 \\ 0, & x < 0,
    \end{cases}
  \end{equation*}
  we can integrate the measures above to obtain discrete counterparts of \eqref{eq:MQ:def}, namely
  \begin{equation*}
    M(x) = \sum_{i=1}^{n}m_i H(x-x_i), \qquad Q(x) = \sum_{i=1}^{n}m_i v_i H(x-x_i).
  \end{equation*}
  Note that \(M\) above only takes values in the set \(\{\theta_i\}_{i=1}^n\) defined in \eqref{eq:thetai}.
  We recall that in order to define the scalar conservation law \eqref{eq:claw} we needed to find \(\Vcal^0\) such that \(Q = \Vcal^0(M)\), which turned out to be the primitive of \(V\).
  For the ``acceleration''-part \(\Acal(M)\) we need to compute the square of \(M\), and to do this we conveniently rewrite it as a sum of step functions with disjoint supports,
  \begin{equation*}
    M(x) = \sum_{i=1}^n \theta_i \chi_{[x_i,x_{i+1})}(x),
  \end{equation*}
  where we for further convenience define \(x_{n+1} \equiv \infty\).
  Then it is clear that
  \begin{equation*}
    M(x)^2 = \sum_{i=1}^n \theta_i^2 \chi_{[x_i,x_{i+1})}(x) = \sum_{i=1}^n (\theta_i^2-\theta_{i-1}^2) H(x-x_i) = \sum_{i=1}^n (\theta_i+\theta_{i-1}) m_i H(x-x_i).
  \end{equation*}
  This shows that \(-\frac{\alpha}{2}(M(t,x)^2-M(t,x))\) can be written as \(\int_0^{M(t,x)}A_n(m)\dee m\) for the piecewise constant function
  \begin{equation}\label{eq:PWacc}
    A_n(m) = \sum_{i=1}^n a_i \chi_{[\theta_{i-1}, \theta_i)}(m), \qquad a_i \coloneqq -\frac{\alpha}{2}\left(\theta_i + \theta_{i-1} -1\right).
  \end{equation}
  Note that \(\Acal_n(m) = \int_0^m A_n(\omega)\dee \omega\) is exactly the piecewise linear interpolation of the quadratic function \(\Acal(m) = -\frac{\alpha}{2}(m^2-m)\) on \([0,1]\) with breakpoints \(\{\theta_i\}_{i=1}^n\).
  Moreover, \eqref{eq:PWacc} is exactly the \(\Lp{2}\)-projection of \(F[X](m)= -\frac{\alpha}{2}(2m-1)\) onto \(\Hscr_\mathbf{m}\) as defined in \cite[Section 5]{brenier2013sticky}, namely
  \begin{equation*}
    F_{\mathbf{m}}[X] \coloneqq \Proj_{\Hscr_\mathbf{m}}F[X] = \sum_{i=1}^{n}a_{\mathbf{m},i} \chi_{[\theta_{i-1}, \theta_i)}(m), \quad a_{\mathbf{m},i} = \frac{1}{m_i} \int_{\theta_{i-1}}^{\theta_i}F[X](m)\dee m = -\frac{\alpha}{2}\left(\theta_i + \theta_{i-1} -1\right),
  \end{equation*}
  where we recognize \(a_{\mathbf{m},i} = A_n(\theta_{i-1}) = a_i\), the acceleration for our particles in \eqref{eq:PWacc}.
  This projected force satisfies \eqref{eq:f&F};
  indeed, recall from the derivation of \(\Acal\) in the continuous flux function \eqref{eq:flux:EP} that \(f[\rho] = -\frac{\alpha}{2}(M(x) + M(x-) -1)\rho\), and so we compute using \eqref{eq:XVdisc} and \eqref{eq:PWacc}
  \begin{align*}
    \int_\R \phiv(x) f[\rho](\dee x) &= -\frac{\alpha}{2}\int_\R \phiv(x) (M(x)+M(x-)-1) \rho\dee x \\
    &= -\frac{\alpha}{2}\sum_{i=1}^n \phiv(x_i) (\theta_i+\theta_{i-1}-1)(\theta_i-\theta_{i-1}) \\
    &= \sum_{i=1}^n \int_{\theta_{i-1}}^{\theta_i}\phiv(X(m))A_n(m)\dee m = \int_0^1 \phiv(X(m))A_n(m)\dee m.
  \end{align*}
  To summarize, from either point of view the acceleration of the \(i\)\textsuperscript{th} particle is \(a_i\) before collision, and we have justified the initial dynamics \eqref{eq:EPparticle}, which can be restated as the system
  \begin{equation*}
    \dot{x}_i = v_i, \qquad
    \dot{v}_i = -\frac{\alpha}{2}\sum_{j=1, j \neq i}^{n}m_j \sgn(x_i-x_j),
  \end{equation*}
  being the particle analogue of \eqref{eq:EP}.
  
  At this stage, with the initial particle evolution defined, we still need to derive the corresponding conservation law \eqref{eq:claw} for this discrete setting.
  From now on, to emphasize that we are considering particle solutions, we will add the subscript \(n\) to both the Lagrangian variables \(X\) and \(V\), the distribution functions \(M\) and \(Q\), and the primitive \(\Vcal^0\).
  Moreover, as the position \(t \mapsto x_i(t)\) and velocity \(t \mapsto v_i(t)\) of the \(i\)\textsuperscript{th} particle depend on time, so will the corresponding \(X_n(t,m)\), \(V_n(t,m)\), \(M_n(t,x)\) and \(Q_n(t,x)\).
  
  Recall from Section \ref{s:motivation} that the prescribed velocities \(u_i(t) = v_i^0 + t a_i\) played a key role in defining the dynamics at collision time, and so we define
  \begin{equation}\label{eq:U:EP:disc}
    U_n(t, m) \coloneqq \sum_{i=1}^n u_i(t) \chi_{[\theta_{i-1}, \theta_i)}(m) = V_n^0(m) + t A_n(m)
  \end{equation}
  inspired by the definition of \(V_n\).
  This is also the discrete counterpart of \eqref{eq:U:EP},
  and from the same considerations as in Section \ref{s:entropy} we find that the flux function for the scalar conservation law is the primitive \(\Ucal_n\) of \eqref{eq:U:EP:disc}, i.e.,
  \begin{equation}\label{eq:flux:EP:disc}
    \Ucal_n(t,m) = \int_0^m U_n(t,\omega)\dee \omega = \Vcal_n^0(m) + t \Acal_n(m),
  \end{equation}
  which is exactly the piecewise linear and continuous interpolation of the flux function \eqref{eq:flux:EP} with breakpoints \(\{\theta_i\}_{i=1}^n\).
  In summary, the scalar conservation law for the particle dynamics is
  \begin{equation}\label{eq:claw:disc}
    \del_t M_n + \del_x \Ucal_n(t,M_n) = 0,
  \end{equation}
  with flux function \(\Ucal_n\) defined in \eqref{eq:flux:EP:disc}.
  
  From the point of view of Lagrangian solutions, we have from our definition of the prescribed velocity that \(\dot{U}_n = F_{\mathbf{m}}[X_n]\), and so, for \(X_n\) to be a Lagrangian solution we require \(U_n - \dot{X}_n \in N_{X_n}\Kscr_{\mathbf{m}}\).
  
  \subsection{Dynamics prior to first collision}
  Here we show that the initial particle dynamics correspond to both the Lagrangian solutions and the entropy solutions of \eqref{eq:claw:disc}.
  
  Before any collision has happened, the particles \(\mathbf{x}\) are located within the interior of the cone \(\K^n\), and consequently the projection of the prescribed velocity onto the tangent cone is simply the identity operator, that is, \(\Proj_{T_{\mathbf{x}}\K^n} \mathbf{u} = \mathbf{u}\).
  Therefore the initial particle trajectories \(\mathbf{x}(t)\) are uniquely determined as the ``free'' trajectories \(\mathbf{y}(t) = \mathbf{x}^0 + t \mathbf{v}^0 + \frac12 t^2 \mathbf{a}\).
  In turn, \(X_n\) is uniquely defined as \(X_n = Y_n\), where
  \begin{equation*}
    Y_n(t,m) \coloneqq \sum_{i=1}^n y_i(t) \chi_{[\theta_{i-1}, \theta_i)}(m), \qquad y_i(t) = x_i^0 + t v_i^0 + \frac12 t^2 a_i.
  \end{equation*}
  Note that \(A_n = F_\mathbf{m}[X]\) is trivially Lipschitz-continuous as a function \(F_\mathbf{m}[X] \colon \Kscr_\mathbf{m} \to \Hscr_\mathbf{m}\) since it does not depend on \(X_n\), and so \((X_n,V_n) = (Y_n, U_n)\) is the unique Lagrangian solution from Theorem \ref{thm:LagSol:EU}.
  
  From the point of view of conservation laws, we see that the particles, through their representation by the piecewise constant mass distribution \(M_n\) and the piecewise linear and continuous flux function \(\Ucal_n\), correspond exactly to Dafermos' \cite{dafermos1972polygonal}, or the front tracking \cite{bressan2000hyperbolic,holden2015front}, method for scalar conservation laws, at least initially.
  This correspondence is exactly what was used to provide ``sticky particle'' solutions of the Euler system in \cite{brenier1998sticky}.
  Since \(M_n(t,x)\) is piecewise constant, it will be a weak solution of \eqref{eq:claw:disc}, the discrete counterpart of the scalar conservation law \eqref{eq:claw}, provided it satisfies the Rankine--Hugoniot condition \eqref{eq:RH} along the discontinuities \(x_i(t)\).
  Moreover, it will be an entropy solution if it additionally satisfies the Ole\u{\i}nik E-condition \eqref{eq:Oleinik} along the discontinuities, and in particular, it will be the unique \(BV\) solution of this conservation law for \(t \in [0,T]\).
  This is indeed the case, since along the \(i\)\textsuperscript{th} trajectory we have a left and right state, \(\theta_{i-1}\) and \(\theta_i\) respectively, such that \([[M_n]] = \theta_i - \theta_{i-1} = m_i\) and \([[\Ucal_n(t,M_n)]] = m_i (v_i^0 + ta_i) = m_i u_i(t)\).
  Therefore, the Rankine--Hugoniot condition \eqref{eq:RH} is satisfied since
  \begin{equation}\label{eq:RH:single}
    \dot{x}_i(t) = u_i(t) = \frac{[[\Ucal_n(t,M_n(t,x_i(t)))]]}{[[M_n(t,x_i(t))]]} = \frac{\Ucal_n(t,\theta_i)-\Ucal_n(t,\theta_{i-1})}{\theta_i-\theta_{i-1}},
  \end{equation}
  and because \(M_n\) cannot take any other values between \(\theta_{i-1}\) and \(\theta_i\), the Ole\u{\i}nik E-condition \eqref{eq:Oleinik} is trivially satisfied.
  Note that since \(U_n - \dot{X}_n = 0 \in N_{X_n}\Kscr_{\mathbf{m}}\), this is also the unique Lagrangian solution. 
  
  \subsection{Collisions in the non-repulsive and repulsive case}
  Now that we have seen that the two notions of solution are equivalent for the initial particle dynamics, it remains to show that this is still true when particles meet.
  We shall treat the non-repulsive (\(\alpha \ge 0\)) and repulsive (\(\alpha < 0\)) cases separately, as there are major differences in the collision dynamics in these two cases;
  in fact, as emphasized in \cite{brenier2013sticky}, \(\alpha \ge 0\) corresponds to a sticking functional \(F[X]\), while for \(\alpha < 0\) the functional does not satisfy the sticking condition of Definition \ref{dfn:sticking}.
  
  To begin with we will establish some convenient notation and terminology.
  Since the particles are ordered according to their positions, we see that the particles on the left must necessarily have velocities larger than or equal to the particles on the right before touching, or else they would not have met.
  To be precise, we define the quantities
  \begin{equation*}
    \Jcal_i(\tau) \coloneqq \{j\in\{1,\dots,n\} \colon x_j(\tau) = x_i(\tau)\}, \quad i_*(\tau) \coloneqq \min\Jcal_i(\tau), \quad i^*(\tau) \coloneqq \max\Jcal_i(\tau).
  \end{equation*}
  That is, \(\Jcal_i(\tau)\) is the set indices of particles located at \(x=x_i(\tau)\) at time \(t=\tau\), while \(i_*(\tau)\) and \(i^*(\tau)\) are respectively the smallest and greatest index in this set.
  From our previous reasoning, at a time of collision \(t=\tau\) we must necessarily have the chain of inequalities
  \begin{equation}\label{eq:chain}
    v_{i_*(\tau)}(\tau-) \ge v_{i_*(\tau)+1}(\tau-) \ge \cdots \ge v_{i^*(\tau)-1}(\tau-) \ge v_{i^*(\tau)}(\tau-).
  \end{equation}
  These are exactly the relations that give rise to the \textit{barycentric lemma} used in \cite{brenier1998sticky,nguyen2008pressureless} to connect the momentum conservation of particles to the Ole\u{\i}nik E-condition for the scalar balance law.
  
  For any \(\alpha \in \R\), the \(i\)\textsuperscript{th} particle initially moves along its free trajectory
  \(y_i(t) = x_i^0 + t v_i^0 + \frac12t^2 a_{i}\) as long as one remains in the interior of the cone \(\K^n\).
  Now assume that some subset of particles \(\Jcal_i(\tau) \ni i\) collide at time \(\tau\).
  Recalling the set \eqref{eq:OX} and the definition of \(X_n\) in \eqref{eq:XVdisc}, we see that \((\theta_{i_*(\tau)-1}, \theta_{i^*(\tau)})\) is a maximal interval of \(\Omega_{X_n(\tau)}\).
  For this set of particles, the Rankine--Hugoniot condition reads
  \begin{equation}\label{eq:RH:part}
    v_i(t) = \frac{\Ucal_n(t,\theta_{i^*(\tau)})-\Ucal_n(t,\theta_{i_*(\tau)-1})}{\theta_{i^*(\tau)}-\theta_{i_*(\tau)-1}} = \frac{\sum_{j\in\Jcal_i(\tau)}m_j u_j(t)}{\sum_{j\in\Jcal_i(\tau)}m_j}.
  \end{equation}
  On the other hand, suppose \(k \in (\theta_{l-1}, \theta_l]\) for some \(l \in \Jcal_i(\tau)\), then the expressions appearing in the Ole\u{\i}nik E-condition \eqref{eq:Oleinik} are given by
  \begin{equation*}
    \frac{\Ucal_n(t,k)-\Ucal_n(t,\theta_{i_*(\tau)-1})}{k-\theta_{i_*(\tau)-1}} = \frac{\sum_{j=i_*(\tau)}^{l-1}m_j u_j(t) + (k-\theta_{l-1})u_l(t)}{\sum_{j=i_*(\tau)}^{l-1}m_j + k-\theta_{l-1}}
  \end{equation*}
  and
  \begin{equation*}
    \frac{\Ucal_n(t,\theta_{i^*(\tau)})-\Ucal_n(t,k)}{\theta_{i^*(\tau)}-k} = \frac{\sum_{j=l+1}^{i^*(\tau)} m_j u_j(t) + (\theta_{l}-k)u_l(t)}{\sum_{j=l+1}^{i^*(\tau)} m_j + \theta_{l}-k}.
  \end{equation*}
  We observe that these expressions attain their extremal values at either \(k = \theta_{l-1}\) or \(k = \theta_l\) depending on the value of \(u_l\), and so it is clear that it is sufficient for us to verify \eqref{eq:Oleinik} for \(k = \theta_l\) for all \(l\in\Jcal_i(\tau)\).
  Moreover, we recall Lemma \ref{lem:ProjHX&O}, which says that \(U_n - \Proj_{\Hscr_{X_n}}U_n \in N_{X_n}\Kscr_{\mathbf{m}}\) is equivalent to the Ole\u{\i}nik E-condition.
  From this it follows that, for these particle dynamics, the Lagrangian and entropy solutions coincide.
  
  Since the particles in \(\Jcal_i(\tau)\) collided in the first place, we know that their free velocities \(u_j(\tau) = v_j(\tau-)\) for \(j \in \Jcal_i(\tau)\) satisfy the ordering \eqref{eq:chain}, meaning \(\theta \mapsto \Ucal_n(t,\theta)\) is concave on \((\theta_{i_*(\tau)-1}, \theta_{i^*(\tau)})\) at the collision time \(t=\tau\).
  This implies that the Ole\u{\i}nik E-condition is satisfied at \(t=\tau\), but what happens for \(t > \tau\) is highly dependent on the sign of \(\alpha\), i.e., if we are in the repulsive or non-repulsive case.
  Note that if one or more of the inequalities in \eqref{eq:chain} are strict, there is a loss of kinetic energy at impact due to Jensen's inequality and the convexity of \(v \mapsto v^2\).
  
  \subsubsection{The non-repulsive case}
  The important point in this case is that \(\Acal_n(\cdot)\) in \eqref{eq:flux:EP:disc} is strictly concave on \([0,1]\) for \(\alpha > 0\), or vanishes for \(\alpha =0\).
  We already established that \(\theta \mapsto \Ucal_n(\tau,\theta)\) is concave on \( (\theta_{i_*(\tau)-1}, \theta_{i^*(\tau)}) \), and so, since \(\theta \mapsto (t-\tau)\Acal_n(\theta)\) is concave for \(t > \tau\), it follows that \(\Ucal_n(t,\theta) = \Ucal_n(\tau,\theta) + (t-\tau)\Acal_n(\theta)\) remains concave on this interval.
  Therefore the Ole\u{\i}nik E-condition remains satisfied, and the entropic solution is to let the particles continue as a single amassed particle with velocity given by the Rankine--Hugoniot condition \eqref{eq:RH:part}.
  
  Equivalently, since \(\alpha \ge 0\) we have that the accelerations \(a_i\) from \eqref{eq:PWacc} are decreasing in \(i\), which combined with \eqref{eq:chain} for \(u_j(\tau) = v_j(\tau-)\), \(j \in \Jcal_i(\tau)\) means that this chain of inequalities remains valid for \(u_j(t)\), \(t > \tau\).
  If \(\alpha > 0\), the inequalities even become strict, so that \(\mathbf{u}(t) \notin T_{\mathbf{x}(t)}\K^n\) for all \(t > \tau\).
  Therefore, from here on the particle velocities will be given by the projection \(\mathbf{v}(t) = \Proj_{T_{\mathbf{x}(t)}\K^n}\mathbf{u}(t)\).
  The same is true when \(\alpha=0\), as there is no acceleration and the particles will continue moving with the mass-averaged velocity given by momentum conservation, or the Rankine--Hugoniot condition, at impact. 
  
  Therefore, in the non-repulsive case, once particles collide, they stick together for all time, exhibiting globally sticky behavior.
  This is in agreement with \(F_{\mathbf{m}}[X] = A_n \) for \(\alpha \ge 0\) being a sticking functional in the sense of Definition \ref{dfn:sticking}.
  In particular, for an attractive potential \(\alpha > 0\) there is some time \(T\) such that for \(t > T\) all particles have been lumped together, moving along the trajectory of the center of mass, which is unaffected by the internal forces and collision dynamics.
  Hence we recover exactly the sticky dynamics already detailed in \cite{nguyen2008pressureless} and \cite{brenier2013sticky}.
  As we will see next, the behavior for a repulsive potential, i.e., \(\alpha < 0\), is less straightforward.
  
  \subsubsection{The repulsive case}
  In this case \(\Acal_n\) is strictly convex since \(\alpha < 0\), which will change the dynamics, as the flux function will not remain concave on the jump set for all time after a collision.
  
  For instance, suppose that at the time of impact \(\tau\), \eqref{eq:chain} holds with all inequalities replaced by equality, that is, the particles meet tangentially.
  Then \(\Ucal(\tau,\theta)\) is linear for \(\theta \in (\theta_{i_*(\tau)-1}, \theta_{i^*(\tau)}) \), and for \(t > \tau\) the flux function \(\Ucal_n(t,\theta)\) will immediately turn strictly convex on \( (\theta_{i_*(\tau)-1}, \theta_{i^*(\tau)}) \), violating the Ole\u{\i}nik E-condition there, except on the ``trivial'' intervals \( (\theta_{i-1}, \theta_i) \), \(i \in \Jcal_i(\tau)\) corresponding to single particles.
  The entropy solution is therefore to have the particles immediately split apart and follow the free trajectories defined by \eqref{eq:RH:single}, and one can think of the particles as only grazing one another at \(t=\tau\).
  
  On the other hand, if \(\Ucal_n(\tau,\cdot)\) is strictly concave on \( (\theta_{i_*(\tau)-1}, \theta_{i^*(\tau)}) \), it will take some time before the Ole\u{\i}nik E-condition is violated.
  This coincides with the lower convex envelope \(\theta \mapsto \Ucal_{n,\smallsmile}(t,\theta)\) on this interval changing from a chord connecting the endpoints of the graph of \(\Ucal_n(t,\theta)\) to breaking into two or more affine segments connected at breakpoints \(\theta_j\) for \(j \in \Jcal_i(\tau)\).
  This happens when there is either a left subgroup \(\{i_*(\tau),\dots,j\} \subsetneq \Jcal_i(\tau)\) where the average of their prescribed velocities \(u_j(t)\) dip below the (average) velocity \(v_i(t)\) of the amassed particles \(\Jcal_i(\tau)\), or alternatively, there is a right subgroup \(\{j,\dots,i^*(\tau)\} \subsetneq \Jcal_i(\tau)\) such that their averaged prescribed velocities surpass the velocity of the amassed particle.
  The entropy solution is then to have the amassed particle break into two or more lumped particles corresponding to the aforementioned breakpoints.
  Since this ``fission'' happens when the subgroups of particles have the same velocity as that of the amassed particle, the kinetic energy does not change as it breaks apart.
  This solution can then be thought of as ``physical'' in the sense that the energy is nonincreasing.
  
  
  Since we now know that we can have amassed particles in the repulsive case as well, if only for a limited time, we observe that similar results hold for collisions between aggregated particles as well.
  Two amassed particles can collide and stick together for some limited time before breaking apart into new subgroups.
  All of this is governed by the flux function \(\Ucal_n(t,\theta)\).
  For this repulsive case \(\alpha < 0\), we see that there must be some time \(T\) such that the flux function is strictly convex for \(t > T\) and there are no more collisions between particles, which then move freely with velocities \(v_i(t) = u_i(t)\).
  
  In the introduction of \cite{brenier2013sticky}, the Lagrangian selection principle is motivated by an inequality involving the collision of two particles.
  Here we will give a concrete such example, to illustrate the ideas. 
  \begin{exm}[Two repulsive particles]\label{exm:2particles}
    Consider the case \(\alpha = -2\), \(\beta = 0\) and two particles where \(m_1 = m_2 = \frac12\), \(x_1^0 = 0\), \(v_1^0 = 2\), \(x_2^0 = 1\) and \(v_2^0 = 0\) such that \(a_1 = -\frac12\) and \(a_2 = \frac12\).
    Clearly, the free trajectories \(y_i\) and prescribed velocities \(u_i\) for \(i =1,2\) are given by
    \begin{equation*}
      y_1(t) = 2t -\frac{t^2}{4}, \quad u_1(t) = 2 - \frac{t}{2}, \qquad y_2(t) = 1 +\frac{t^2}{4}, \quad u_2(t) = \frac{t}{2}.
    \end{equation*}
    The undisturbed trajectories intersect at two points, given by the times \(t = 2\pm\sqrt{2}\).
    One way of obtaining a generalized Lagrangian solution is to project the undisturbed trajectories onto the cone \(\K^2\), and writing \(\tilde{\mathbf{x}} = \Proj_{\K^2}\mathbf{y}\) we then obtain
    \begin{equation*}
      \tilde{x}_1(t) = \begin{cases}
        2t -\frac{t^2}{4}, & \abs{t-2} \ge \sqrt{2}, \\
        \frac12 + t, & \abs{t-2} < \sqrt{2},
      \end{cases} \qquad \tilde{x}_2(t) = \begin{cases}
        1 +\frac{t^2}{4}, & \abs{t-2} \ge \sqrt{2}, \\
        \frac12 + t, & \abs{t-2} < \sqrt{2}.
      \end{cases}
    \end{equation*}
    That is, this generalized Lagrangian solution is given by letting the particles stick together as long as \(\mathbf{y} \notin \K^2\), i.e., \(y_1(t) > y_2(t)\), and otherwise they follow the free trajectories.
    Note that these trajectories correspond to the weak solution of the repulsive Euler--Poisson system found in the similar example \cite[Example 3]{tadmor2011variational}.
    On the other hand, the Lagrangian, or entropic, trajectories are given by letting the particles stick together as long as the prescribed velocity \(\mathbf{u} \notin T_{\mathbf{x}(t)}\K^2\), i.e., \(u_1(t) > u_2(t)\). This means that the ``right'' time to split is at \(t = 2\), and the corresponding trajectories are given by
    \begin{equation*}
      x_1(t) = \begin{cases}
        2t -\frac{t^2}{4}, & t < 2 -\sqrt{2}, \\
        \frac12 + t, & 2- \sqrt{2} \le t < 2, \\
        \frac12 + t -\frac{(t-2)^2}{4}, & t \ge 2,
      \end{cases} \qquad x_2(t) = \begin{cases}
        1 +\frac{t^2}{4}, & t < 2 -\sqrt{2}, \\
        \frac12 + t, & 2- \sqrt{2} \le t < 2, \\
        \frac12 + t +\frac{(t-2)^2}{4}, & t \ge 2.
      \end{cases}
    \end{equation*}
    
    \begin{figure}
      \begin{subfigure}[b]{0.495\linewidth}
        \includegraphics[width=1\linewidth]{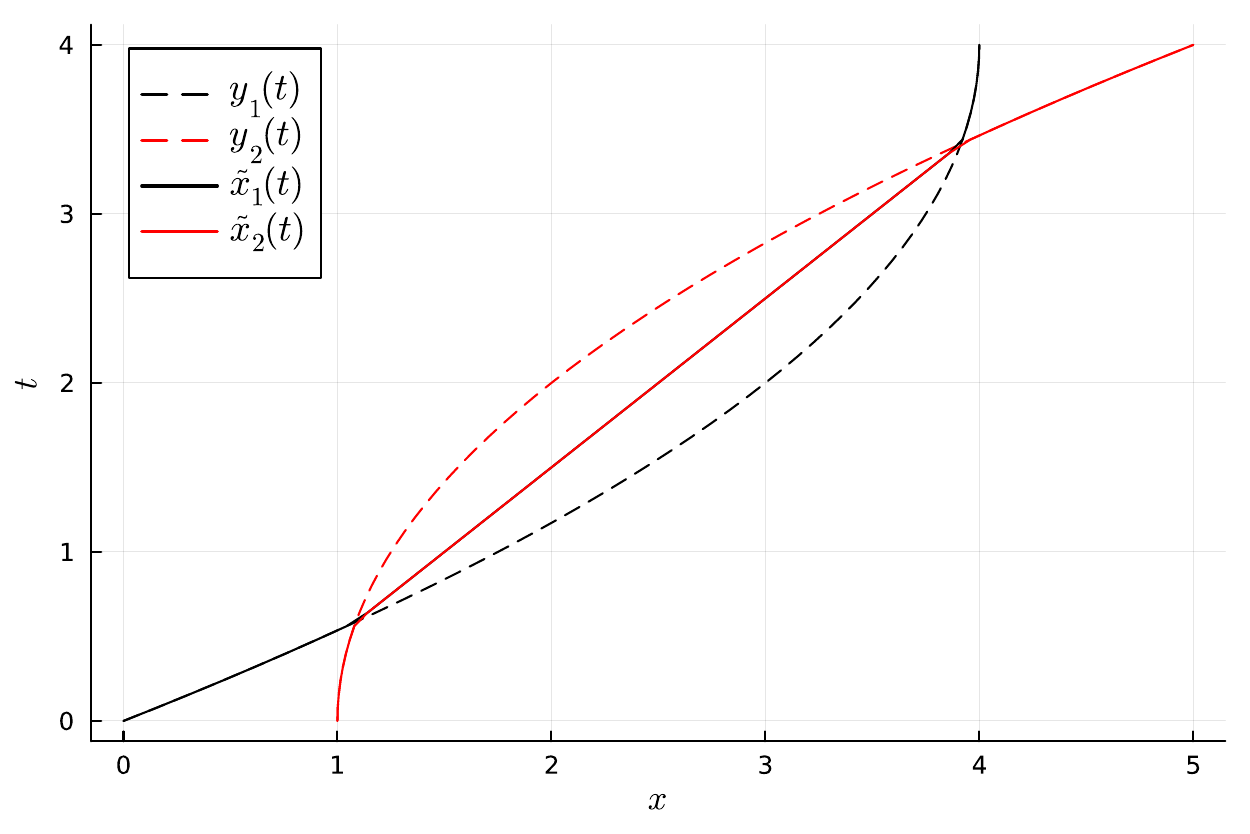}
        \caption{}
      \end{subfigure}
      \begin{subfigure}[b]{0.495\linewidth}
        \includegraphics[width=1\linewidth]{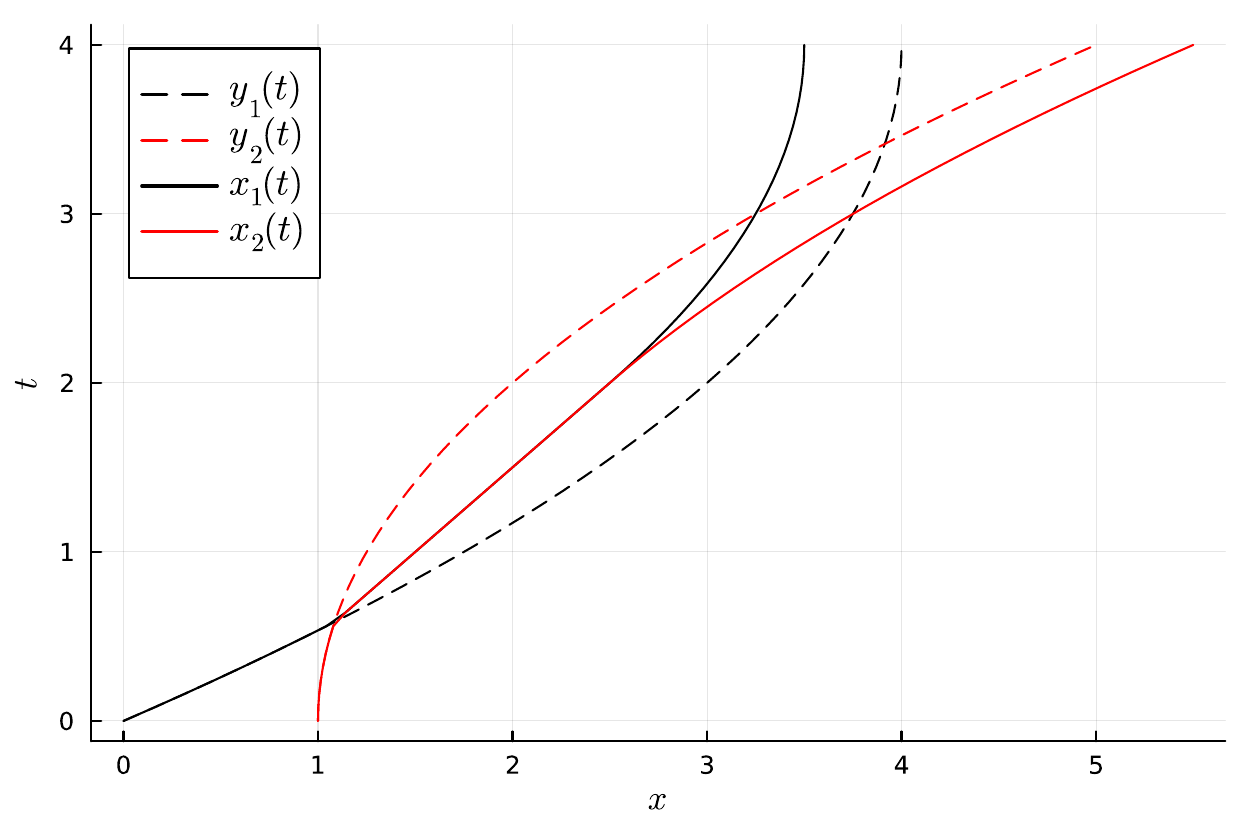}
        \caption{}
      \end{subfigure}
      \caption{Free trajectories \(y_i\) compared with \textsc{(a)} projected trajectories \(\tilde{x}_i = (\Proj_{\K^2}\mathbf{y})_i\) and \textsc{(b)} Lagrangian, or entropy, trajectories \(x_i(t)\) for \(i \in \{1,2\}\).}
      \label{fig:2particles}
    \end{figure}
    
    Figure \ref{fig:2particles} highlights the differences between the two solutions.
    In both cases, when the particles collide at time \(t = 2-\sqrt{2}\), the velocity is discontinuous, as one expects from the conservation of momentum during a collision, and Jensen's inequality shows that the kinetic energy decreases on impact.
    However, there is an important difference in the splitting of the particles.
    For \(x_i\) the velocity is continuous at \(t=2\) when the particles split, which ensures that the kinetic energy \(\frac12(\dot{x}_1^2 + \dot{x}_2^2)\) is conserved at the time of splitting.
    On the other hand, for \(\tilde{x}_i\), the velocity is discontinuous also at the splitting time \(t = 2+\sqrt{2}\), which by Jensen's inequality shows that the kinetic energy increases at this time.
    This would then suggest that the physically relevant way of splitting the particles is given by the entropic solution.
    Of course, another solution for which the energy is nonincreasing is to let the particles remain stuck after collision, but this would not take into account the repulsive forces of the problem, and violates the Ole\u{\i}nik E-condition.
    Indeed, for this amassed particle, the inequalities \eqref{eq:Oleinik} are equivalent to \(u_2(t) \le \dot{x}_i \equiv 1 \le u_1(t)\), which is violated for \(t > 2\).
  \end{exm}
  
  \begin{rem}[Numerical algorithm]
    Since one can use empirical measures for approximations, cf.\ \cite[Section 4]{bonaschi2015equivalence}, one could approximate general initial data for the Euler--Poisson system \eqref{eq:EP} and evolve it according to the particle dynamics described in this section, i.e., with particles following parabolic trajectories to approximate solutions.
    In practice, one could then implement a front-tracking algorithm in the spirit of \cite{lie2000front}, where one considers time-varying velocity fields leading to nonaffine front trajectories.
    One could also imagine doing the same with the more general system \eqref{eq:EF}, using an ODE solver to compute the trajectories between collisions.
  \end{rem}
  
  \section{Repulsive Poisson Forces}\label{s:examples}
  In this section we want to use the equivalence we have established to study some more involved examples.
  The first of these expands upon an example from \cite{brenier2013sticky}, and can be seen as a more involved, continuous version of Example \ref{exm:2particles}, although we here consider symmetric data which simplifies the analysis and formulas.
  In the second example and its subcases we use the equivalence framework on the repulsive Euler--Poisson system with a quadratic confinement potential studied in \cite{carrillo2016pressureless} with the aim of providing a well-defined notion of solution beyond singularity formation.
  
  \subsection{Entropic and non-entropic solutions to the repulsive Euler--Poisson system} \label{ss:BGSW}
  Let us consider the example found in \cite[Section 6.2.2]{brenier2013sticky} where the authors consider the Euler--Poisson system \eqref{eq:EP} with \(\alpha=-2\) and \(\beta=0\), and initial data \(X^0(m) = m-\frac12\), \(V^0(m) = -\sgn\left(m-\frac12\right)\) for the differential inclusion \eqref{eq:DIU}.
  Note that this corresponds to the initial data \(\rho_0 = \chi_{[-\frac12, \frac12]}\) and \(v_0(x) = -\sgn(x)\) for \eqref{eq:EP}, which immediately leads to mass concentration, or collision of trajectories, at the origin.
  In the example, they use the representation formula \eqref{eq:sticky:X} for sticky Lagrangian solutions to compute a trajectory which fails to satisfy the conditions for a Lagrangian solution.
  
  We will use this example as a starting point to highlight the equivalence of Lagrangian solutions, as defined in Section \ref{s:gradflow}, and entropy solutions, as defined in Section \ref{s:entropy}, for the repulsive Euler--Poisson system.
  In particular, we will compute the Lagrangian solution \(X(t,m)\) and its generalized inverse \(M(t,x)\) and show that \(M\) is an entropy solution.
  Furthermore, we will show that the inverse \(\tilde{M}\) of the trajectory \(\tilde{X}\) computed from \eqref{eq:sticky:X}, although it is a weak solution, does not satisfy the entropy condition.
  
  \subsubsection{Computing the (non-)Lagrangian trajectories}
  According to \eqref{eq:U:EP}, the prescribed velocity is
  \begin{equation*}
    U(t,m) = -\sgn\left(m-\frac12\right) + t \left(2m-1\right),
  \end{equation*}
  from which we find the free trajectory
  \begin{equation*}
    Y(t,m) = (1+t^2)\left(m-\frac12\right) - t\sgn\left(m-\frac12\right).
  \end{equation*}
  To obtain an admissible, i.e., nondecreasing in \(m\), trajectory we can simply project \(Y\) onto the cone \(\Kscr\), as indicated by the projection formula \eqref{eq:sticky:X}; to this end we introduce the primitive
  \begin{equation*}
    \Ycal(t,m) = \frac12 (1+t^2)\left[\left(m-\frac12\right)^2 -\frac14 \right] +t \left[ \frac12 - \abs*{m-\frac12} \right],
  \end{equation*}
  and then according to \cite[Theorem 3.1]{natile2009wasserstein}, we have that the projection of \(Y\) onto \(\Kscr\) is \(\tilde{X} = \Proj_\Kscr Y = \diff{^+}{m}\Ycal^{**}\).
  To find the lower convex envelope \(\Ycal^{**}\) of \(\Ycal\), we note that \(\Ycal(0,m)\) is convex, while for fixed \(t > 0\), \(\Ycal(t,m)\) has critical points at \(m-\frac12 = \pm \frac{t}{1+t^2}\) and a singular point at \(m = \frac12\).
  The critical points are local minima, while the singular point is a local maximum, and from this we deduce that the lower convex envelope is
  \begin{equation*}
    \Ycal^{**}(t,m) = \begin{cases}
      \Ycal(t,m), & \abs{m-\frac12} \ge \frac{t}{1+t^2}, \\
      \Ycal\left(t,\frac{t}{1+t^2}\right), & \abs{m-\frac12} < \frac{t}{1+t^2},
    \end{cases}
  \end{equation*}
  from which we find the corresponding trajectory
  \begin{align*}
    \tilde{X}(t,m) &= \begin{cases}
      (1+t^2)(m-\frac12) - t\sgn(m-\frac12), & \abs{m-\frac12} \ge \frac{t}{1+t^2}, \\
      0, & \abs{m-\frac12} < \frac{t}{1+t^2}.
    \end{cases} \\
    &= \sgn\left(m-\frac12\right) \max\left\{ \left(1+t^2\right)\abs*{m-\frac12} - t, 0 \right\}
  \end{align*}
  For the Lagrangian, or entropy, solution we will have to instead project the prescribed velocity \(U\) onto the tangent cone \(T_X \Kscr\), characterized by \eqref{eq:TC2}.
  Note for a given \(t\), \(U(t,m)\) is decreasing for \(\abs{m-\frac12} < \frac{1}{2t}\) and nondecreasing otherwise.
  In particular, we see that it is decreasing on \([0,1]\) for \(t < 1\).
  On the other hand, from the trajectories \(Y\) we see that initially we have \(\Omega_{X(t)} = \left\{m \colon \abs{m} < \frac{t}{1+t^2} \right\}\), and so for \(t < 1\) we have
  \begin{equation*}
    \Proj_{T_{X(t)}\Kscr}U(t,m) = \begin{cases}
      U(t,m), & \abs{m-\frac12} \ge \frac{t}{1+t^2} \\
      0, & \abs{m-\frac12} < \frac{t}{1+t^2}.
    \end{cases}
  \end{equation*}
  Combining the above, we see that at \(t=1\), all mass is concentrated at the origin with zero velocity.
  For \(0 \le t < 1\), \(\Omega_{X(t)}\) was contained in the set where \(U\) is decreasing.
  For \(t > 1\) we have the opposite situation, and therefore
  \begin{equation*}
    \Proj_{T_{X(t)}\Kscr}U(t,m) = \begin{cases}
      U(t,m), & \abs{m-\frac12} \ge \frac{1}{2t}, \\
      0, & \abs{m-\frac12} < \frac{1}{2t}
    \end{cases}
  \end{equation*}
  in this case.
  Integrating in time we obtain the Lagrangian solution
  \begin{equation}\label{eq:BGSW:X}
    X(t,m) = \begin{cases}
      (1+t^2)(m-\frac12) - t\sgn(m-\frac12), & \frac{t}{1+t^2} \le \abs*{m-\frac12} \le \min\left\{\frac{1}{2t}, \frac12 \right\}, \\
      0, & \abs*{m-\frac12} < \min\left\{\frac{t}{1+t^2}, \frac{1}{2t} \right\}, \\
      \left(m-\frac12\right)\left(t-\frac{1}{\abs{2m-1}}\right)^2, & \abs*{m-\frac12} \ge \min\left\{\frac{1}{2t}, \frac12 \right\}.
    \end{cases}
  \end{equation}
  The generalized Lagrangian solution \(\tilde{X}\) and the Lagrangian solution \(X\) are compared in Figure \ref{fig:BGSW:X}, where we observe that the trajectories emanating from mass diffusion in \(X\) are tangential, in contrast to \(\tilde{X}\).
  Note that in the original example of \cite{brenier2013sticky}, they use the characterization of Lemma \ref{lem:NX} to show that \(U - \dot{\tilde{X}} \notin N_{\tilde{X}}\Kscr\), hence \(\tilde{X}\) is not a Lagrangian solution.
  In particular, they provide a counterexample to the equivalent condition
  \begin{equation}\label{eq:BGSW:cond}
    \pdiff{}{t} \left( \Ycal(t,m) - \Ycal^{**}(t,m) \right) \ge 0, \qquad m \in \Omega_{\tilde{X}(t)}.
  \end{equation}
  This is then again equivalent to the Ole\u{\i}nik E-condition \eqref{eq:Oleinik}.
  \begin{figure}
    \begin{subfigure}[t]{0.495\linewidth}
      \includegraphics[width=1\linewidth]{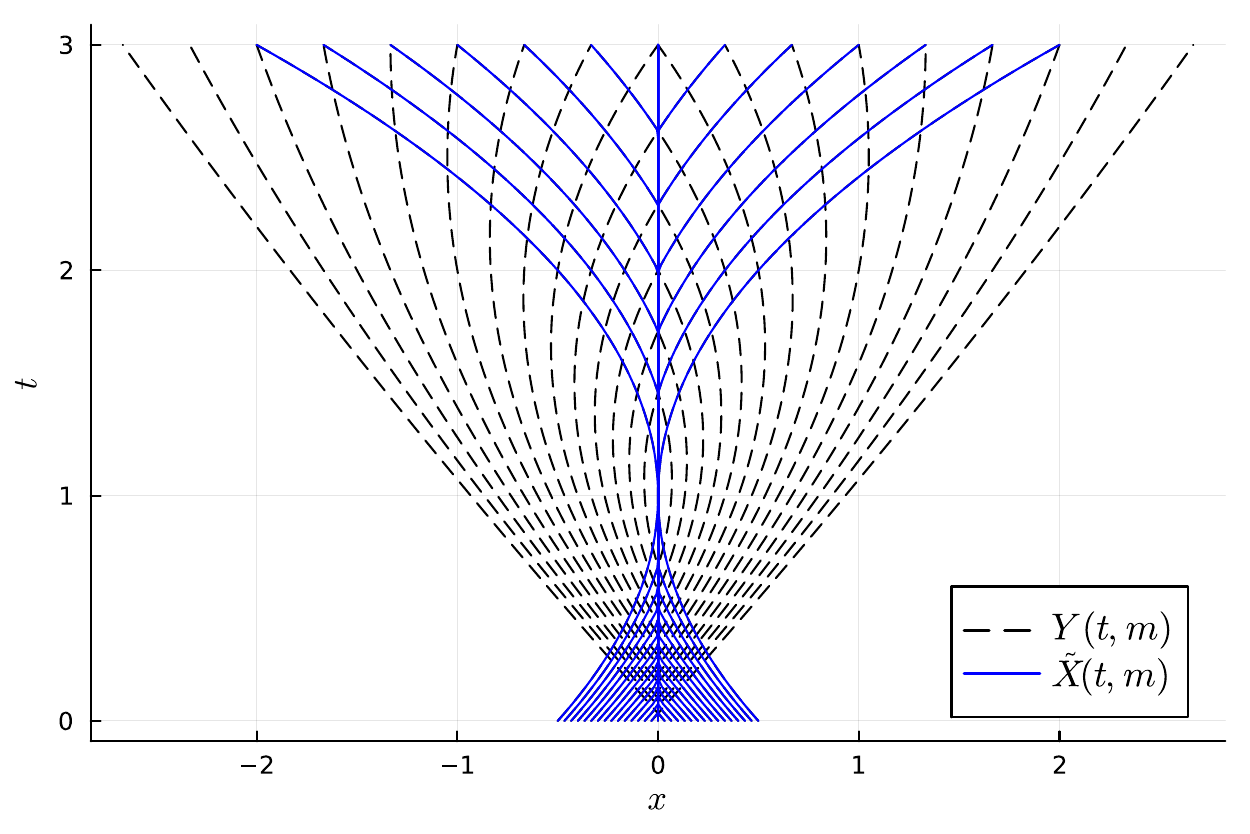}
      \caption{}
    \end{subfigure}
    \begin{subfigure}[t]{0.495\linewidth}
      \includegraphics[width=1\linewidth]{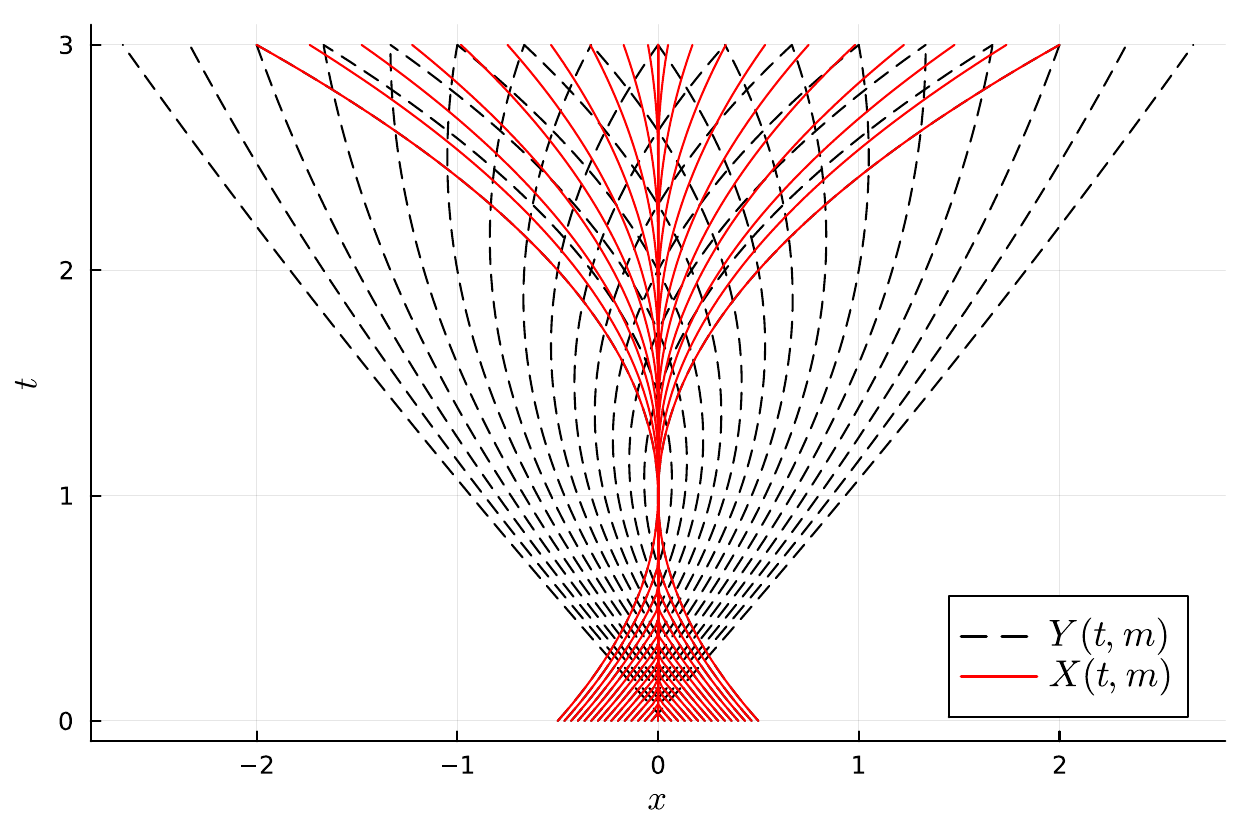}
      \caption{}
    \end{subfigure}
    \caption{The free trajectories \(Y(t,\theta_i)\) for \(\theta_i = \frac{i}{30}\), \(i \in \{0,1,\dots,30\}\), together with the (\textsc{a}) generalized Lagrangian solution \(\tilde{X}(t,\theta_i)\) and (\textsc{b}) Lagrangian solution \(X(t,\theta_i)\). Note that \(\tilde{X}(t,m) = X(t,m)\) for \(t \in [0,1]\).}
    \label{fig:BGSW:X}
  \end{figure}
  \begin{figure}
    \begin{subfigure}[t]{0.495\linewidth}
      \includegraphics[width=1\linewidth]{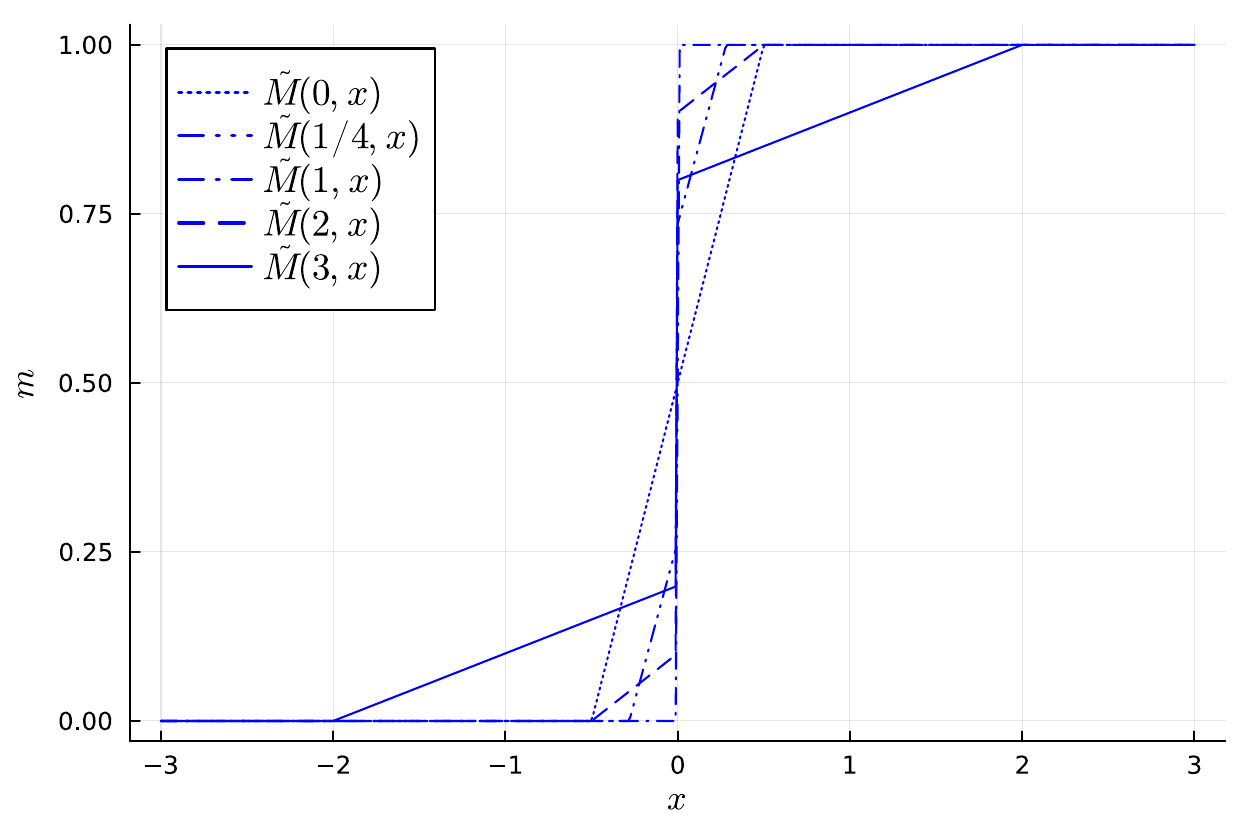}
      \caption{}
    \end{subfigure}
    \begin{subfigure}[t]{0.495\linewidth}
      \includegraphics[width=1\linewidth]{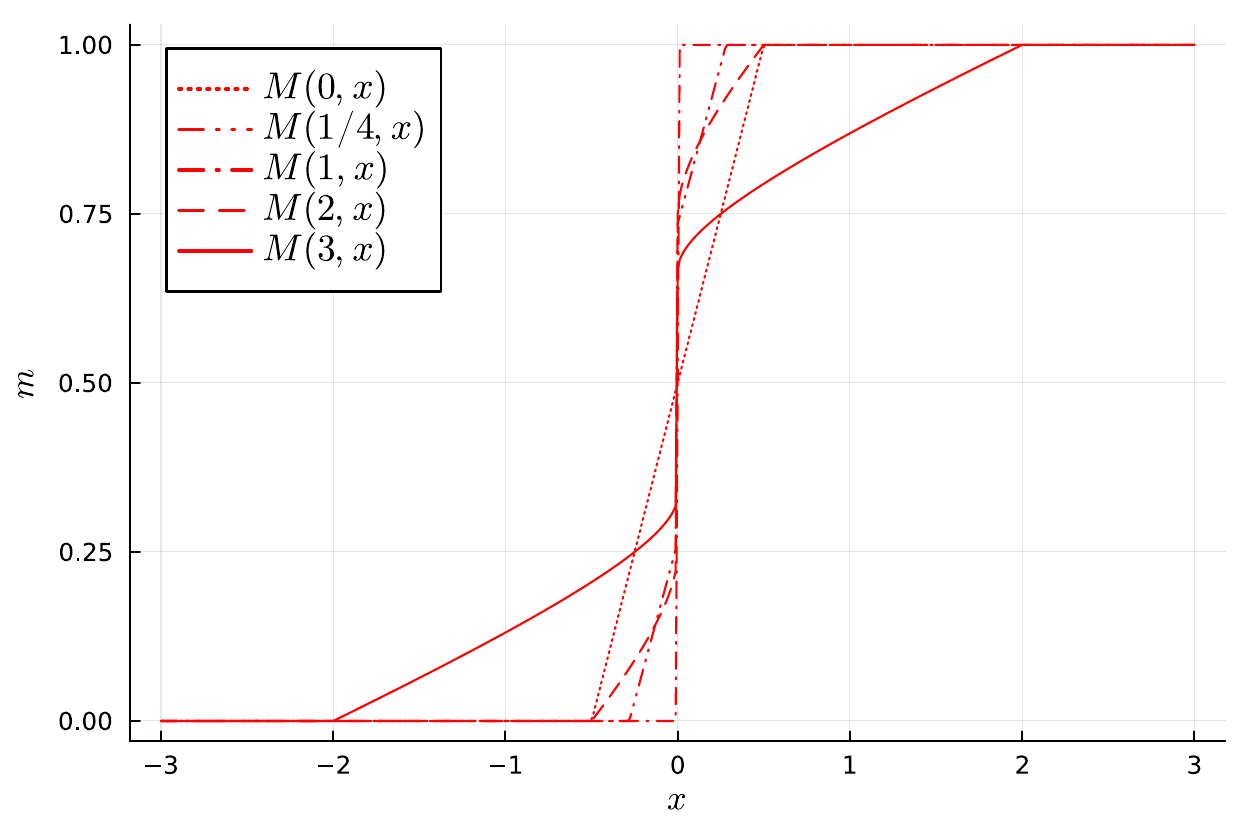}
      \caption{}
    \end{subfigure}
    \caption{The (\textsc{a}) non-entropic \(\tilde{M}(t,\cdot)\) and (\textsc{b}) entropic \(M(t,\cdot)\) weak solution for various times \(t\). These are the respective generalized inverses of \(\tilde{X}(t,\cdot)\) and \(X(t,\cdot)\) from Figure \ref{fig:BGSW:X}. Note that \(\tilde{M}(t,\cdot) = M(t,\cdot)\) for \(t \in [0,1]\).}
    \label{fig:BGSW:M}
  \end{figure}
  
  \subsubsection{Entropy admissibility}
  Next we will consider this problem from the point of view of the scalar conservation law \eqref{eq:claw}, which in this case has the flux function
  \begin{equation*}
    \Ucal(t,m) = t\left[\left(m-\frac12\right)^2 -\frac14 \right] + \frac12 - \abs*{m-\frac12}.
  \end{equation*}
  Some computations show that the corresponding right-continuous inverses of \(\tilde{X}\) and \(X\) are respectively
  \begin{equation*}
    \tilde{M}(t,x) = \begin{cases}
      \frac12 \left( 1 + \sgn(x) \right), & \abs{x} > \frac12(t-1)^2, \\
      \frac12 + \frac{x + t \sgn^{\!+}(x)}{1+t^2}, & \abs{x} \le \frac12(t-1)^2
    \end{cases}
  \end{equation*}
  and
  \begin{equation*}
    M(t,x) = \begin{cases}
      \tilde{M}(t,x), & t < 1, \\
      \begin{cases}
        \frac12 \left( 1 + \sgn(x) \right), & \abs{x} > \frac12(t-1)^2, \\
        \frac12 + \sgn^{\!+}(x)\left(\frac{\sqrt{\abs{x}}+\sqrt{2t+\abs{x}}}{2t}\right)^2, & \abs{x} \le \frac12(t-1)^2
      \end{cases}, & t \ge 1,
    \end{cases}
  \end{equation*}
  where we have introduced the right-continuous sign function \(\sgn^{\!+}(x)\), i.e., \(\sgn(x)\) with the altered convention \(\sgn^{\!+}(0) = 1\).
  Further computations, rather cumbersome in the case of \(M\), show that the generalized inverses are piecewise smooth solutions of \eqref{eq:claw}, with a jump discontinuity, or stationary shock, at the origin for \(t > 0\).
  Hence, to be weak and entropy solutions it is sufficient that they satisfy respectively the Rankine--Hugoniot condition and Ole\u{\i}nik E-condition at the shock.
  The symmetry of the flux function around \(m=\frac12\) yields \([[\Ucal(t,\tilde{M}(t,0))]] = 0 = [[\Ucal(t,M(t,0))]]\), so we see that both solutions satisfy the Rankine--Hugoniot condition at the stationary shock.
  
  In order to check the entropy conditions we have to study the convexity properties of the flux function \(\Ucal(t,m)\).
  To this end we note that \(\Ucal(0,m)\) is concave, and for a given \(t >0\), \(\Ucal(t,m)\) has critical points at \(t\abs{2m-1} = 1\) and a singular point at \(m = \frac12\).
  Again, the critical points are local minima and the singular point is a local maximum, showing that \(\Ucal(t,m)\) is convex for \(\abs{m-\frac12} \ge \frac{1}{2t}\) and concave for \(\abs{m-\frac12} \le \frac{1}{2t}\).
  
  Now we consider the sets in \([0,1]\) where \(\tilde{X}(t,\cdot)\) and \(X(t,\cdot)\) are constant, namely
  \begin{equation*}
    \Omega_{\tilde{X}(t)} = \left\{ m \colon \abs*{m-\frac12} < \frac{t}{1+t^2} \right\}, \qquad \Omega_{X(t)} = \left\{ m \colon \abs*{m-\frac12} < \min\left\{\frac{t}{1+t^2}, \frac{1}{2t} \right\} \right\}.
  \end{equation*}
  Observe that \(\Ucal(t,m)\) is always concave on the set \(\Omega_{X(t)}\), and since both this set and the flux function \(\Ucal(t,m)\) are symmetric around \(m=\frac12\), its lower convex envelope \(\Ucal_\smallsmile(t,m)\) on this set will be a horizontal chord connecting the endpoints.
  In consequence, we have \(\Ucal(t,m) > \Ucal_\smallsmile(t,m)\) for \(m \in \Omega_{X(t)}\), which in turn implies the Ole\u{\i}nik E-condition \eqref{eq:Oleinik}; compare this to the condition \eqref{eq:BGSW:cond}.
  On the other hand, on the set \(\Omega_{\tilde{X}(t)} \setminus \Omega_{X(t)}\), which becomes nonempty only for \(t > 1\), \(\Ucal(t,m)\) is convex.
  Therefore, the lower convex envelope on the set \(\Omega_{\tilde{X}(t)}\) will have a negative slope for \(-\frac{t}{1+t^2} < m < -\frac{1}{2t}\) and a positive slope for \(\frac{1}{2t}< m < \frac{t}{1+t^2}\), and this shows that the stationary shock of \(\tilde{M}\) violates \eqref{eq:Oleinik}.
  These two solutions are compared in Figure \ref{fig:BGSW:M}, where we observe, as seen from the formulas, that the shock decreases faster for the entropy solution \(M\).
  
  \subsubsection{Energy of trajectories}
  Recall that the energy for the system is the sum of the kinetic and potential energy, which for our Euler--Poisson system with \(\beta = 0\) reads
  \begin{equation*}
    \frac12 \int_\R v(t,x)^2\dee\rho(t,x) +\frac{\alpha}{4}\iint_{\R\times\R}\abs{x-y}\dee\rho(t,x)\dee\rho(t,y).
  \end{equation*}
  Using the push-forward relation and \(\alpha = -2\) this becomes
  \begin{equation*}
    \frac12 \int_\Omega V(t,m)^2\dee m - \frac12 \iint_{\Omega\times\Omega}\abs*{X(t,m)-X(t,\omega)}\dee m \dee \omega.
  \end{equation*}
  For the initial configuration we can compute
  \begin{equation*}
    \int_\Omega (V^0)^2(m)\dee m = 1, \qquad \iint_{\Omega\times\Omega}\abs*{X^0(m)-X^0(\omega)}\dee m \dee \omega = \frac13,
  \end{equation*}
  meaning the initial energy is \(\frac13\).
  At \(t = 1\), all mass is concentrated at the origin with zero velocity, and so both the kinetic and potential energy are zero, i.e., the initial amount of energy has dissipated.
  Now, given the formula \eqref{eq:BGSW:X} for the entropic solution for \(t \ge 1\), we can compute the kinetic and potential energy to find the identity
  \begin{equation*}
    \int_\Omega V(t,m)^2\dee m = \iint_{\Omega\times\Omega}\abs*{X(t,m)-X(t,\omega)}\dee m \dee \omega = \frac{(t-1)^3}{3t}, \qquad t \ge 1.
  \end{equation*}
  Hence, as the system tends to decrease its potential energy, for \(t \ge 1\) there is no more dissipation of energy for the entropic solution \((X,V)\), as the kinetic and potential energy exactly balance one another.
  On the other hand, for \((\tilde{X},\tilde{V})\) the energy will increase again for \(t>1\) and tend asymptotically to the initial value.
  Therefore, as in Example \ref{exm:2particles}, we see that the Lagrangian, or entropy, solution provides a more physically desirable behavior in the sense that energy is nonincreasing.
  
  \begin{rem}[The globally sticky solution]
    In the above we could also have considered a globally sticky solution where for \(0 \le t < 1\) the evolution is identical to \(X\) and \(\tilde{X}\), while for \(t \ge 1\) all mass is concentrated in the stationary shock at the origin.
    This would also be a generalized Lagrangian solution, as the zero element is always contained in the tangent cone \(T_{X}\Kscr\), as well as a weak solution of the scalar conservation law since the Rankine--Hugoniot condition is satisfied for the generalized inverse.
    In addition, there would clearly be no more dissipation of energy.
    However, from analogous arguments as before, it fails to satisfy the entropy condition, as \(\Ucal(t,m)\) is no longer concave on the whole of \([0,1]\) for \(t > 1\).
  \end{rem}
  
  \subsubsection{General energy considerations}
  Consider the repulsive Euler--Poisson system \eqref{eq:EP} with no external forcing, i.e.\ \(\alpha < 0\) and \(\beta=0\).
  Once more we write the energy in Lagrangian variables to obtain
  \begin{equation*}
    E(t) \coloneqq \frac12 \int_\Omega V(t,m)^2\dee m + \frac{\alpha}{4} \iint_{\Omega\times\Omega}\abs*{X(t,m)-X(t,\omega)}\dee m \dee \omega.
  \end{equation*}
  From \(V = \Proj_{T_{X(t)}\Kscr}U\) it follows that \(\norm{V}_{\Lp{2}(0,1)} \le \norm{U}_{\Lp{2}(0,1)}\). Then we can instead use an expression involving \(U\) to bound the energy, and the advantage of this is that we have a differential equation for \(U\).
  That is, for
  \begin{equation*}
    \hat{E}(t) \coloneqq \frac12 \int_\Omega U(t,m)^2\dee m + \frac{\alpha}{4} \iint_{\Omega\times\Omega}\abs*{X(t,m)-X(t,\omega)}\dee m \dee \omega
  \end{equation*}
  we find
  \begin{align*}
    \diff{^+}{t} \hat{E}(t) &= \frac{-\alpha}{2}\int_\Omega U(t,m) (2m-1)\dee m + \frac{\alpha}{4} \iint_{\Omega\times\Omega}\sgn(X(t,m)-X(t,\omega))(V(t,m)-V(t,\omega)) \dee m \dee \omega \\
    &=\frac{-\alpha}{2}\int_\Omega U(t,m) (2m-1)\dee m + \frac{\alpha}{2} \int_\Omega V(t,m) \int_\Omega \sgn(X(t,m)-X(t,\omega)) \dee \omega \dee m.
  \end{align*}
  Now we note that
  \begin{equation*}
    \int_\Omega \sgn(X(t,m)-X(t,\omega)) \dee \omega = \begin{cases}
      \displaystyle\int_0^m\dee\omega - \int_m^1\dee\omega = 2m-1, & \text{for a.e.\ }m \notin \Omega_{X(t)} \\[2mm]
      \displaystyle\int_0^{m_l}\dee\omega-\int_{m_r}^1\dee\omega = m_l+m_r-1, & m \in (m_l,m_r) \subset \Omega_{X(t)},
    \end{cases}
  \end{equation*}
  that is, for a.e.\ \(m \in \Omega\) we have
  \begin{equation}\label{eq:signavg}
    \int_\Omega \sgn(X(t,m)-X(t,\omega)) \dee \omega = \Proj_{\Hscr_{X(t)}} (2m-1).
  \end{equation}
  Rewriting a bit we then have
  \begin{align*}
    \diff{^+}{t} \hat{E}(t) = -\frac{\alpha}{2}\int_\Omega (U(t,m)-V(t,m)) (2m-1)\dee m - \frac{\alpha}{2}\int_\Omega V(t,m) \left(2m-1 - \Proj_{\Hscr_{X(t)}}(2m-1)\right)\dee m.
  \end{align*}
  In the above, the final term vanishes due to \(V \in \Hscr_{X(t)}\) and \(2m-1 - \Proj_{\Hscr_{X(t)}}(2m-1) \in \Hscr_{X(t)}^\perp\).
  On the other hand, integrating by parts, we have
  \begin{equation*}
    \int_{\Omega}(U(t,m)-V(t,m))(2m-1)\dee m = -2 \int_\Omega\int_0^m(U(t,\omega)-V(t,\omega))\dee \omega \dee m \le 0,
  \end{equation*}
  where the final inequality comes from \(U-V \in N_{X(t)}\Kscr\) and the characterization provided by Lemma \ref{lem:NX}.
  Since \(\alpha < 0\) we then see that \(\hat{E}(t)\) is decreasing in time.
  Hence, \(E(t) \le \hat{E}(t) \le \hat{E}(0) = E(0)\).
  Note that a different strategy is needed for the attractive case, as the acceleration will make the prescribed velocity \(\norm{U}_{\Lp{2}(0,1)}\) increase through collisions.
  
  \subsection{Repulsive Poisson and quadratic confinement forces}\label{ss:CCZ}
  As we have seen, internal forces of the Euler--Poisson system \eqref{eq:EP} are either repulsive or attractive depending on the sign of \(\alpha\).
  The dynamics become even more complex when both attractive and repulsive forces are present: for instance, one could combine repulsive Poisson forces with a quadratic confinement potential, like in \cite{carrillo2016pressureless}.
  Following these lines, we modify the damped Euler--Poisson system \eqref{eq:EPld} by taking \(\alpha = -\lambda^2 < 0\), \(\gamma \ge 0\) and \(\Phi(t,x) = \phi \ast \rho(t,x)\) with \(\phi(x) = \abs{x}-\frac12 x^2\); this gives a repulsive Poisson potential with quadratic confinement and linear damping as described by \eqref{eq:EPQ}, namely
  \begin{equation*}
    \del_t \rho + \del_x \left(\rho v\right) = 0, \qquad
    \del_t (\rho v) + \del_x \left(\rho v^2\right) = -\gamma \rho v + \lambda^2 \rho \int_\R \left(\sgn(x-y)-(x-y)\right)\dee\rho(t,y) - \beta \rho,
  \end{equation*}
  where we observe that the effect of the confinement potential is to drive \(\rho\) toward the center of mass with a force that increases linearly with distance.
  As mentioned in Example  \ref{exm:flux:EPld}, this force distribution \(f[\rho,v]\) is not directly covered by \eqref{eq:EF} since it also involves the velocity \(v\).
  However, the same argument as in \cite[Theorem 3.6]{brenier2013sticky} can be adapted for this case as well. Indeed, one can still consider the operator \(A(X,U) = (\del I _{\Kscr}(X)-U,F[X,U])\) as a Lipschitz perturbation of \(\del I_{\Kscr}(X)\), as long as \(F[X,U]\) is Lipschitz in both \(X\) and \(U\).
  
  \subsubsection{Derivation of the prescribed velocity and flux function in the smooth case}
  Note that, using the \textit{a priori} knowledge that \(\rho\) is a probability measure, we can, as for the Euler--Poisson system, rewrite \((\sgn\ast\rho)\rho = \del_x (M^2-M)\) as before.
  On the other hand, under the assumption that \(\rho(t) \in \Pscr_2(\R)\),  the term \((\Id\ast\rho)\rho\) is a well-defined measure since \((\Id\ast\rho)\) is continuous for a given \(t\).
  Indeed, for such \(\rho\) we should have \(\abs{(\Id\ast\rho)(t,x)} \le 1+\abs{x}+\overline{x^2}(t)\), where \(\overline{x^2}(t)\) is the second moment at time \(t\).
  Moreover, \((\Id\ast\rho)(t,x) = x+(\Id\ast\rho)(t,0) = x-\bar{x}(t)\) and \((\Id\ast\rho)(t,x)-(\Id\ast\rho)(t,y) = x-y\), where we have introduced the center of mass, or first moment, \(\bar{x}(t)\).
  Recalling Definitions \ref{dfn:bnd:f} and \ref{dfn:unicont:f}, we find that for fixed \(t\) it is a locally bounded, linear function in \(x\), with Lipschitz-constant 1.
  Integrating \((\Id\ast\rho)\rho\) we obtain
  \begin{equation*}
    \int_\R \dee\rho(t,y) \int_{(-\infty,x]} y\dee\rho(t,y) - \int_\R y \dee\rho(t,y) \int_{(-\infty,x]} \dee\rho(t,y) = \int_{(-\infty,x]} y\dee\rho(t,y) - \bar{x}(t) M(t,x).
  \end{equation*}
  Therefore we see that this term cannot directly be expressed in terms of the distribution function \(M\), contrary to the primitive of \((\sgn\ast\rho)\rho\).
  
  However, we can still apply some \textit{a priori} knowledge to the trajectory of the center of mass.
  This should not be affected by the internal forces, i.e., the forces coming from \(\phi'\ast\rho\), but only the external damping.
  Multiplying the continuity equation with \(x\) and integrating we obtain
  \begin{equation*}
    \dot{\bar{x}}(t) = \del_t \int_\R x\dee\rho(t,x) = \int_\R v\dee\rho(t,x) \eqqcolon \bar{v}(t),
  \end{equation*}
  that is, the velocity of the center of mass is given by the total momentum.
  Furthermore, integrating the momentum equation, the internal forces vanish due to symmetry, or, if one prefers, Newton's third law of motion, and we are left with
  \begin{equation*}
    \dot{\bar{v}}(t) = -\gamma \bar{v}(t) - \beta.
  \end{equation*}
  Integrating these equations we then obtain
  \begin{equation}\label{eq:CCZ:means}
    \bar{x}(t) = \bar{x}(0) + \bar{v}(0) \frac{1-\e^{-\gamma t}}{\gamma} + \beta \frac{1-\gamma t - \e^{-\gamma t}}{\gamma^2} \qquad \bar{v}(t) = \e^{-\gamma t}\bar{v}(0)-\beta\frac{1-\e^{-\gamma t}}{\gamma},
  \end{equation}
  where the case \(\gamma=0\) corresponds to the limit \(\gamma \to 0\).
  These relations should remain valid no matter how one resolves the breakdown of classical solutions, i.e., concentration of mass.
  
  Now we want to derive the flux function for our problem, by following the arguments leading up to \eqref{eq:Q&U}.
  We then suggestively introduce
  \begin{align*}
    u(t,X(t,m)) &= v_0 \circ X^0(m) + \int_0^t F[X](m)\dee s \\
    &= V^0(m) + \int_0^t \left[ -\gamma u(s,X(s,m)) +\lambda^2 \left( 2m-1 - X(s,m) + \bar{x}(s) \right) \right] \dee s,
  \end{align*}
  where \(F[X]\) is in accordance with \eqref{eq:f&F} for our right-hand side.
  This relation together with \(\dot{X}(t,m) = u(t,X(t,m)) = U(t,m)\) should be satisfied by the solution as long as it remains smooth; alternatively we can write
  \begin{equation*}
    \dot{X}(t,m) = U(t,m), \qquad \dot{U}(t,m) = -\gamma U(t,m) + \lambda^2 \left( 2m-1 + \bar{x}(t) - X(t,m) \right),
  \end{equation*}
  which should be compared to the smooth trajectories of the numerical example \cite[Equation (1.40)]{brenier2013sticky}.
  Differentiating we obtain the second-order inhomogeneous ODE 
  \begin{equation*}
    \ddot{U} + \gamma \dot{U} + \lambda^2 U = \lambda^2 \bar{v}(t), \quad 
  \end{equation*}
  with initial conditions
  \begin{equation*}
    U(0,m) = V^0(m), \qquad \dot{U}(0,m) = -\gamma V^0(m) + \lambda^2 (2m-1 + \bar{x}(0)-X^0(m)).
  \end{equation*}
  Note that since the total momentum, or mean velocity, \(\bar{v}\) from \eqref{eq:CCZ:means} satisfies \(\ddot{\bar{v}}+\gamma \dot{\bar{v}} = 0\), the above can be restated as a homogeneous ODE for the deviation \(\tilde{U} = U - \bar{v}\).
  Let us from here on write \(\gamma = 2\kappa\) to simplify expressions.
  We are interested in the case where classical solutions oscillate, and so we take \(\lambda > \kappa \ge 0\).
  Solving the ODE we obtain the following expression for the smooth prescribed velocity,
  \begin{equation}\label{eq:CCZ:U}
    \begin{aligned}
      U(t,m) &= \bar{v}(t) + \left( V^0(m) - \bar{v}(0) \right) \e^{-\kappa t} \cos\left(\sigma t\right) \\
      &\quad+ \left[\beta -\kappa \left(V^0(m)-\bar{v}(0)\right) - \lambda^2 \left(X^0(m) -2m+1 -\bar{x}(0)\right)\right] \e^{-\kappa t} \frac{\sin\left(\sigma t\right)}{\sigma},
    \end{aligned}
  \end{equation}
  where we have introduced \(\sigma \coloneqq \sqrt{\lambda^2-\kappa^2}\).
  The corresponding flux function is then given by the primitive of this velocity, which is of the form \eqref{eq:flux}.
  Indeed, this follows since \(\rho_0 \in \Pscr_2(\R)\) means \(\Id \in \Lp{2}(\R,\rho_0)\) which again implies \(\Id\circ X^0 = X^0 \in \Lp{2}(0,1)\), while \(v_0 \in \Lp{2}(\R,\rho_0)\) leads to \(V^0 = v_0\circ X^0 \in \Lp{2}(0,1)\).
  Writing
  \begin{equation*}
    \Xcal^0(m) = \int_{0}^{m}X^0(\omega)\dee \omega, \qquad \Vcal^0(m) = \int_{0}^{m}V^0(\omega)\dee\omega = \int_{0}^{m}v_0(X^0(\omega))\dee \omega,
  \end{equation*}
  the flux function can be expressed as
  \begin{align*}
    \Ucal(t,m) &= \bar{v}(t) m + \left( \Vcal^0(m) - \bar{v}(0) m \right) \e^{-\kappa t} \cos\left(\sigma t\right) \\
    &\quad+ \left[\beta m -\kappa \left(\Vcal^0(m)-\bar{v}(0) m\right) - \lambda^2 \left(\Xcal^0(m) -m^2+m -\bar{x}(0) m\right)\right] \e^{-\kappa t} \frac{\sin\left(\sigma t\right)}{\sigma} \\
    &= m \diff{\bar{x}}{t} + \left( \Vcal^0(m) - \bar{v}(0) m \right) \diff{}{t} \e^{-\kappa t} \frac{\sin(\sigma t)}{\sigma} \\
    &\quad+ \left[\beta m -\kappa \left(\Vcal^0(m)-\bar{v}(0) m\right) - \lambda^2 \left(\Xcal^0(m) -m^2+m -\bar{x}(0) m\right)\right] \diff{}{t} \e^{-\kappa t} \left( \cos(\sigma t) + \frac{\kappa}{\sigma}\sin(\sigma t) \right).
  \end{align*}
  
  \begin{rem}[Comparison with existing formulas]
    Let us verify that this agrees with the formulas found in \cite{carrillo2016pressureless}, where we first note that their mass is not normalized to one.
    In their formulas they denote the total (conserved) mass by \(M_0\) and the initial momentum by \(M_1\); we will instead denote these by respectively \(\mu_0\) and \(\mu_1\) to avoid confusion with the distribution function \(M\).
    Moreover, the parameters for their system are \(\lambda = 1\), \(\gamma = 2\kappa = 1\) and \(\beta=0\), such that \(1 < 4\mu_0\) is what they call case C, exhibiting oscillatory behavior.
    We would therefore have to rescale \(\rho\) to obtain the corresponding set of equations.
    Writing their density as \(\tilde{\rho} = \mu_0 \rho\) for \(\rho \in \Pscr(\R)\), this amounts to the rescaling \(\lambda^2 \mapsto \lambda^2 \mu_0\) and the identities \(\sigma = \frac12 \sqrt{4\mu_0-1}\), \(\bar{v}(0) = \mu_1/\mu_0\) in our equations.
    Now, for this set of parameters we would like our expression \eqref{eq:CCZ:U} to coincide with the smooth characteristic velocity given in \cite[Equation (2.8)]{carrillo2016pressureless}.
    However, their characteristics \(\eta(t,x)\), satisfying \(\del_t \eta(t,x) = u(t,\eta(t,x))\) and \(\eta(0,x) = x\), are parametrized by the initial position \(x \in \R\) rather than the material coordinate \(m \in [0,1]\) for our \(X(t,m)\).
    Since we are in the smooth case, we have the identity \(M(t,\eta(t,x)) = M^0(x)\), and so evaluating \eqref{eq:CCZ:U} at \(m = M^0(x)\) yields the characteristic velocity parametrized by \(x\).
    In addition we have the identities \(X^0(M^0(x)) = x\) and \(V^0(M^0(x)) = v_0(x)\) which hold \(\rho_0\)-a.e., and so we plug into \eqref{eq:CCZ:U} to find
    \begin{align*}
      U(t,M^0(x)) &= \frac{\mu_1}{\mu_0} \e^{-t} + \left( v_0(x) - \frac{\mu_1}{\mu_0}\right) \e^{-t/2} \cos\left(\frac{\sqrt{4\mu_0-1}}{2}t\right) \\
      &\quad+ \frac{2}{\sqrt{4\mu_0-1}} \left[ \mu_0 \left( 2M^0(x) - 1 + \bar{x}(0) - x\right) - \frac12\left(v_0(x) - \frac{\mu_1}{\mu_0}\right)  \right] \e^{-t/2} \sin\left(\frac{\sqrt{4\mu_0-1}}{2}t\right),
    \end{align*}
    which is exactly the formula we are after, since we have the relations
    \begin{equation*}
      \mu_0 M^0(x) = \int_{-\infty}^x \tilde{\rho}_0(y)\dee y, \qquad \mu_0 \bar{x}(t) = \int_\R y \tilde{\rho}(t,y)\dee y.
    \end{equation*}
    Hence, we are in agreement with \cite{carrillo2016pressureless} as long as the solutions remain classical, which is as far as their analysis goes.
  \end{rem}
  
  Next, with the Lagrangian and entropy solution framework in hand, we aim to provide well-defined solutions beyond the breakdown of classical solutions, and for this we give some examples.
  
  \subsubsection{A simple (computable) case}
  Let us, as in Section \ref{ss:BGSW}, take \(X^0(m) = m-\frac12\), but, as explained below, we will for simplicity take \(V^0(m) = -\left(m-\frac12\right)\).
  Then it is clear that the initial center of mass is \(\bar{x}(0) = 0\) and the initial total momentum, or average velocity, is \(\bar{v}(0) = 0\), which simplifies the flux function a bit.
  To simplify even more, we let the external force be zero, \(\beta = 0\).
  Computing \(\Xcal^0(m)\) we find
  \begin{equation*}
    \Xcal^0(m) = \int_0^m X^0(\omega)\dee \omega = \frac{m^2 -m}{2},
  \end{equation*}
  which leads to the flux function
  \begin{align*}
    \Ucal(t,m) &= \Vcal^0(m) \e^{-\kappa t}\left[ \cos\left(\sigma t\right) - \frac{\kappa}{\sigma} \sin\left(\sigma t\right) \right] + \frac{\lambda^2}{2}\left(m^2-m\right) \e^{-\kappa t} \frac{\sin(\sigma t)}{\sigma} \\
    &= -\frac12 \left(m^2-m\right)\e^{-\kappa t}\left[ \cos\left(\sigma t\right) - \frac{\kappa+\lambda^2}{\sigma} \sin\left(\sigma t\right) \right].
  \end{align*}
  Then the prescribed velocity \(U = \pdiff{}{m}\Ucal\) is
  \begin{equation*}
    U(t,m) = -\left(m-\frac12\right) \e^{-\kappa t}\left[ \cos\left(\sigma t\right) - \frac{\kappa+\lambda^2}{\sigma} \sin\left(\sigma t\right) \right]
  \end{equation*}
  which we integrate to find the free trajectory
  \begin{align*}
    Y(t,m) &= X^0(m) + V^0(m) \e^{-\kappa t} \frac{\sin\left(\sigma t\right)}{\sigma} - \frac{m^2-m}{2} \e^{-\kappa t} \left[\cos\left(\sigma t\right) + \frac{\kappa}{\sigma} \sin\left(\sigma t\right)\right] \\
    &= \left(m-\frac12\right) \left[ 2 - \e^{-\kappa t}\left( \cos\left(\sigma t\right) + \frac{1+\kappa}{\sigma} \sin\left(\sigma t\right) \right) \right].
  \end{align*}
  Since \(V^0\) is decreasing, the flux function is initially concave, and will periodically alternate between being concave and convex in between the times \(t\) satisfying \(\cot(\sigma t) = \frac{\kappa+\lambda^2}{\sigma}\), which coincide with the times when the free trajectory \(Y(\cdot,m)\) has its extremal amplitudes.
  Moreover, since mass concentrates when the free trajectories are no longer increasing, we see that either all or no mass concentrates depending on the function
  \begin{equation*}
    c(t) \coloneqq 2 - \e^{-\kappa t}\left( \cos\left(\sigma t\right) + \frac{1+\kappa}{\sigma} \sin\left(\sigma t\right) \right),
  \end{equation*}
  which attains its minimum \(c(\tau_*)\) at \(t = \tau_*\), where
  \begin{equation*}
    \tau_* = \arctan\left(\frac{\sigma}{\kappa+\lambda^2}\right), \qquad c(\tau_*) = 2 - \exp\left(-\frac{\kappa}{\sigma}\arctan\left(\frac{\sigma}{\kappa+\lambda^2}\right)\right)\frac{\sqrt{\lambda^2+2\kappa+1}}{\lambda}.
  \end{equation*}
  If \(c(\tau_*)\) is nonnegative, there is no mass concentration, except possibly at the instant \(t = \tau_*\) if \(c(\tau_*) = 0\).
  However, in this case, the velocity of the free trajectories will all be zero, and thus tangential to one another, and the Lagrangian solution will still be given by the smooth curves given by the free trajectory \(Y(t,m)\).
  On the other hand, if \(c(\tau_*) < 0\), then there is some \(\tau_0 \in (0,\tau_*)\) such that \(c(\tau_0) = 0\) and all mass collides at the origin.
  We note that in the undamped case \(\kappa = 0\) we explicitly find
  \begin{equation*}
    \tau_0 = \arcsin\left(\frac{\lambda \left(2-\sqrt{1-3\lambda^2}\right)}{1+\lambda^2}\right),
  \end{equation*}
  which also emphasizes that mass concentration can only occur for \(0 < \lambda < 1/\sqrt{3}\) in the undamped case.
  From time \(t= \tau_0\) onward we can no longer use the smooth free trajectory found by solving
  \begin{equation*}
    \dot{X}(t,m) = U(t,m) = V^0(m) + \int_0^t \left[ -2\kappa U(s,m) +\lambda^2 \left( 2m-1 - X(s,m) \right) \right] \dee s,
  \end{equation*}
  since \(X(t,m) = 0\) due to the concentration at the origin.
  Instead we take the velocity \(U(\tau_0,m)\) as initial condition for the equation
  \begin{equation*}
    U(t,m) = U(\tau_0,m) + \int_{\tau_0}^t \left[ -2\kappa U(s,m) +\lambda^2 \left( 2m-1 \right) \right] \dee s
  \end{equation*}
  and find
  \begin{align*}
    U(t) &= \e^{-2\kappa(t-\tau_0)}U(\tau_0,m) + \lambda^2 \e^{-2\kappa t}(t-\tau_0) (2m-1) \\
    &= \left(m-\frac12\right) \e^{-2\kappa t} \left[2\lambda^2(t-\tau_0) -\e^{\kappa \tau_0} \left( \cos(\sigma \tau_0) - \frac{\kappa}{\sigma}\sin(\sigma\tau_0) \right) \right].
  \end{align*}
  Since \(U(\tau_0,m)\) is decreasing in \(m\), meaning \(\Ucal(\tau_0,m)\) is concave, we see that the mass will remain at the origin until \(U(t,m)\) becomes increasing at times \(t > \tau_1\) where
  \begin{equation*}
    \tau_1 = \tau_0 + \frac{\e^{\kappa \tau_0}}{2\lambda^2} \left( \cos(\sigma \tau_0) - \frac{\kappa}{\sigma}\sin(\sigma\tau_0) \right),
  \end{equation*}
  and we note that \(U(\tau_1,m) = 0\).
  From this point on we solve the original second order, inhomogeneous equation with initial data \(X(\tau_1,m) = U(\tau_1,m) = 0\) to find
  \begin{equation*}
    X(t,m) = \left(2m- 1\right) \left[1 - \e^{-\kappa (t-\tau_1)} \cos\left(\sigma(t-\tau_1)\right) \right].
  \end{equation*}
  Note that even in the undamped case \(\kappa = 0\), although some kinetic energy is lost at the first collision \(t = \tau_0\), for the remaining collisions happening at \(t = \tau_1 + n 2\pi/\lambda\) for \(n \in \N\), the velocities are tangential and there is no loss of energy.
  On the other hand, in the damped case we observe that \(X(t,m) \to 2m-1\) as \(t \to \infty\), corresponding to the density \(\rho(t,x)\) tending to \(\frac12 \chi_{[-1,1]}(x)\).
  
  In total, the Lagrangian solution is given by 
  \begin{equation*}
    X(t,m) = \begin{cases}
      Y(t,m), & c(\tau_*) \ge 0, \\
      \begin{cases}
        Y(t,m), & 0 \le t < \tau_0, \\
        0, & \tau_0 \le t < \tau_1, \\
        (2m-1) \left[1 - \e^{-\kappa (t-\tau_1)} \cos\left(\sigma(t-\tau_1)\right) \right], & t \ge \tau_1
      \end{cases} & c(\tau_*) < 0.
    \end{cases}
  \end{equation*}
  As we have shown, this solution then also corresponds to the entropy solution since it is in accordance with the characterization of the tangent cone in \eqref{eq:TC2}.
  Like in the previous section, we can compare this with the trajectory \(\tilde{X} = \Proj_{\Kscr}Y\) defined by the projection formula \eqref{eq:sticky:X}, which in this case can be expressed as
  \begin{equation*}
    \tilde{X}(t,m) = \sgn\left(m-\frac12\right) \max\left\{ \sgn\left(m-\frac12\right)Y(t), 0 \right\} = \left(m-\frac12\right) \max\left\{c(t), 0\right\}.
  \end{equation*}
  See Figures \ref{fig:linearV0undamped} and \ref{fig:linearV0damped} for illustrations of the undamped and damped cases respectively.
  Note that for our initial velocity, \(\lambda > 0\) needs to be rather small to have concentration, meaning the repulsive force is quite weak and we see in Figures \ref{fig:linearV0undamped}\textsc{(b)} and \ref{fig:linearV0damped}\textsc{(b)} that the mass stays concentrated for a while before spreading out again.
  
  \begin{figure}
    \begin{subfigure}[b]{0.495\linewidth}
      \includegraphics[width=1\linewidth]{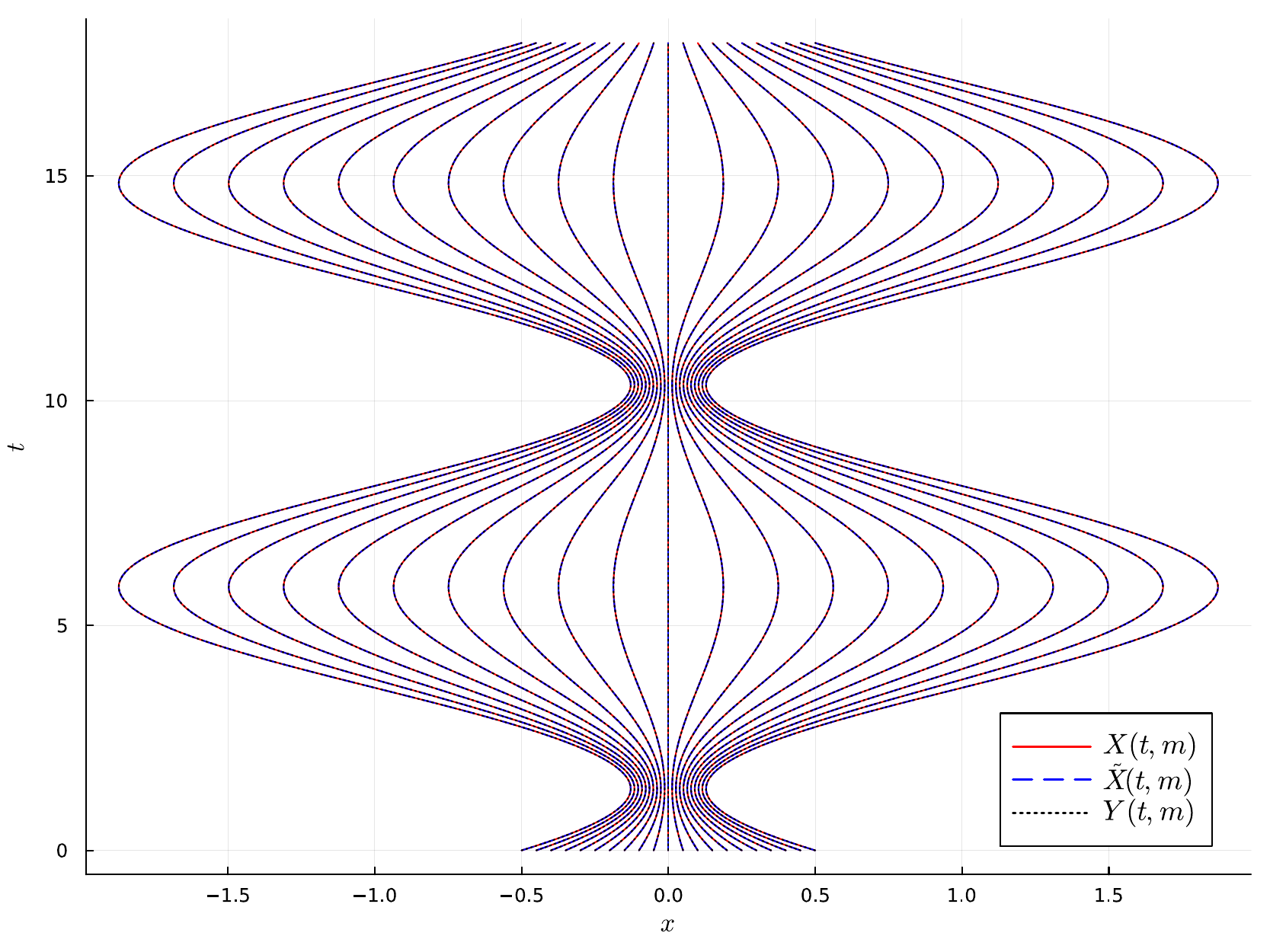}
      \caption{\(\lambda = 0.7\)}
    \end{subfigure}
    \begin{subfigure}[b]{0.495\linewidth}
      \includegraphics[width=1\linewidth]{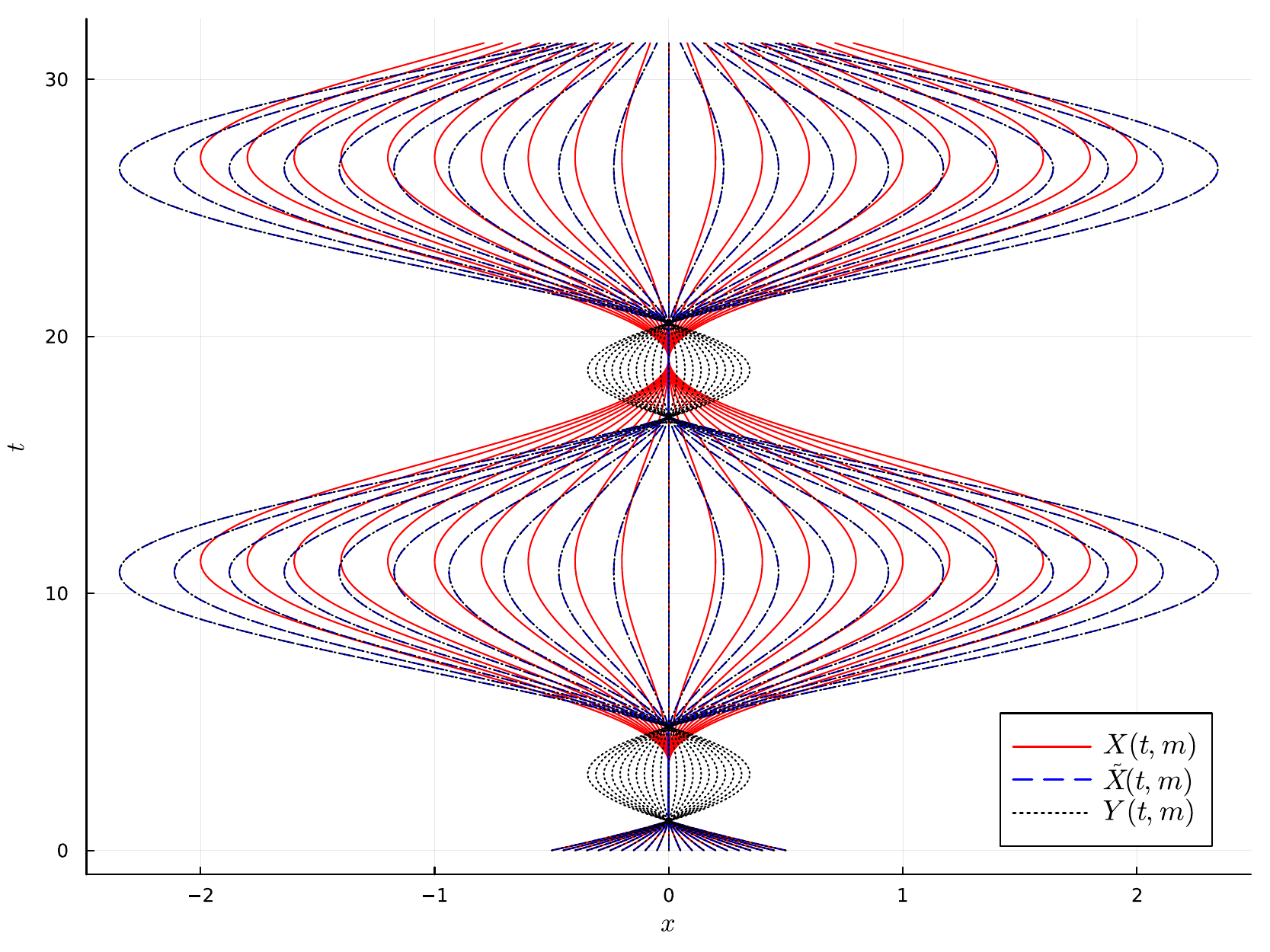}
      \caption{\(\lambda = 0.5\)}
    \end{subfigure}
    \caption{The trajectories \(X(t,\theta_i)\), \(\tilde{X}(t,\theta_i) = (\Proj_{\Kscr} Y)(t,\theta_i)\) and \(Y(t,\theta_i)\) for \(\theta_i = \frac{i}{20}\), \(i \in \{0,1,\dots,20\}\), \(\kappa = 0\) and sub- \textsc{(a)} and supercritical \textsc{(b)} values of \(\lambda\).}
    \label{fig:linearV0undamped}
  \end{figure}
  \begin{figure}
    \begin{subfigure}[b]{0.495\linewidth}
      \includegraphics[width=1\linewidth]{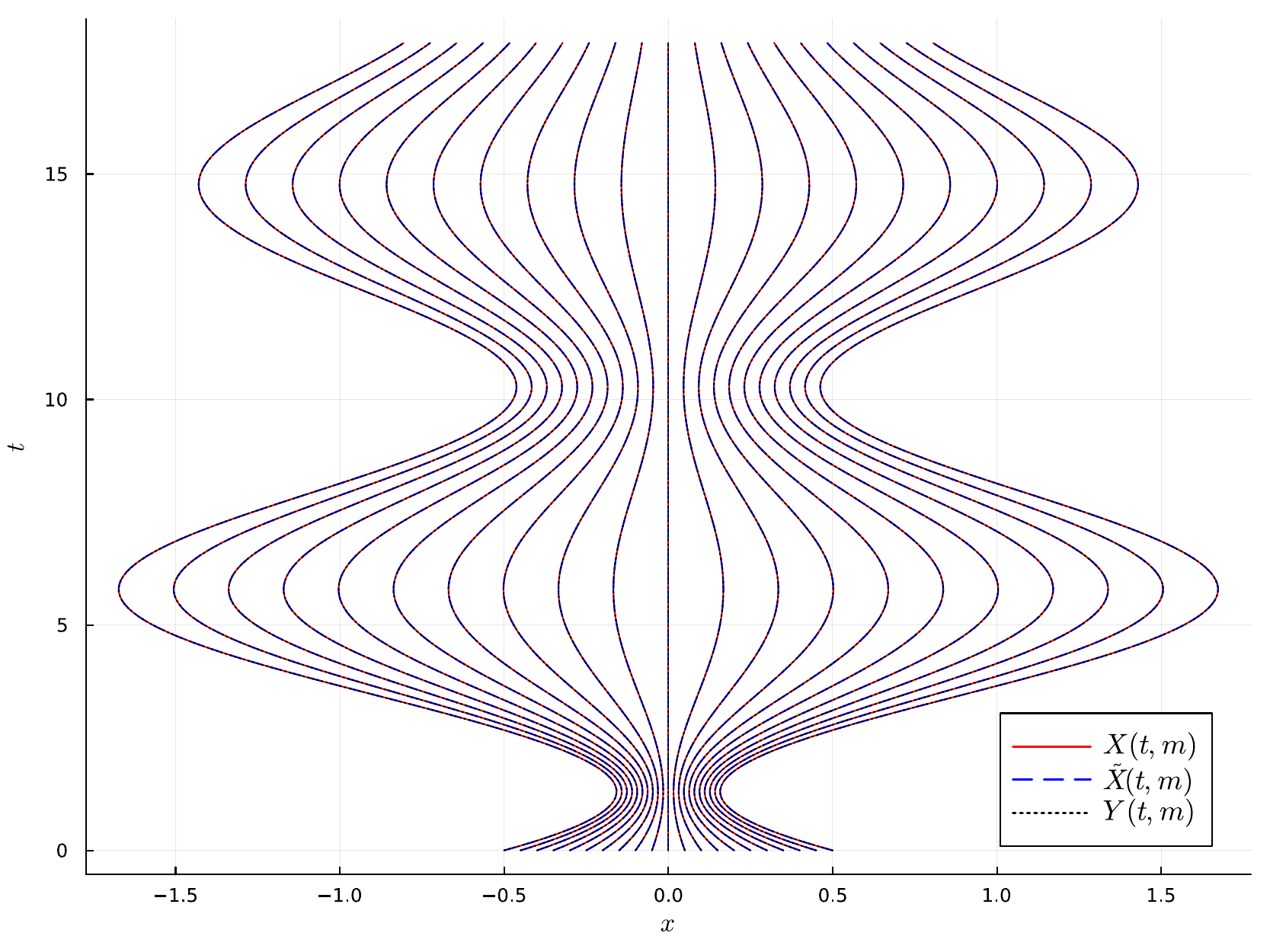}
      \caption{\(\lambda = 0.7\)}
    \end{subfigure}
    \begin{subfigure}[b]{0.495\linewidth}
      \includegraphics[width=1\linewidth]{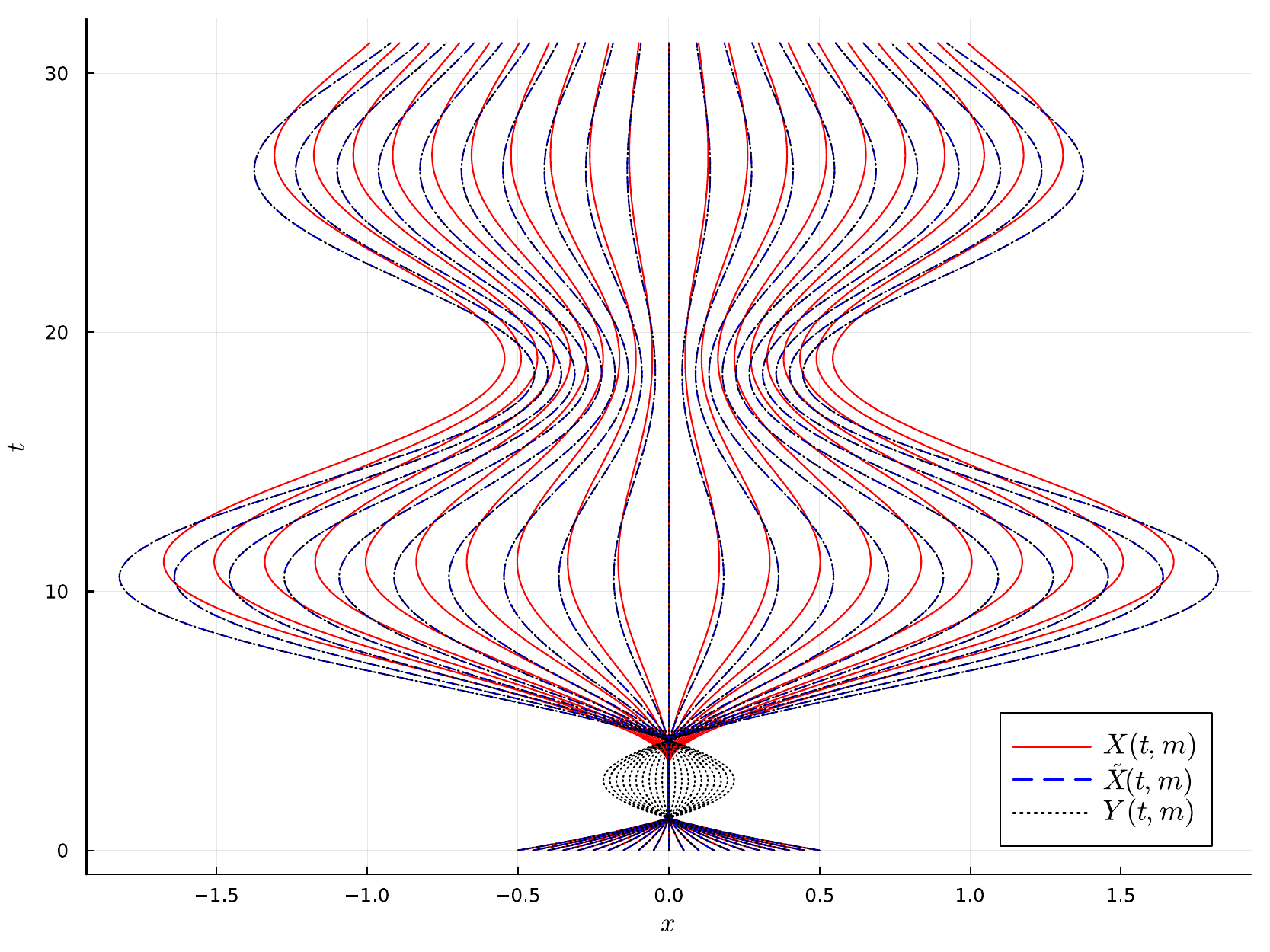}
      \caption{\(\lambda = 0.5\)}
    \end{subfigure}
    \caption{The trajectories \(X(t,\theta_i)\), \(\tilde{X}(t,\theta_i) = (\Proj_{\Kscr} Y)(t,\theta_i)\) and \(Y(t,\theta_i)\) for \(\theta_i = \frac{i}{20}\), \(i \in \{0,1,\dots,20\}\), \(\kappa=0.05\), and sub- \textsc{(a)} and supercritical \textsc{(b)} values of \(\lambda\).}
    \label{fig:linearV0damped}
  \end{figure}
  
  \subsubsection{A more involved example with particles}
  The choice of initial data \(X^0\) and \(V^0\) in the previous example simplified the analysis, since the convexity properties of the flux function \(m \mapsto \Ucal(t,m)\), or equivalently, the monotonicity properties of the prescribed velocity \(m \mapsto U(t,m)\) depended only on time \(t\).
  The smoothness of \(V^0\) then meant that the inter-particle forces had to be rather weak, i.e., \(\lambda\) small enough, for concentration to happen.
  Choosing instead initial data as in the repulsive Euler--Poisson case before, namely \(X^0(m) = m-\frac12\) and \(V^0(m) = -\sgn(m-\frac12)\), the discontinuous \(V^0\) ensures that mass will immediately concentrate near the origin.
  However, in this case, the convexity properties will depend on both \(t\) and \(m\), causing far more complex dynamics where mass will concentrate also outside the origin.
  
  Therefore, to simplify the analysis we will consider a particle approximation of this problem, where the uniformly distributed mass on \((-\frac12, \frac12)\) is first divided into four intervals \((\frac{i-3}{4}, \frac{i-2}{4})\) for \(i \in \{1,2,3,4\}\) for then to be concentrated at the midpoints \(x_i^0 = \frac{2i-5}{8}\) with \(m_i = \frac14\).
  Moreover we define the cumulative quantity \(\theta_i = \frac{i}{4}\) for \(i \in \{0,\dots,5\}\), where we note the relations \(x_i^0 = \frac12 \left(\theta_{i-1}+\theta_i -1\right)\) and \(v_i^0 = -\sgn(\theta_{i-1}+\theta_i -1)\) for \(i \in \{1,2,3,4\}\).
  As in the previous examples, the symmetry of the initial data ensures that the center of mass is always located at the origin, simplifying our equations, and so the initial particle dynamics are given by
  \begin{equation*}
    \dot{x}_i = v_i, \qquad
    \dot{v}_i = -2\kappa v_i + \lambda^2 \left( \theta_{i-1} + \theta_i -1 -x_i \right).
  \end{equation*}
  This we may integrate to find the ``smooth'' trajectories
  \begin{equation*}
    x_i(t) = x_i^0 \left[ 2 - \e^{-\kappa t} \left( \cos(\sigma t) + \frac{\kappa}{\sigma}\sin(\sigma t) \right) \right] + v_i^0 \e^{-\kappa t} \frac{\sin(\sigma t)}{\sigma}.
  \end{equation*}
  Initially we see that the inter-particle repulsion dominates the confinement force which attracts the particles to the center of mass at the origin.
  For \(\lambda > \kappa > 0\) small enough, the repulsive and damping forces will be too weak to stop the particles from colliding.
  By symmetry, particles \(i = 2, 3\) will collide first, at \(t = \tau_{1}^-\).
  Now these particles will be stuck at the origin for some time, where their prescribed velocities satisfy
  \begin{equation*}
    \dot{u}_i + 2\kappa u_i = \lambda^2 \left(\theta_{i-1} +\theta_i -1\right), \qquad u_i(\tau_1^-) = v_i(\tau_1^-),
  \end{equation*}
  which for \(t \ge \tau_1^-\) integrates to
  \begin{equation*}
    u_i(t) = \e^{-2\kappa(t-\tau_1^-)} v_i(\tau_1^-) + \lambda^2 \left(\theta_{i-1} +\theta_i -1\right) \frac{1 - \e^{-2\kappa(t-\tau_1^-)}}{2\kappa}, \qquad i \in \{2,3\}.
  \end{equation*}
  Since the particles collided at \(t = \tau_1^-\), we know that \(v_2(\tau_1^-) > v_3(\tau_1^-) = -v_2(\tau_1^-)\), and so \(v_i(t) = (\Proj_{T_{\mathbf{x}(t)}\K^n}\mathbf{u})_i = 0\) for \(i = 2, 3\) until \(u_3(t) \ge u_2(t)\).
  However, the second pair of particles hits the origin at time \(t = \tau_1^+ > \tau_1^-\) before this can happen.
  At this time we will have the chain of inequalities
  \begin{equation*}
    u_{1}(\tau_1^+) > u_2(\tau_1^+) > u_3(\tau_1^+) > u_4(\tau_1^+),
  \end{equation*}
  where by symmetry \(u_4(\tau_1^+) = -u_1(\tau_1^+)\) and \(u_3(\tau_1^+) = -u_2(\tau_1^+)\).
  Analogously, the prescribed velocity of the second particle pair for \(t \ge \tau_1^+\) is
  \begin{equation*}
    u_i(t) = \e^{-2\kappa(t-\tau_1^+)} v_i(\tau_1^+) + \lambda^2 \left(\theta_{i-1} +\theta_i -1\right) \frac{1 - \e^{-2\kappa(t-\tau_1^+)}}{2\kappa}, \qquad i \in \{1,4\}.
  \end{equation*}
  The Ole\u{\i}nik E-condition will be violated when \(u_1(t)\) changes sign at the time \(t = \tau_2^+\), and then \(v_i(t) = (\Proj_{T_{\mathbf{x}(t)}\K^n}\mathbf{u}(t))_i = u_i(t)\) for \(i = 1, 4\), releasing them from the origin with initial data \(x_i(\tau_2^+) = 0\) and \(v_i(\tau_2^+) = 0\),
  yielding trajectories
  \begin{equation*}
    x_i(t) = \left(\theta_{i-1} + \theta_i - 1 \right) \left[1 - \e^{-\kappa (t-\tau_2^+)} \left( \cos(\sigma (t-\tau_2^+)) + \frac{\kappa}{\sigma} \sin(\sigma (t-\tau_2^+))\right) \right]
  \end{equation*}
  At the later time \(t = \tau_2^- > \tau_2^+\), \(u_2(t)\) will change sign, and particles \(i = 2, 3\) will be released with corresponding trajectories
  \begin{equation*}
    x_i(t) = \left(\theta_{i-1} + \theta_i - 1 \right) \left[1 - \e^{-\kappa (t-\tau_2^-)} \left( \cos(\sigma (t-\tau_2^-)) + \frac{\kappa}{\sigma} \sin(\sigma (t-\tau_2^-))\right) \right]
  \end{equation*}
  For weak enough damping, at a time \(t = \tau_3\) there will be a new collision between particles \(i = 1, 2\) and \(i = 3, 4\) respectively.
  Since there was a collision, one has that \(v_{i}(\tau_3) > v_{i+1}(\tau_3)\) for \(i = 1, 3\), and the particles will stick together until the Ole\u{\i}nik E-condition no longer holds, i.e., when \(u_i(t) \le u_{i+1}(t)\), which happens at \(t=\tau_4\).
  In the interval \(t \in [\tau_3, \tau_4]\), the pairs of particles will in fact move along their centers of mass, which is the average of the two previous trajectories.
  In the damped case, \(t = \tau_4\) is the last time the particles are in contact, see Figure \ref{fig:4part_d}.
  Indeed, for \(t \ge \tau_4\), the trajectory of the \(i\)\textsuperscript{th} particle for \(i \in \{1,2,3,4\}\) is given by  
  \begin{align*}
    x_i(t) &= x_i(\tau_4) + \left( \theta_{i-1} + \theta_i -1 -x_i(\tau_4) \right) \left[1 - \e^{-\kappa (t-\tau_4)} \left( \cos(\sigma(t-\tau_4)) + \frac{\kappa}{\sigma}\sin(\sigma(t-\tau_4)) \right) \right] \\
    &\quad+ v_i(\tau_4) \e^{-\kappa (t-\tau_4)}\frac{\sin(\sigma(t-\tau))}{\sigma}.
  \end{align*}
  As we can see, for \(\kappa > 0\), the particle trajectories will asymptotically tend to the equilibrium state
  \begin{equation*}
    \lim\limits_{t\to\infty}x_i(t) \coloneqq \theta_{i-1} + \theta_i -1, \qquad \lim\limits_{t\to\infty} v_i(t) = 0,
  \end{equation*}
  which is the four-particle approximation of \(\rho = \frac12 \chi_{(-1,1)}\) and \(v \equiv 0\).
  In the undamped case \(\kappa=0\), the trajectories of particles 1 and 2, and similarly 3 and 4, will still periodically meet at times \(t = \tau_4 + 2\pi n/\sigma\), but tangentially, so there is no loss of energy, see Figure \ref{fig:4part_ud}.
  
  \begin{figure}
    \includegraphics[width=0.8\linewidth]{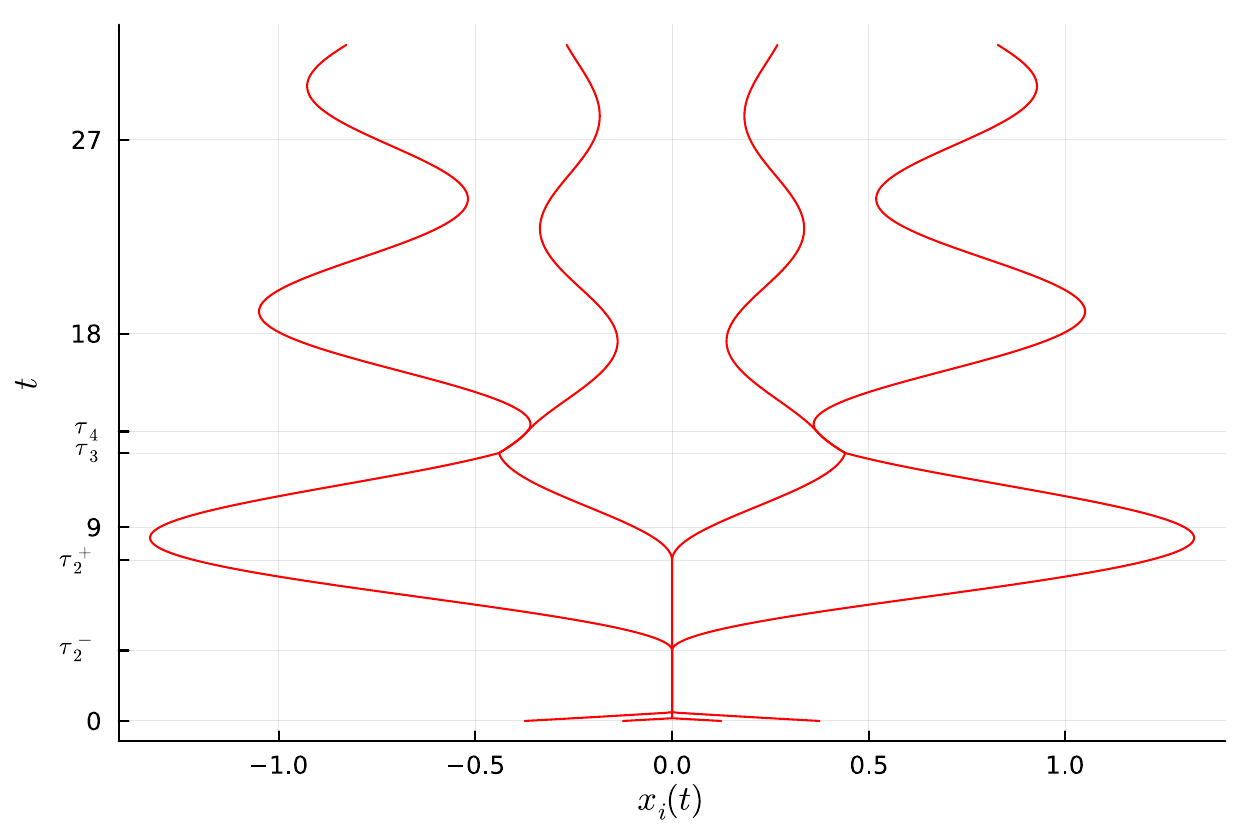}
    \caption{Poisson repulsion and quadratic confinement with \(\lambda = 0.6\) and damping parameter \(\kappa = 0.05\) for \(t \in [0, 6\pi/\sigma]\).} \label{fig:4part_d}
  \end{figure}
  \begin{figure}
    \includegraphics[width=0.8\linewidth]{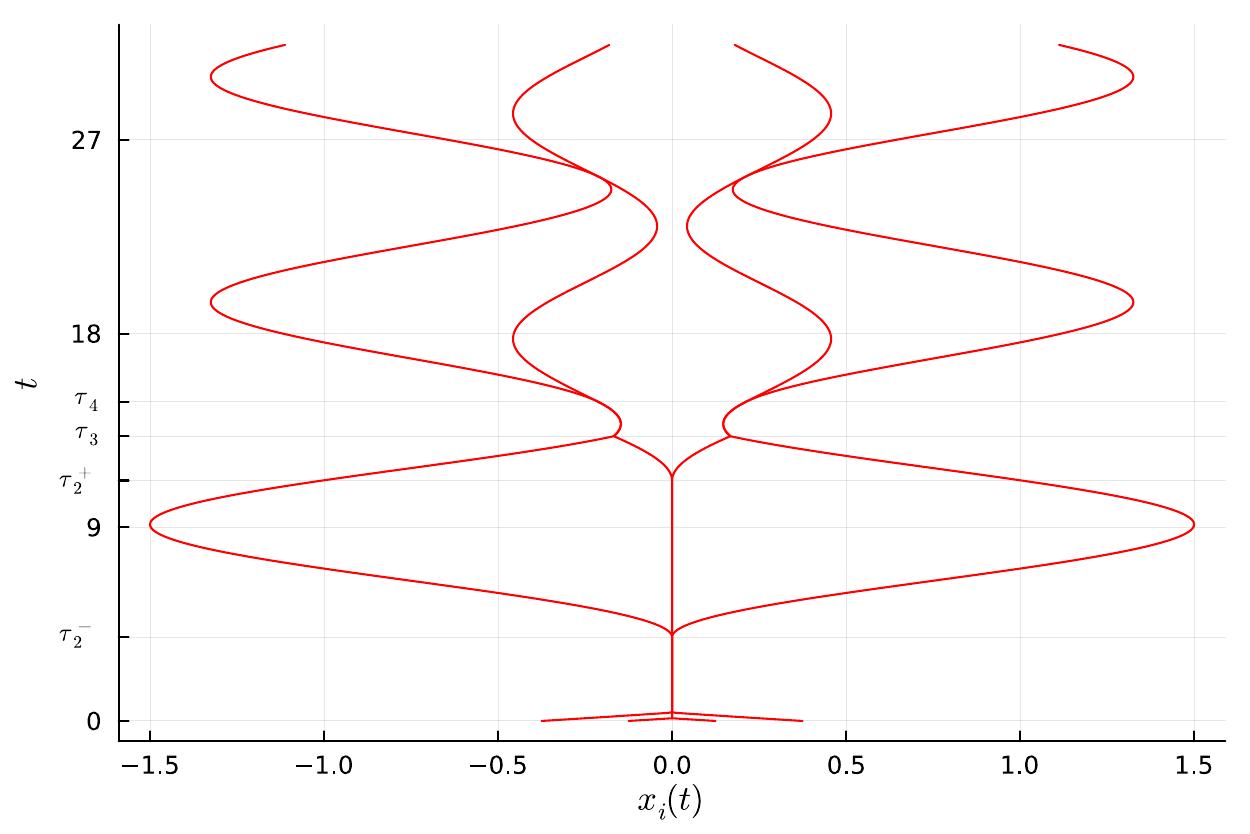}
    \caption{Poisson repulsion and quadratic confinement with \(\lambda = 0.6\) and no damping for \(t \in [0, 6\pi/\sigma]\).} \label{fig:4part_ud}
  \end{figure}
  
  The undamped case should be compared to the numerical example considered in \cite[Section 1.4]{brenier2013sticky} where a time-discrete scheme is applied to the repulsive Euler--Poisson system with neutralizing background and periodic domain; in this case the resulting equations correspond to the repulsive Poisson and quadratic confinement potential above.
  There it is observed that after an initial period of trajectories colliding and dissipating energy, at some point the trajectories only touch tangentially and no more energy is dissipated.
  
  \subsubsection{Asymptotic behavior of the damped system}
  In our previous examples we saw that the presence of damping drove our system to a steady equilibrium state in the form of a uniform probability distribution on a finite interval.
  This asymptotic behavior was also shown to happen at an exponential rate in \cite{carrillo2016pressureless} for solutions which remain classical.
  We will show that this is also the case for our Lagrangian, or entropic, solutions.
  Let us as above consider \(\lambda > \kappa > 0\), this is not strictly necessary and all the other cases can also be dealt with in the same manner. In Lagrangian variables, the system reads
  \begin{equation}\label{eq:EPquad:DI}
    \dot{X}(t) + \del I_\Kscr(X(t)) \ni U(t), \qquad \dot{U}(t) = -2\kappa U(t) + \lambda^2 \left(2m-1 + \int_\Omega X(t,\omega)\dee\omega -X(t) \right).
  \end{equation}
  Observe that \(\int_\Omega X(t,\omega)\dee\omega\) is simply the center of mass \(\bar{x}(t)\) expressed in Lagrangian variables, and in our case \(\beta = 0\) we find from \eqref{eq:CCZ:means}
  \begin{equation}\label{eq:asymptotes}
    \bar{x}_\infty \coloneqq \lim\limits_{t\to\infty}\bar{x}(t)  = \bar{x}(0) + \frac{\bar{v}(0)}{2\kappa}, \qquad \bar{v}_\infty = \lim\limits_{t\to\infty} \bar{v}(t) = 0.
  \end{equation}
  This suggests the steady equilibrium state with \(\dot{X} = V = 0\) and
  \begin{equation}\label{eq:steadyX}
    X = 2m-1 + \int_\Omega X^0(m)\dee m + \frac{1}{2\kappa}\int_\Omega V^0(m)\dee m = 2m -1 + \int_\R x \dee\rho_0(x) + \frac{1}{2\kappa}\int_\R v_0(x)\dee\rho_0(x),
  \end{equation}
  which, with appropriate rescaling of \(\rho\) and parameter choices, agrees with the asymptotics for global-in-time classical solutions in \cite[Theorem 4.1]{carrillo2016pressureless}.
  We will show that the damping term indeed forces any initial data \(X^0 \in \Kscr\), \(V^0 \in \Lp{2}(0,1)\) to this equilibrium state exponentially fast.
  
  \begin{thm}[Asymptotic behavior]
    Let \((\rho,v)\) be the solution of \eqref{eq:EPQ} corresponding to the Lagrangian, or entropy, solution for initial data \((\rho_0, v_0)\) satisfying \eqref{eq:init}. Then the asymptotic solution is given by
    \begin{equation*}
      \rho_\infty(x) \coloneqq \lim\limits_{t\to\infty}\rho(t,x) = \frac12 \chi_{(-1,1)}(x-\bar{x}_\infty), \qquad v_\infty(x) \coloneqq \lim\limits_{t\to\infty}v(t,x) \equiv 0
    \end{equation*}
    with \(\bar{x}_\infty\) defined in \eqref{eq:asymptotes}. Moreover, this convergence is exponentially fast in the \(\Tscr_2\)-metric \eqref{eq:TP2met} at a rate proportional to the damping factor \(\gamma = 2\kappa\).
  \end{thm}
  \begin{proof}
    Let us start by considering the evolution of the energy, which is the sum of a kinetic and a potential part
    \begin{equation*}
      \frac12 \int_\R v(t,x)^2\dee\rho(t,x) + \frac{\lambda^2}{2} \iint_{\R\times\R} \left( \frac12(x-y)^2 - \abs{x-y} \right)\dee\rho(t,y)\dee\rho(t,x),
    \end{equation*}
    and which in Lagrangian variables reads
    \begin{equation*}
      E(t) \coloneqq \frac12 \int_\Omega V(t,m)^2\dee m + \frac{\lambda^2}{2} \iint_{\Omega\times\Omega} \left( \frac12(X(t,m)-X(t,\omega))^2 - \abs{X(t,m)-X(t,\omega)} \right)\dee\omega\dee m.
    \end{equation*}
    However, there is the issue that we do not really have an evolution equation for the velocity \(V\).
    Fortunately, we have the relation \(V = \Proj_{T_X\Kscr} U\), meaning \(\norm{V}_{\Lp{2}(0,1)} \le \norm{U}_{\Lp{2}(0,1)}\), which allows us to bound \(E(t)\) from above by
    \begin{equation*}
      \hat{E}(t) \coloneqq \frac12 \int_\Omega U(t,m)^2\dee m + \frac{\lambda^2}{2} 	\iint_{\Omega\times\Omega} \left( \frac12(X(t,m)-X(t,\omega))^2 - \abs{X(t,m)-X(t,\omega)} \right)\dee\omega\dee m,
    \end{equation*}
    and we have an evolution equation for \(U\).
    Differentiating this quantity we find
    \begin{align*}
      \diff{^+}{t}\hat{E}(t) &= \int_\Omega U(t,m) \left[ -2\kappa U(t,m) + \lambda^2 \left(2m-1 +\int_\Omega X(t,\omega)\dee\omega -X(t,m) \right) \right]\dee m \\
      &\quad+ \frac{\lambda^2}{2} \int_\Omega \int_\Omega \left( X(t,m) - X(t,\omega) - \sgn\left(X(t,m)-X(t,\omega)\right) \right) (V(t,m)-V(t,\omega))\dee\omega\dee m \\
      &= -2\kappa \int_\Omega U(t,m)^2\dee m + \lambda^2\int_\Omega U(t,m) \left(2m-1 +\int_\Omega X(t,\omega)\dee\omega -X(t,m) \right)\dee m \\
      &\quad+ \lambda^2 \int_\Omega V(t,m) \left( X(t,m) - \int_\Omega X(t,\omega)\dee\omega - \int_\Omega \sgn\left(X(t,m)-X(t,\omega)\right)\dee\omega \right)\dee m \\
      &= -2\kappa \int_\Omega U(t,m)^2\dee m - \lambda^2 \int_\Omega \left(U(t,m)-V(t,m)\right) \left( X(t,m) - \int_\Omega X(t,\omega)\dee\omega \right)\dee m  \\
      &\quad+ \lambda^2 \int_\Omega V(t,m) \left(2m-1 - \Proj_{\Hscr_{X(t)}}(2m-1)\right)\dee m + \lambda^2 \int_\Omega \left(U(t,m)-V(t,m)\right) \left(2m-1\right)\dee m,
    \end{align*}
    where we again have used \eqref{eq:signavg}.
    In the final expression, the second and third integral vanish due to Lemma \ref{lem:LagSol:props} and \eqref{eq:intPerp}; the second because of \(U-V \in N_X\Kscr \subset \Hscr_{X}^\perp\), while the third vanishes for a.e.\ \(t\) since \(V(t) \in \Hscr_{X(t)}\) for a.e.\ \(t\).
    We note that \(U-V \in N_X\Kscr\) while \(2m-1\) is nondecreasing, which allow us to combine the characterizations \eqref{eq:TC1} and \eqref{eq:TC2} to deduce that the final term is nonpositive, yielding
    \begin{equation*}
      \diff{^+}{t}\hat{E}(t) \le -2\kappa \int_\Omega U(t,m)^2\dee m.
    \end{equation*}
    From this we deduce \(E(t) \le \hat{E}(t) \le \hat{E}(0) = E(0)\).
    Moreover, this implies that \(V\) decays exponentially: indeed,
    \begin{align*}
      \frac12 \int_\Omega U(t,m)^2\dee m \le \hat{E}(t) \le \hat{E}(0) - 2\kappa \int_0^t \int_\Omega U(s,m)^2\dee m \dee s
    \end{align*}
    which by Gr\"{o}nwall's inequality implies
    \begin{equation*}
      \norm{V(t)}_{\Lp{2}(0,1)}^2 \le \norm{U(t)}_{\Lp{2}(0,1)}^2 \le 2 \e^{-4\kappa t} \hat{E}(0) = 2 \e^{-4\kappa t} E(0).
    \end{equation*}
    This suggests that a damped system, i.e., with \(\kappa > 0\), will approach an equilibrium state at an exponential rate proportional to \(\kappa\).
    To simplify notation we introduce
    \begin{align*}
      \tilde{X}(t,m) &= X(t,m)-2m+1 - \bar{x}(t), \\
      \tilde{V}(t,m) &= V(t,m) - \bar{v}(t), \quad \tilde{U}(t,m) = U(t,m) - \bar{v}(t),
    \end{align*}
    where we note that the center of mass \(\bar{x}(t) = \bar{X}(t)\) and total momentum \(\bar{v}(t) = \bar{V}(t)\) satisfy
    \begin{equation*}
      \diff{^+}{t} \bar{X}(t) = \bar{V}(t), \qquad \bar{v}(t) = \int_0^1 V(t,\omega)\dee\omega = \int_0^1 U(t,\omega)\dee\omega.
    \end{equation*}
    Here \((X-\tilde{X},V-\tilde{V})\) corresponds to a uniform distribution centered at the center of mass for the system, and according to \eqref{eq:asymptotes} this converges uniformly pointwise in \(\Omega\) to the proposed steady state \eqref{eq:steadyX}.
    From the fact that \( -\tilde{X}, \tilde{V} \in T_{X}\Kscr \), \eqref{eq:TC1}, \eqref{eq:intPerp}, \( \tilde{U}-\tilde{V} = U - V \in N_X\Kscr \) and \( \int_0^1 \tilde{X}\dee m = 0 = \int_0^1 \tilde{U}\dee m \) we derive the following relations,
    \begin{align*}
      \int_0^1 \tilde{X} \tilde{V}\dee m &\le \int_0^1 \tilde{X} \tilde{U} \dee m = \int_0^1 \tilde{X} U \dee m, \\
      \int_0^1\tilde{V}\tilde{U}\dee m &= \int_0^1 \tilde{V}^2 \dee m = \int_0^1 V^2 \dee m - \bar{v}^2 \le \int_0^1 U^2\dee m - \bar{v}^2 = \int_0^1 \tilde{U}^2\dee m = \int_0^1 \tilde{U} U \dee m.
    \end{align*}
    The first identity in the second line holds for a.e.\ \(t\) since \(\tilde{V}(t) \in \Hscr_{X(t)}\) for a.e.\ \(t\), alternatively it holds for all times with an inequality since \(\tilde{V} \in T_X\Kscr\).
    From these relations, the evolution equations for \((\tilde{X},\tilde{V})\) and \eqref{eq:EPquad:DI},   
    we then, omitting the subscript \(\Lp{2}(0,1)\) from norms and inner products for brevity, obtain
    \begin{equation*}
      \diff{^+}{t} \left(\frac{c_1}{2} \norm{\tilde{X}}^2 + c_2 \ip{\tilde{X}}{\tilde{U}} + \frac12\norm{\tilde{U}}^2 \right)
      \le -c_2 \lambda^2 \norm{\tilde{X}}^2 + (c_1-2\kappa c_2 -\lambda^2) \ip{\tilde{X}}{\tilde{U}} + (c_2-2\kappa) \norm{\tilde{U}}^2
    \end{equation*}
    for \(c_1, c_2 \ge 0\).
    Choosing \(c_1 = \lambda^2\) and \(c_2 = \kappa\), we can apply Gr\"{o}nwall's inequality to find
    \begin{equation*}
      \lambda^2\norm{\tilde{X}(t)}^2 + 2\kappa \ip{\tilde{X}(t)}{\tilde{U}(t)} + \norm{\tilde{U}(t)}^2 \le \left(\lambda^2\norm{\tilde{X}(0)}^2 + 2\kappa \ip{\tilde{X}(0)}{\tilde{U}(0)} + \norm{\tilde{U}(0)}^2\right) e^{-2\kappa t},
    \end{equation*}
    or, recalling \(\lambda^2 = \kappa^2 + \sigma^2\),
    \begin{equation*}
      \sigma^2 \norm{\tilde{X}(t)}^2 + \norm{\tilde{U}(t) + \kappa \tilde{X}(t)}^2 \le \left( \sigma^2 \norm{\tilde{X}(0)}^2 + \norm{\tilde{U}(0) + \kappa \tilde{X}(0)}^2\right) \e^{-2\kappa t}.
    \end{equation*}
    That is, both \(\tilde{X}\) and \(\tilde{V}\) tend exponentially fast to zero in the \(\Lp{2}\)-norm with rates proportional to \(\gamma = 2\kappa\).
    Then, as a Lagrangian solution satisfies \(V(t,m) = v(t,X(t,m))\) for a.e.\ \(t\), we use \eqref{eq:TP2&L2} to conclude.
  \end{proof}
  
  \subsubsection{Lyapunov stability of undamped system}
  In the undamped cases of our examples, we have seen that the solutions did not tend to a steady state.
  One could then instead ask, if two solutions are close in some sense, will they remain close?
  
  To this end, take two solutions \((X_1, V_1)\) and \((X_2,V_2)\) with corresponding prescribed velocities \(U_1\) and \(U_2\).
  We further introduce the averaged quantities \(\bar{x}_i\), \(\bar{v}_i\), as well as \(\hat{X}_i \coloneqq X_i - \bar{x}_i\), \(\hat{V}_i \coloneqq V_i - \bar{v}_i\), \(\hat{U}_i \coloneqq U_i - \bar{v}_i\) for \(i\in\{1,2\}\).
  From the equations we then get 
  \begin{equation*}
    \diff{^+}{t} \left( \hat{X}_1-\hat{X}_2 \right) = \hat{V}_1 - \hat{V}_2, \qquad \diff{}{t} \left(\hat{U}_1-\hat{U}_2\right) = - \lambda^2 \left(\hat{X}_1 - \hat{X}_2\right),
  \end{equation*}
  leading to
  \begin{align*}
    \diff{^+}{t} \left( \lambda^2 \norm{\hat{X}_1-\hat{X}_2}^2 + \norm{\hat{U}_1-\hat{U}_2}^2 \right) &= 2\lambda^2 \ip{\hat{X}_1-\hat{X}_2}{\hat{V}_1-\hat{V}_2} - 2\lambda^2 \ip{\hat{U}_1-\hat{U}_2}{\hat{X}_1-\hat{X}_2} \\
    &= -2\lambda^2 \ip{\hat{X}_1-\hat{X}_2}{U_1-V_1 - (U_2-V_2)} \\
    &= 2\lambda^2 \left( \ip{\hat{X}_1}{U_2-V_2} + \ip{\hat{X}_2}{U_1-V_1} -\sum_{i=1}^2 \ip{\hat{X}_i}{U_i-V_i} \right).
  \end{align*}
  Now, since \(\hat{X}_i \in \Hscr_{X_i}\) and \(U_i-V_i \in N_{X_i}\Kscr \subset \Hscr_{X_i}^\perp\), we have \( \ip{\hat{X}_i}{U_i-V_i} = 0 \).
  For the first cross-term, we note that \(\hat{X}_1(t,\cdot)\) is nondecreasing, so it belongs to the tangent cone \(T_{X_2}\Kscr\) by Lemma \ref{eq:TC2}, and then it follows from \eqref{eq:TC1} that
  \(\ip{\hat{X}_1}{U_2-V_2} \le 0\).
  The same argument applies to the second cross-term, and we have shown
  \begin{equation*}
    \diff{^+}{t} \left( \lambda^2 \norm{\hat{X}_1-\hat{X}_2}^2 + \norm{\hat{U}_1-\hat{U}_2}^2 \right) \le 0,
  \end{equation*}
  that is,
  \begin{equation*}
    \lambda^2 \norm{\hat{X}_1(t)-\hat{X}_2(t)}^2 + \norm{\hat{U}_1(t)-\hat{U}_2(t)}^2 \le \lambda^2 \norm{\hat{X}_1(0)-\hat{X}_2(0)}^2 + \norm{\hat{V}_1(0)-\hat{V}_2(0)}^2.
  \end{equation*}
  In particular, comparing \((X_1(t), V_1(t)) = (X(t), V(t))\) with the `equilibrium state' \( (\tilde{X}_2(t), \tilde{V}_2(t)) = (2m-1+\bar{x}(t),\bar{v}(t)) \), writing \(\tilde{X} = \hat{X}-(2m-1)\), we have the stability result
  \begin{equation*}
    \lambda^2 \norm{\tilde{X}(t)}^2 + \norm{\hat{V}(t)}^2 \le \lambda^2 \norm{\tilde{X}(t)}^2 + \norm{\hat{U}(t)}^2 \le \lambda^2 \norm{\tilde{X}(0)}^2 + \norm{\hat{V}(0)}^2.
  \end{equation*}
  In particular, for this undamped case we have \(\bar{v}(t) \equiv \bar{v}(0)\) such that \(\bar{x}(t) = \bar{x}(0) + t \bar{v}(0)\), and so the above result can be restated as
  \begin{equation*}
    \lambda^2 \norm{\tilde{X}(t)}^2 + \norm{V(t)}^2 \le \lambda^2 \norm{\tilde{X}(t)}^2 + \norm{U(t)}^2 \le \lambda^2 \norm{\tilde{X}(0)}^2 + \norm{V(0)}^2.
  \end{equation*}
  Let us finally mention that this stability estimate resembles the result in \cite[Theorem 3.6]{brenier2013sticky}, and we point it out here as a consequence of the previous results.
  
  \subsubsection*{Acknowledgments}
  JAC was supported by the Advanced Grant Nonlocal-CPD (Nonlocal PDEs for Complex Particle Dynamics: Phase Transitions, Patterns and Synchronization) of the European Research Council Executive Agency (ERC) under the European Union's Horizon 2020 research and innovation programme (grant agreement No. 883363) and the EPSRC grant number EP/V051121/1.
  STG was supported by the Research Council of Norway through grant no.~286822, \textit{Wave Phenomena and Stability --- a Shocking Combination (WaPheS)}, and the Swedish Research Council under grant no.~2021-06594 while in residence at Institut Mittag-Leffler in Djursholm, Sweden during the fall semester of 2023.
  
  

\end{document}